\documentclass[10pt]{amsart}

\usepackage{mathtools}
\mathtoolsset{showonlyrefs=true}
\usepackage[hmargin=0.8in,height=8.6in]{geometry}
\usepackage{amssymb,amsthm, times}
\usepackage{delarray,verbatim}
\usepackage{ifpdf}
\ifpdf
\usepackage[pdftex]{graphicx}
 \else
\usepackage[dvips]{graphicx}
 \fi

\usepackage{bm}

\usepackage{amsmath}

\usepackage{ifpdf}
\usepackage{color}
\definecolor{webgreen}{rgb}{0,.5,0}
\definecolor{webbrown}{rgb}{.8,0,0}
\definecolor{emphcolor}{rgb}{0.5,0.95,0.95}

\usepackage{hyperref}
\hypersetup{%
	colorlinks=true,
	linkcolor=webbrown,
	filecolor=webbrown,
	citecolor=webgreen,
	breaklinks=true}
\ifpdf \hypersetup{pdftex,
	pdfstartview=FitH, 
	bookmarksopen=true,
	bookmarksnumbered=true
    } \else \hypersetup{dvips} \fi
\allowdisplaybreaks

\newcommand {\ud}{{\rm d}}

\newcommand {\B}{\mathcal{B}}

\numberwithin{equation}{section}

\newtheorem{theorem}{Theorem}[section]
\newtheorem{proposition}{Proposition}[section]

\newtheorem{remark}{Remark}[section]
\newtheorem{lemma}{Lemma}[section]

\numberwithin{remark}{section} \numberwithin{proposition}{section}
\numberwithin{corollary}{section}
\newcommand {\R}{\mathbb{R}}

\newcommand {\F}{\mathcal{F}}

\newcommand {\N}{\mathbb{N}}
\newcommand {\p}{\mathbb{P}}

\newcommand {\E}{\mathbb{E}}
\newcommand {\bE}{\mathbb{E}}

\newcommand {\bR}{\mathbb{R}}

\newcommand{\lev}{L\'{e}vy }

\newcommand{\e}{\mathbb{E}}

\newcommand{\Ei}{\rm{e}_1}

\begin{document}
	\title{Fluctuation theory for spectrally negative Lévy processes killed by additive functionals}
	\author[K. Noba]{Kei Noba$^\dagger$}
	\thanks{$\dagger$\, Department of Mathematics, Graduate School of Science, The University of Osaka Toyonaka, Osaka 560-0043, Japan. Email: noba.kei.sci@osaka-u.ac.jp}
	
	\author[J. L. P\'erez]{Jos\'e-Luis P\'erez$^{*}$}
	
	\thanks{$*$\, Department of Probability and Statistics, Centro de Investigaci\'on en Matem\'aticas A. C. Calle Jalisco
		s/n. C.P. 36240, Guanajuato, M\'exico. Email: jluis.garmendia@cimat.mx}
	\begin{abstract}
In this paper, we study fluctuation identities for spectrally negative Lévy processes killed by a general class of additive functionals. We consider positive co-natural additive functionals (PcNAFs), which include as special cases both absolutely continuous functionals and finite mixtures of local times. Our main result shows that the associated fluctuation identities—such as two-sided exit problems and resolvent measures—retain the same structure as in the classical case and can be expressed in terms of generalized scale functions.
These scale functions are characterized as the unique solutions to Volterra-type integral equations driven by Radon measures, thereby extending the results of \cite{LP,LZ}. Our approach is based on representing the additive functional as a mixture of local times with respect to its Revuz measure, combined with classical fluctuation identities and an approximation scheme for general Radon measures using Poisson random measures.

\noindent \small{\noindent  MSC 2020 Subject Classifications: 60G51, 60J45, 60J55, 60J35. }
		\end{abstract}
	\maketitle
	\section{Introduction}
Fluctuation theory for L\'evy processes provides a powerful framework for the study of key path properties such as first passage times, overshoots, undershoots, and potential measures. Classical results in this area are based on the Wiener--Hopf factorization, It\^o excursion theory, and martingale methods, and have found numerous applications in areas such as queueing theory, mathematical finance, insurance risk, and biology, among many others.

In the case of spectrally negative L\'evy processes, fundamental fluctuation identities—including exit problems and potential measures—can be expressed in terms of scale functions. A comprehensive survey of the fluctuation theory of spectrally negative L\'evy processes, their associated scale functions, and applications is provided by Kuznetsov et al.~\cite{KKR}. We also refer the reader to the monograph by Kyprianou~\cite{K} for comprehensive accounts.

A natural and widely studied problem is to understand how these identities are modified when the process is observed up to a random time determined by a killing mechanism. In the classical setting, killing occurs either at an independent exponential time or upon exiting a given domain. More recently, attention has shifted to killing mechanisms that depend on the trajectory of the process itself. A notable example is the class of $\omega$-killed Lévy processes introduced by Li and Palmowski~\cite{LP}. 

In this setting, one considers a spectrally negative Lévy process $X=\{X_t:t\geq0\}$ that is killed at a state-dependent rate $\omega(X_t)$, where $\omega$ is a locally bounded real-valued function. Equivalently, the killing mechanism can be described in terms of an additive functional that is absolutely continuous with respect to the Lebesgue measure and admits $\omega(X_t)$ as its density. In this case, the associated fluctuation identities can still be expressed in terms of scale functions, which are characterized as solutions to certain integral equations. 
These results were subsequently extended by Czarna et al.~\cite{Czarna} to the case where the underlying process is a Markov additive process.

Another important extension is due to Li and Zhou~\cite{LZ}, who studied the case in which the additive functional is given by a finite linear combination of local times, under the assumption that the underlying spectrally negative L\'evy process is of unbounded variation.  Their results show that fluctuation identities can be preserved in this setting as well, at the expense of replacing classical scale functions by a recursively defined family of generalized scale functions.

In this paper, we consider a class of additive functionals known as positive co-natural additive functionals (PcNAFs) (see Section \ref{PcNAFforLevy} for details), which includes both the absolutely continuous case studied in~\cite{LP} and the finite mixtures of local times considered in~\cite{LZ}. Our aim is to generalize the fluctuation identities obtained in~\cite{LP,LZ} to the case of a general PcNAF $A=\{A_t:t\geq0\}$. This includes, in particular, the joint Laplace transforms of
\[
(\tau_b^+,A_{\tau_b^+}), \qquad (\tau_c^-,A_{\tau_c^-}),
\]
as well as the associated resolvent measures, where $\tau_a^+=\inf\{t\geq0:X_t>a\}$ and $\tau_a^-=\inf\{t\geq0:X_t<a\}$ denote the first passage times above and below the level $a$, respectively, with the convention that $\inf\emptyset=\infty$.  Our main result shows that, in this general setting, the corresponding fluctuation identities retain the same structure as in the classical case and can be expressed in terms of generalized scale functions characterized as solutions to Volterra-type integral equations driven by Radon measures. 

The starting point of our approach is the observation that any PcNAF associated with a spectrally negative Lévy process admits a representation as a mixture of local times at different levels with respect to its Revuz measure. We show that this measure is always Radon and that, conversely, any Radon measure gives rise to a PcNAF through such a representation. This yields a canonical parametrization of additive functionals in terms of measures and allows us to formulate the problem at the level of integral equations driven by Radon measures.

The analysis proceeds in two steps. We first consider the case where the Revuz measure is finite and atomic, corresponding to additive functionals given by finite mixtures of local times. In this setting, we extend the approach of~\cite{LZ} to include both bounded and unbounded variation cases, and we characterize the associated scale functions as solutions to integral equations. In a second step, we treat the general case by approximating the Revuz measure with a sequence of random atomic measures constructed from Poisson random measures independent of the underlying Lévy process. Since each element of the approximating sequence is almost surely a finite atomic measure, we can apply the previous results conditionally on the underlying Poisson random measure.

Finally, by establishing stability and convergence properties for the associated integral equations, we identify the limiting scale functions and derive the desired fluctuation identities.

Our results also connect with recent work on extending the notion of scale functions beyond the classical setting of spectrally negative L\'evy processes. Indeed, since $\{e^{-A_t}:t\ge0\}$ is a decreasing right multiplicative functional, it induces a new right process {(see, e.g., p.~38 of \cite{Sharpe}) with semigroup $\{Q_t:t\ge0\}$ given by
\[
Q_t f(x) := \E_x\!\left[e^{-A_t} f(X_t)\right],\qquad t\ge 0,
\]
via Theorem~(61.5) of \cite{Sharpe}. This process may be interpreted as the original process $X$ subjected to killing (or discounting) through the additive functional $A$.} The scale functions introduced here can then be viewed as the corresponding scale functions associated with this transformed process.

Another application of our results concerns the work of \cite{LP}. Their approach allows one to obtain scale functions for processes such as self-similar Markov processes, which are constructed from spectrally negative L\'evy processes via Lamperti-type time changes. With a minor modification of our results, it is also possible to derive scale functions for a spectrally negative L\'evy process time-changed by a general continuous additive functional, following the approach of \cite[(2.11), Section~V]{BG}.

The paper is organized as follows. In Section~\ref{Pre} we recall basic properties of spectrally negative Lévy processes and their scale functions. Section~\ref{PcNAFforLevy} introduces positive co-natural additive functionals and their representation in terms of local times and Revuz measures. The main fluctuation identities are stated in Section~\ref{sec:main_results}. Section~\ref{sec:prelimiary_results} contains the analysis of the associated integral equations and the case of finite mixtures of local times. The proofs of the main results are completed in Section~\ref{sec:proof_main_results} via an approximation argument based on Poisson random measures.

	\section{Preliminaries}\label{Pre}
	\subsection{Spectrally negative \lev processes}
	Let $X=\{X_t:  t\geq 0\}$ be a L\'evy process defined on a  probability space $(\Omega, \mathcal{F}, \p)$ as in Definition 1.1 of \cite{K}.  For $x\in \R$, we denote by $\p_x$ the law of $X$ when it starts at $x$ and write for convenience  $\p$ in place of $\p_0$. Accordingly, we shall write $\e_x$ and $\e$ for the associated expectation operators. In this paper we shall assume throughout that $X$ is \textit{spectrally negative},   meaning here that it has no positive jumps and that it is not the negative of a subordinator. 
	{Let ${\psi}:[0,\infty) \to \R$ be the Laplace exponent of the process $X$, i.e.}
	\begin{equation}\label{lk_formula}
		\e\big[{\rm e}^{\theta X_t}\big]=:{\rm e}^{\psi(\theta)t}, \qquad t, \theta\ge 0,
	\end{equation}
	given by the \emph{L\'evy-Khintchine formula}
	\begin{equation}\label{lk}
		\psi(\theta):=\gamma\theta+\frac{\sigma^2\theta^2}{2}+\int_{(-\infty , 0)}\big({\rm e}^{\theta x }-1-\theta x\mathbf{1}_{\{x>-1\}}\big)\Pi(\ud x), \quad \theta \geq 0,
	\end{equation}
	where $\gamma\in \R$, $\sigma\in \R$, and $\Pi$ is a measure on $(-\infty , 0)$ called the L\'evy measure of $X$ that satisfies
	\[
	\int_{(-\infty,0)}(1\land x^2){\Pi}(\ud x)<\infty.
	\]
		In particular, $X$ has bounded variation paths if and only if $\sigma=0$ and $\int_{(-\infty , 0)}(1\land |x|)\Pi(\ud x)<\infty$.
		Then, we can write 
		\begin{equation}
		\psi(\theta)={\mathbf{d}}\theta+\int_{(-\infty , 0)}\big({\rm e}^{\theta x }-1\big)\Pi(\ud x), \quad \theta \geq 0,
	\end{equation}
	where 
	\begin{align}
	{\mathbf{d}}:=\gamma-\int_{(-1 , 0)}x\Pi(\ud x).
	\end{align}
	\subsection{Scale functions for spectrally negative L\'evy processes}\label{Sec_scale_function}
	Fix $q \geq 0$. We use $W^{(q)}$ for the scale function of the spectrally negative \lev process $X$.  This is the mapping from $\R$ to $[0, \infty)$ that takes value zero on the negative half-line, while on the positive half-line it is a continuous and strictly increasing function that is defined by its Laplace transform:
	\begin{align} \label{scale_function_laplace}
		\begin{split}
			\int_0^\infty  \mathrm{e}^{-\theta x} W^{(q)}(x) dx&= \frac 1 {\psi(\theta)-q}, \quad \theta > \Phi(q),
		\end{split}
	\end{align}
	where $\psi$ is as defined in \eqref{lk} and
	\begin{align}
		\begin{split}
			\Phi(q) := \sup \{ \lambda \geq 0: \psi(\lambda) = q\} . 
		\end{split}
		\label{def_varphi}
	\end{align}
	We also define, for $x \in \R$,
	\begin{align}\label{fun_z_sf}
		Z^{(q)}(x) := 1 + q \overline{W}^{(q)}(x),
	\end{align}
	where $\overline{W}^{(q)}(x):=\int_0^x W^{(q)}(y) dy$. 
	Because $W^{(q)}(x) = 0$ for $-\infty < x < 0$, we have
	$Z^{(q)}(x) = 1$.
	
	By identity (6) in \cite{LRZ}, for $q,\lambda\geq0$, the following identities hold:
		\begin{align}\label{scale_fun_prop}
			W^{(q+\lambda)}(x-d)-W^{(q)}(x-d)&=\lambda\int_d^xW^{(q+\lambda)}(x-z)W^{(q)}(z-d)dz\notag\\
			&=\lambda\int_d^xW^{(q)}(x-z)W^{(q+\lambda)}(z-d)dz,\qquad x\geq d.
		\end{align}
				Regarding the asymptotic behavior near zero, as in Lemma 3.1 of \cite{KKR},
				\begin{align}\label{eq:Wqp0}
						W^{(q)} (0) &= \left\{ \begin{array}{ll} 0 & \textrm{if $X$ is of unbounded
								variation,} \\ 1/\mathbf{d} & \textrm{if $X$ is of bounded variation.}
						\end{array} \right. 
				\end{align}
				On the other hand, as in Lemma 3.3 of \cite{KKR},
				\begin{align}
					\begin{split}
						{\rm e}^{-\Phi(q) x}W^{(q)} (x) \nearrow \psi'(\Phi(q))^{-1}, \quad \textrm{as } x \uparrow \infty.
					\end{split}
					\label{W_q_limit}
				\end{align}
		For the process $X$, we define the first down- and up-crossing
		times, respectively, by 
		\[
		\tau_c^- := \inf \left\{ t \geq 0: {X_t} < c \right\}\quad \text{and}\quad \tau^+_b:= \inf \left\{ t \geq 0: {X_t}>  b \right\},\quad \text{for any $c\leq x\leq b$}. 
	\]
		Then by Theorem 1.2 and 2.6 in \cite{KKR}, we have
			
				\begin{align}
					\begin{split}
						\E_x \left[ {\rm e}^{-q \tau_b^+} 1_{\left\{ \tau_b^+ < \tau_c^- \right\}}\right] &= \frac {W^{(q)}(x-c)}  {W^{(q)}(b-c)}, \\
						\E_x \left[ {\rm e}^{-q \tau_c^-} 1_{\left\{ \tau_b^+ > \tau_c^- \right\}}\right] &= Z^{(q)}(x-c) -  Z^{(q)}(b-c) \frac {W^{(q)}(x-c)}  {W^{(q)}(b-c)}.
					\end{split}
					\label{laplace_in_terms_of_z}
				\end{align}
			
			Additionally, as in Theorem 2.7 in \cite{KKR}, for $c\leq x\leq b$ and a 
			measurable and bounded function 
			$f:\R\to\R$
			\begin{align}\label{resolvent}
				\E_x\left[\int_0^{\tau_b^+\wedge\tau_c^-}e^{-qt}f(X_t)dt\right]=\int_{c}^bf(y)\left\{\frac{W^{(q)}(x-c)}{W^{(q)}(b-c)}W^{(q)}(b-y)-W^{(q)}(x-y)\right\}dy.
			\end{align}
			
			\section{Additive functionals of L\'evy processes} \label{PcNAFforLevy}
A positive co-natural additive functional (PcNAF) {$A=\{A_t:t\geq0\}$} of the L\'evy process $X$ is a $\R_+$-valued process defined on $\Omega_A \in \F$ which satisfies the following conditions. 
\begin{enumerate}
\item $\p_x(\Omega_A)=1$ for all $x \in \R$. 
\item ${\theta_t}(\Omega_A) \subset \Omega_A$ for all $t>0$, {where $\theta_t:\Omega\to\Omega$ denotes the shift operator.} 
\item For {every} $\omega\in\Omega_A$, the map $t\mapsto A_t (\omega)$ is non-decreasing, right-continuous, takes finite values, and satisfies $A_0(\omega) =0$.  
\item There exists a non-negative measurable function $f_A$ such that {for all $\omega\in\Omega_A$ and $t>0$ we have}
\begin{align}
A_t (\omega)-A_{t-}(\omega)=f_A(X_t(\omega)).
\end{align}
\item {For all $\omega\in\Omega_A$ and $s,t>0$, we have}
			\begin{align}
			A_{s+t}(\omega)=A_t(\omega)+A_s(\theta_t(\omega)) <\infty.
			\end{align}
\end{enumerate}
{For clarity, we have written explicit dependence on $\omega$; however, this is typically omitted. This definition follows {from} \cite[p.229]{BB}.}
			
			In particular, when $X$ has paths of unbounded variation, such additive functionals have continuous paths. 
				In fact, suppose the PcNAF $A$ associated with a spectrally negative L\'evy process of unbounded variation has jumps. Then, there exists $x\in\R$ such that the function $f_A$ appearing in condition (4) satisfies $f_A(x)>0$. In this case, $x$ is not polar but regular for itself, so once $X$ hits $x$, it subsequently hits $x$ again infinitely many times in an accumulating manner. As a result, $A$ diverges to infinity, and the condition (3) is no longer satisfied. A PcNAF with continuous paths is called a positive continuous additive functional (PCAF).
				
			 Known examples of continuous additive functionals can be obtained by, for instance, considering a locally bounded measurable function ${\omega}:\R\to[0,\infty)$, and defining
			\begin{equation}\label{Af}
			A^{{\omega}}_t:=\int_0^t{\omega}(X_s)ds, \qquad t\geq0.
			\end{equation}
			Then, $A^{{\omega}}$ is a PCAF. 
			
			When $X$ has paths of unbounded variation, another example of a PCAF is given by the local time at level $a\in\R$, denoted by $L^a:=\{L_t^a:t\geq0\}$, which is defined by
			\[
			L^a_t:=\limsup_{\varepsilon\to0+}\frac{1}{2\varepsilon}\int_0^t1_{\{|X_s-a|\leq \varepsilon\}}\,ds, \qquad a\in\R,\ t\geq0.
			\]
			By Proposition~2 in Chapter~V of \cite{B}, the convergence holds in $L^2(\mathbb{P})$, uniformly on compact time intervals. The local time plays a fundamental role in this paper. 

			When $X$ has paths of bounded variation, we define the local time process  $L^a=\{L_t^a:t\geq0\}$ by
			\begin{align}\label{lt_bv}
			L^a_t=\frac{1}{{\mathbf{d}}}\sum_{s\in(0, t]}1_{\{a\}}(X_s),\quad t\geq 0, 
			\end{align}
			which is not a PCAF but a PcNAF. 
			A key property that is shared by $\{L^a\}_{a\in\R}$ in both the bounded and unbounded variation cases is that it satisfies the occupation density formula (see, (2) in Chapter V of \cite{B} and (1.5) in \cite{FP}).

			There is a close connection between local times and 
				PcNAFs of a L\'evy process, namely that every PcNAF can be represented as a mixture of local times at different levels.
			To describe this relationship, we first introduce the Revuz measure $\nu_A$ of a PcNAF $A$, defined with respect to the Lebesgue measure by
\begin{align}
\nu_A(B)=\lim_{t\downarrow0}\frac{1}{t}
\int_\R \E_x\left(\int_{(0, t]}1_B(X_s)d A_s\right) dx
,\quad B\in\mathcal{B}(\R).
\label{Revuz_measure}
\end{align}
This definition follows \cite[(8.9)]{Getoor} with $q=0$.
			
{We collect several properties of Revuz measures associated with PcNAFs of spectrally negative L\'evy processes; the proofs are deferred to Appendix~\ref{prop_PCNAF}. In particular, we first note that the Revuz measure of any PcNAF is always a Radon measure.}
\begin{proposition}\label{PropD2}
	The Revuz measure $\nu_A$ associated with a PcNAF $A$ is a Radon measure.
\end{proposition}

Moreover, the converse statement also holds.
\begin{proposition}\label{conv_radon}
	For any Radon measure $\nu$, there exists a PcNAF $A$ such that its Revuz measure is $\nu$.
\end{proposition}
			{We now turn our attention to the representation of a PcNAF as a mixture of local times at different points in terms of its Revuz measure.
				The local times defined in this section play a significant role in this representation, as shown in the following result.
\begin{proposition}\label{Prfof(2.10)}
	Let $A$ be a PcNAF and $\nu_A$ its Revuz measure. Then, for any $x \in \R$,
	\begin{align}
		A_t=\int_{\R} L^a_t \, \nu_A(da), \qquad t \geq 0, \label{representation_equation}
	\end{align}
	$\p_x$-a.s.
\end{proposition}
	This representation is proved in Theorem V.5(iii) of \cite{B} for integrable PCAFs of Lévy processes with unbounded variation paths. In Proposition~\ref{Prfof(2.10)}, we extend \eqref{representation_equation} to arbitrary PcNAFs of spectrally negative Lévy processes.
	
	Conversely, for every Radon measure $\nu$, the process $\left\{\int_{\R} L_t^a \, \nu(da) : t \geq 0 \right\}$ defines a PcNAF. This follows from Proposition~\ref{conv_radon} together with \eqref{representation_equation}.
	
	This representation will play a central role in the sequel: we first analyze the case of atomic Revuz measures and then extend the results to general Radon measures via approximation.} 

			\section{Main results}\label{sec:main_results}
			In this section, we present the main fluctuation identities for spectrally negative Lévy processes killed by a positive co-natural additive functional (PcNAF). Our results provide a unified and substantial extension of the existing theory by covering, in a single framework, both state-dependent killing mechanisms (as in \cite{LP}) and finite mixtures of local times (as in \cite{LZ}). In contrast to the existing literature, we treat general PcNAFs, thereby capturing a broad class of killing mechanisms that go beyond both state-dependent rates and finite mixtures of local times.
				
				Our main contribution shows that, even in this level of generality, the classical fluctuation identities—such as two-sided exit problems and resolvent measures—retain their canonical form. This is achieved by introducing a new class of generalized scale functions, characterized as the unique solutions to Volterra-type integral equations driven by Radon measures. This representation not only extends the classical theory but also provides a robust analytical framework that allows for approximation, stability, and convergence arguments, which are essential in handling general additive functionals. 
			
			We begin by presenting the identities for the two-sided exit problem. We then derive the corresponding expressions for one-sided exit problems, both upwards and downwards, and conclude with formulas for the associated resolvent measures.
			
			Let $\mathcal{M}$ denote the space of Radon measures on $\mathbb{R}$. For any $\nu \in \mathcal{M}$, we denote its atomic and diffuse parts by $\nu^a$ and $\nu^d$, respectively. We define a mapping ${\mathbf{T}}: \mathcal{M} \to \mathcal{M}$ as follows. If $\nu \in \mathcal{M}$ such that $\nu^a(dz) = \sum_{i=1}^\infty p_i \delta_{a_i}(dz)$, where $a_i \in \mathbb{R}$ and $p_i \geq 0$, then 
			\begin{align}\label{oper_T}
			{\mathbf{T}}(\nu)(dz)=\sum_{i=1}^{\infty}{\mathbf{d}}(1-e^{-p_i/{\mathbf{d}}})\delta_{a_i}(dz)+\nu^d(dz).
			\end{align}
			{Recall that when $X$ has paths of unbounded variation, the previous identities are understood in the limiting sense as ${\mathbf{d}} \to \infty$, and thus we regard ${\mathbf{d}}(e^{p/{\mathbf{d}}}-1)=p$.}
			
			The fluctuation identities will be expressed in terms of a new class of scale functions, which will be characterized as the unique solutions to integral equations. Thus, we first establish the following result, which provides existence and uniqueness for a class of integral equations. The proof is deferred to Appendix \ref{proof_exist_uni}.
			\begin{proposition}\label{exist_uni}
				For any $T>y$
				, a locally bounded function $h:\R\to\R$, and two Radon measures $\nu$ and $\lambda$ such that 
				\begin{align}\label{assum_atoms}
					\inf_{x\in[y,T]}(1-W^{(q)}(0)(\nu-\lambda)\{x\})>0,
				\end{align}
				the equation
				\begin{align}\label{volterra_prop}
					H(x,y)= h(x-y)+\int_{[y,x]}W^{(q)}(x-z)H(z,y)(\nu-\lambda)(dz),\qquad x\in\R.
				\end{align}
				admits a unique locally bounded solution {$H(x, y)=H_{\nu-\lambda}^{(q)}(x, y)$} on $[y,T]$. Furthermore, if $h(x)\geq0$ for all $x\in[y,T]$ and $\nu-\lambda$ is a non-negative measure, then $H(x,y)\geq0$ for all $x\in[y,T]$.
			\end{proposition}
			\begin{remark}\label{exi_uni_T}
			Note that the inequality $(1-e^{-x})\leq x$ for $x\in\R$, implies that
			\begin{align}\label{ine_T}
				{\mathbf{T}}(\nu)(dz)\leq \nu(dz).
			\end{align}
			Hence, if $\nu(dz)$ is a Radon measure, so is $T(\nu)(dz)$. Moreover, for any interval $[y,T]\subset\R$, we have
				\begin{align}\label{assum_prop_ex_uni}
				\sup_{x\in[y,T]}(1-W^{(q)}(0){\mathbf{T}}(\nu)\{x\})^{-1}&\leq\sup_{i\in\mathbb{N}}\left(1-W^{(q)}(0){\mathbf{d}}(1-e^{-p_i/{\mathbf{d}}})1_{\{a_i\in[y,T]\}}\right)^{-1}\notag\\&=\sup_{i\in\mathbb{N}}e^{p_i/{\mathbf{d}}}1_{a_i\in[y,T]}
                <\infty,
			\end{align}
			where the last inequality follows from the fact that $\nu$ is a Radon measure, and thus $\nu[y,T]<\infty$. 
			 Hence, by Proposition \ref{exist_uni}, there exists a unique solution to the integral equation:
			\begin{align*}
		W_{{\mathbf{T}}(\nu)}^{(q)}(x,y)&=W^{(q)}(x-y)+\int_{[y,x]}W^{(q)}(x-z)W_{{\mathbf{T}}(\nu)}^{(q)}(z,y){\mathbf{T}}(\nu)(dz),\qquad x\geq y.
			\end{align*}
			\end{remark}
			We recall that $\nu_A$ denotes the Revuz measure of the PcNAF $A$, which allows us to express $A$ as a mixture of the local times of the L\'evy process $X$ at different points, weighted by the Revuz measure $\nu_A$, as in \eqref{representation_equation}. In the following result, we provide the solution to the two-sided exit problem for the L\'evy process $X$ killed by the PcNAF $A$.
		\begin{theorem}\label{main_sf}
					Let $q \geq 0$ and $x \in (c,b)$. Then 
			\begin{align}\label{two_sided_up}
				\E_x\left[e^{-q\tau_b^+-A_{\tau_b^+}};\tau_b^+<\tau_c^-\right]=\frac{W_{{\mathbf{T}}(\nu_A)}^{(q)}(x,c)}{W_{{\mathbf{T}}(\nu_A)}^{(q)}(b,c)},
			\end{align}
			where $W_{\mathbf{T}(\nu_A)}$ is the unique solution to
			\begin{align}\label{volterra_sf_w}
				W_{{\mathbf{T}}(\nu_A)}^{(q)}(x,y)&=W^{(q)}(x-y)+\int_{[y,x]}W^{(q)}(x-z)W_{{\mathbf{T}}(\nu_A)}^{(q)}(z,y){\mathbf{T}}(\nu_A)(dz),\qquad x\geq y.
			\end{align}
		\end{theorem}
		For the next result we define for fixed $c\in\R$, $q\geq0$ and a Radon measure $\nu$,  
		\begin{align}\label{sf_Z}
			Z^{(q)}_{\nu}(x,c):=1+\int_{[c,x]}W^{(q)}_{\nu}(x,y)\nu(dy)+q\int_c^xW^{(q)}_{\nu}(x,y)dy,\qquad x,c\in\R.
		\end{align}
		We now provide a solution to the two-sided exit problem downwards for the L\'evy process $X$ killed by the PcNAF $A$.
		\begin{theorem}\label{two_sided_down_thm}
			Let $q \geq 0$ and $x \in (c,b)$. Then
			\begin{align}\label{two_sided_down_iden}
				\E_x\left[e^{-q\tau_c^--A_{\tau_c^-}};\tau_c^-<\tau_b^+\right]=Z^{(q)}_{{\mathbf{T}}(\nu_A)}(x,c)-\frac{Z^{(q)}_{{\mathbf{T}}(\nu_A)}(b,c)}{W_{{\mathbf{T}}(\nu_A)}^{(q)}(b,c)}W_{{\mathbf{T}}(\nu_A)}^{(q)}(x,c),\qquad c\leq x\leq b,
			\end{align}
			where $Z^{(q)}_{{\mathbf{T}}(\nu_A)}$ is defined in \eqref{sf_Z} with $\nu={\mathbf{T}}(\nu_A)$. Moreover, $Z^{(q)}_{{\mathbf{T}}(\nu_A)}$ is the unique solution to the following integral equation,
			\begin{align}\label{volterra_sf_Z_new}
				Z^{(q)}_{{\mathbf{T}}(\nu_A)}(x,c)=Z^{(q)}(x-c)+\int_{[c,x]}W^{(q)}(x-u)Z^{(q)}_{{\mathbf{T}}(\nu_A)}(u,c){\mathbf{T}}(\nu_A)(du),\qquad x\geq c.
			\end{align}
		\end{theorem}
		{We now give an expression for the} resolvent measure of the L\'evy process $X$ killed by the PcNAF $A$ in terms of the scale function $W_{\mathbf{T}(\nu_A)}$. This is the content of the next result.
\begin{theorem}\label{resol_ts}
	Let $q \geq 0$, $x \in (c,b)$, and let $f:\R \to \R$ be a bounded, non-negative measurable function. Then
	\begin{align}\label{resolvent_iden}
		\E_x\left[\int_0^{\tau_b^+\wedge\tau_c^-} e^{-qt-A_t} f(X_t)\,dt\right]
		= \int_c^b f(y)\left(
		\frac{W_{{\mathbf{T}}(\nu_A)}^{(q)}(x,c)}{W_{{\mathbf{T}}(\nu_A)}^{(q)}(b,c)}
		W_{{\mathbf{T}}(\nu_A)}^{(q)}(b,y)
		- W_{{\mathbf{T}}(\nu_A)}^{(q)}(x,y)
		\right)dy.
	\end{align}
\end{theorem}
		{
		\begin{remark}
	Note that, for the PcNAF \(A^{\omega}\) defined in \eqref{Af}, the occupation density formula
	(see (2) in Chapter V of \cite{B} and (1.5) in \cite{FP}) implies that its associated Revuz
	measure is
	\[
	\nu_{A^{\omega}}(dz) = \omega(z)\,dz.
	\]
	Since \(\nu_{A^{\omega}}\) is diffuse, we have \(\mathbf{T}(\nu_{A^{\omega}})=\omega(z)\,dz\).
	By \eqref{volterra_sf_w} and \eqref{volterra_sf_Z_new}, the associated scale functions
	\(W^{(q)}_{\mathbf{T}(\nu_{A^{\omega}})}\) and
	\(Z^{(q)}_{\mathbf{T}(\nu_{A^{\omega}})}\) are the unique solutions of the following integral
	equations, for \(x \ge y\):
	\begin{align*}
		W^{(q)}_{\mathbf{T}(\nu_{A^{\omega}})}(x,y)
		&= W^{(q)}(x-y)
		+ \int_{[y,x]} W^{(q)}(x-z)\,
		W^{(q)}_{\mathbf{T}(\nu_{A^{\omega}})}(z,y)\,\omega(z)\,dz, \\
		Z^{(q)}_{\mathbf{T}(\nu_{A^{\omega}})}(x,c)
		&= Z^{(q)}(x-c)
		+ \int_{[c,x]} W^{(q)}(x-z)\,
		Z^{(q)}_{\mathbf{T}(\nu_{A^{\omega}})}(z,c)\,\omega(z)\,dz.
	\end{align*}
	As a consequence, Theorems \ref{main_sf}, \ref{two_sided_down_thm}, and \ref{resol_ts}
	generalize Theorems~2.1 and~2.2 of \cite{LP}.
		\end{remark}}
		\subsection{One-sided exit problem upwards}\label{one_sided_section}
		Fix $d\in\R$. Throughout this section we will assume that the PcNAF $A$ satisfies that
		\[
		A_t=\int_0^t1_{\{X_s\geq d\}}dA^1_s+\int_0^t\eta1_{\{X_s<d\}}ds,\qquad t>0,
		\]
		where $A^1$ is a PcNAF with Revuz measure $\nu_{A^1}$. In this case, the Revuz measure $\nu_A$, associated to the PcNAF $A$,  can be written as 
		\begin{equation}\label{revu_dec_one_sided}
			\nu_A(dz)=1_{\{z\geq d\}}\nu_A^1(dz)+\eta1_{\{z<d\}}dz.
		\end{equation}
		We now provide the following auxiliary result.
		\begin{lemma}\label{lemma_one_sided_limit}
			For $q\geq 0$, let $W_{{\mathbf{T}}(\nu_A)}^{(q)}(\cdot,c)$ be the unique solution to \eqref{volterra_sf_w}. Then, as $c\to -\infty$, the function
			$
			\frac{W_{{\mathbf{T}}(\nu_A)}^{(q)}(\cdot,c)}{W^{(\eta+q)}(-c)}
			$
			converges pointwise to a function $u^{(q)}_{{\mathbf{T}}(\nu_A)}$, which is the unique solution to the integral equation
			\begin{align}\label{fun_u}
				u^{(q)}_{{\mathbf{T}}(\nu_A)}(x)
				=
				e^{\Phi(\eta+q)x}
				+\int_{[d,x]}W^{(\eta+q)}(x-z)u^{(q)}_{{\mathbf{T}}(\nu_A)}(z)\left({\mathbf{T}}(\nu_A)(dz)-\eta\,dz\right),
				\qquad x\in\mathbb{R}.
			\end{align}
		\end{lemma}
		With the previous auxiliary result we now provide the solution to the one-sided exit problem upwards along with the {associated} resolvent. 
		\begin{proposition}\label{one_sided_upwards}
			Let $x \leq b$ and $q \geq 0$. Then
			\begin{align}\label{one_sided_iden_down}
				\E_x\left[e^{-q\tau_b^+-A_{\tau_b^+}};\,\tau_b^+<\infty\right]
				= \frac{u^{(q)}_{{\mathbf{T}}(\nu_A)}(x)}{u^{(q)}_{{\mathbf{T}}(\nu_A)}(b)}.
			\end{align}
			Moreover, for any bounded, non-negative measurable function $f:\R\to\R$,
			\begin{align}\label{one_sided_resolvent_down}
				\E_x\left[\int_0^{\tau_b^+} e^{-qt-A_t} f(X_t)\,dt\right]
				= \int_{-\infty}^b f(y)\left(
				\frac{u_{{\mathbf{T}}(\nu_A)}^{(q)}(x)}{u_{{\mathbf{T}}(\nu_A)}^{(q)}(b)}
				W_{{\mathbf{T}}(\nu_A)}^{(q)}(b,y)
				- W_{{\mathbf{T}}(\nu_A)}^{(q)}(x,y)
				\right)dy,
			\end{align}
			where $W_{{\mathbf{T}}(\nu_A)}^{(q)}$ and $u^{(q)}_{{\mathbf{T}}(\nu_A)}$ are the unique solutions to \eqref{volterra_sf_w} and \eqref{fun_u}, respectively.
		\end{proposition}
		\begin{proof}
			Identity \eqref{one_sided_iden_down} follows by taking $c\to-\infty$ in \eqref{two_sided_up} and using Lemma \ref{lemma_one_sided_limit}. On the other hand, note that identity \eqref{two_sided_up} implies that the mapping $c\mapsto {\frac{W_{{\mathbf{T}}(\nu_A)}^{(q)}(x,c)}{W_{{\mathbf{T}}(\nu_A)}^{(q)}(b,c)}}=\E_x\left[e^{-q\tau_b^+-A_{\tau_b^+}};\tau_b^+<\tau_c^-\right]$ is non-increasing for fixed $x\leq b$. Hence, taking $c\to-\infty$ in \eqref{resolvent_iden} and using monotone convergence we obtain \eqref{one_sided_resolvent_down}.
		\end{proof}
		\subsection{One-sided exit problem downwards}
		In this last section we provide the solution to the one-sided exit problem downwards, which is the content of the next result.
		\begin{proposition}\label{one_sided_downwards}
			{(i) For $x\geq c$ and {$q\geq0$}, {the limit $C^{(q)}_{{\mathbf{T}}(\nu_A)}(c):=\lim_{b\to\infty}\frac{Z^{(q)}_{{\mathbf{T}}(\nu_A)}(b,c)}{W^{(q)}_{{\mathbf{T}}(\nu_A)}(b,c)}$ is well-defined, and}}
			\begin{align}\label{one_sided_down_iden}
				\E_x\left[e^{-q\tau_c^--A_{\tau_c^-}};\tau_c^-<\infty\right]=Z^{(q)}_{{\mathbf{T}}(\nu_A)}(x,c)-C^{(q)}_{{\mathbf{T}}(\nu_A)}(c)W_{{\mathbf{T}}(\nu_A)}^{(q)}(x,c). 
			\end{align}
			
			(ii) For $x,y\geq c$ and $q\geq0$, the limit $c^{(q)}_{{\mathbf{T}}(\nu_A)}(y,c):=\lim_{b\to\infty}\frac{W_{{\mathbf{T}}(\nu_A)}^{(q)}(b,y)}{W_{{\mathbf{T}}(\nu_A)}^{(q)}(b,c)}$ is well-defined. {Moreover, for any bounded, non-negative measurable function $f:\R\to\R$,
			\begin{align}\label{resol_down}
				\E_x\left[\int_0^{\tau_c^-}e^{-qt-A_t}f(X_t)dt\right]=\int_c^{\infty}f(y)\left(c^{(q)}_{{\mathbf{T}}(\nu_A)}(y,c)W_{{\mathbf{T}}(\nu_A)}^{(q)}(x,c)-W_{{\mathbf{T}}(\nu_A)}^{(q)}(x,y)\right)dy.
			\end{align}}
			{Additionally, we have that $c^{(q)}_{{\mathbf{T}}(\nu_A)}(y,c)<\infty$ for every $y\geq c$.} 
		\end{proposition}
	\section{Some preliminary results}\label{sec:prelimiary_results}
	The proofs of the main results rely on a two-step strategy based on the representation of positive co-natural additive functionals established in Section~3. Recall that any PcNAF can be described in terms of its Revuz measure, which is a Radon measure on $\mathbb{R}$, and that this representation allows one to express the additive functional as a mixture of local times. This viewpoint enables us to formulate the problem entirely at the level of measures and associated integral equations.
		
		Our approach proceeds by first analyzing the case in which the Revuz measure is finite and atomic. In this setting, the additive functional reduces to a finite linear combination of local times, and the corresponding fluctuation identities can be obtained by extending the recursive arguments developed in \cite{LZ}, allowing in particular for both bounded and unbounded variation cases. Moreover, we show that the associated generalized scale functions can be characterized as solutions to integral equations. This case serves as the foundation of our analysis, as it provides explicit constructions of the generalized scale functions and reveals the structure of the associated integral equations.
		
		In a second step, we extend these results to general Radon measures. This is achieved by approximating the Revuz measure with a sequence of random atomic measures constructed from Poisson random measures. The corresponding generalized scale functions then satisfy random integral equations, and we establish their convergence to the deterministic solutions associated with the original measure. This approximation procedure allows us to transfer the fluctuation identities from the atomic setting to the general case.
		
		The section is organized as follows. In Section~\ref{sec:local_times} we treat the case of finite mixtures of local times and derive the corresponding fluctuation identities. Section~\ref{sec:integral equations} is devoted to the study of structural properties of the associated integral equations. Finally, in Section~\ref{sec:convergence_sf} we establish convergence results for these equations under Poissonian approximations, which are instrumental in completing the proofs of the main results.
	\subsection{Fluctuation identities for finite mixtures of local times}\label{sec:local_times}
	In this section, we establish the fluctuation identities in the case where the additive functional is a finite mixture of local times. This corresponds to the situation where the associated Revuz measure is a finite atomic measure. This case constitutes the first step in our approach and serves as the foundation for the general results obtained later.
	
	More precisely, we consider a positive co-natural additive functional $A$ of the form
	$
	A_t = \sum_{i=1}^n p_i L^{a_i}_t$,  $t \ge 0$,
	where $a_i \in \mathbb{R}$, $p_i > 0$, and $L^{a_i}$ denotes the local time of the Lévy process $X$ at level $a_i$. In this case, the associated Revuz measure is given by
$
	\nu_A(dz) = \sum_{i=1}^n p_i \delta_{a_i}(dz).
$
	
	Our goal is to derive explicit expressions for the fluctuation identities in this setting and to characterize the associated generalized scale functions. To this end, we follow and extend the recursive approach developed in \cite{LZ}, allowing in particular for both bounded and unbounded variation cases. Finally, we show that the resulting scale functions can be represented as solutions to integral equations.
	
	{                                                                                                                                                                                                                                   Throughout this section we fix $c\in\R$.}
	Now, consider $\{p_k>0\}$ a positive sequence and $\{a_k\}_{k\geq 1}$ a 
	strictly increasing sequence with $a_1>c$ for some $c\in\R$. Then, following \cite{LZ}, we start by introducing the following functions, defined inductively for {$q\geq 0$}, $x,y\in\R$ and $k\geq 2$:
	\begin{align}\label{volterra_lt}
		W^{(q,p_1,\dots,p_k)}_{(a_1,\dots,a_k)}(x,y)&:=W^{(q,p_1,\dots,p_{k-1})}_{(a_1,\dots,a_{k-1})}(x,y)+{\mathbf{d}}(e^{p_k/{\mathbf{d}}}-1)W^{(q)}(x-a_k)W^{(q,p_1,\dots,p_{k-1})}_{(a_1,\dots,a_{k-1})}(a_k,y),\\
		\label{volterra_lt_reso}
		Z^{(q,p_1,\dots,p_k)}_{(a_1,\dots,a_k)}(x,y)&:=Z^{(q,p_1,\dots,p_{k-1})}_{(a_1,\dots,a_{k-1})}(x,y)+{\mathbf{d}}(e^{p_k/{\mathbf{d}}}-1)W^{(q)}(x-a_k)Z^{(q,p_1,\dots,p_{k-1})}_{(a_1,\dots,a_{k-1})}(a_k,y),
	\end{align} 
	with initial conditions
	\begin{align}\label{volterra_lt_initial}
		W^{(q,p_1)}_{(a_1)}(x,y)&:=W^{(q)}(x-y)+{\mathbf{d}}(e^{p_1/{\mathbf{d}}}-1)W^{(q)}(x-a_1)W^{(q)}(a_1-y),\notag\\
		Z^{(q,p_1)}_{(a_1)}(x,y)&:=Z^{(q)}(x-y)+{\mathbf{d}}(e^{p_1/{\mathbf{d}}}-1)W^{(q)}(x-a_1)Z^{(q)}(a_1-y).
	\end{align}
When the Lévy process $X$ has paths of unbounded variation, these expressions are understood in the limiting sense as $\mathbf{d} \to \infty$, in which case $\mathbf{d}(e^{p/\mathbf{d}} - 1) = p$.

The following result provides the fluctuation identities for this class of additive functionals. This result extends Theorem 3.2 in \cite{LZ} to the case of bounded variation. We defer the proof to Appendix \ref{LiZhouProof}.
	\begin{theorem}\label{LiZhouProof_thm}
		For {$q\geq0$} and $x \in (c,b)$, we have 
		\begin{align}
			\E_x\left[e^{-q\tau_b^+-\sum_{i=1}^np_iL^{a_i}_{\tau_b^+}};\tau_b^+<\tau_c^-\right]=\frac{W^{(q,p_1,\dots,p_n)}_{(a_1,\dots,a_n)}(x,c)}{W^{(q,p_1,\dots,p_n)}_{(a_1,\dots,a_n)}(b,c)},
			\label{local_time_exit}
		\end{align}	
		and	
		\begin{align}\label{tsd_u_0}
			\E_x\left[e^{-q\tau_c^--\sum_{i=1}^np_iL^{a_i}_{\tau_c^-}};\tau_c^-<\tau_b^+\right]=Z^{(q,p_1,\dots,p_n)}_{(a_1,\dots,a_n)}(x,c)-\frac{W^{(q,p_1,\dots,p_n)}_{(a_1,\dots,a_n)}(x,c)}{W^{(q,p_1,\dots,p_n)}_{(a_1,\dots,a_n)}(b,c)}Z^{(q,p_1,\dots,p_n)}_{(a_1,\dots,a_n)}(b,c).
		\end{align}	
		{Additionally, for $x\in(c,b)$ and any bounded measurable function $f:\R\to\R$,
			\begin{align}
				\E_x\Bigg[\int_0^{\tau_b^+\wedge\tau_c^-}&e^{-qt-\sum_{i=1}^np_iL^{a_i}_{t}}f(X_t)dt\Bigg]\notag\\
				&=\int_c^bf(y)\left(\frac{W^{(q,p_1,\dots,p_n)}_{(a_1,\dots,a_n)}(x,c)}{W^{(q,p_1,\dots,p_n)}_{(a_1,\dots,a_n)}(b,c)}W^{(q,p_1,\dots,p_n)}_{(a_1,\dots,a_n)}(b,y)-W^{(q,p_1,\dots,p_n)}_{(a_1,\dots,a_n)}(x,y)\right)dy. \label{local_time_resolvent}
		\end{align}}
	\end{theorem}
	%
We now show that $W^{(q,p_1,\dots,p_n)}_{(a_1,\dots,a_n)}$ can be characterized as the unique solution to an integral equation. {This characterization is crucial for deriving the fluctuation identities in the general case, since it enables us to approximate the scale functions associated with general PcNAFs by those corresponding to finite linear combinations of local times at distinct levels.}
	\begin{lemma}\label{sf_lt_lemma}
		The scale functions $W^{(q,p_1,\dots,p_n)}_{(a_1,\dots,a_n)}$  and $Z^{(q,p_1,\dots,p_n)}_{(a_1,\dots,a_n)}$ defined in \eqref{volterra_lt}, are the unique solutions to the following integral equations
		\begin{align}
			W^{(q,p_1,\dots,p_n)}_{(a_1,\dots,a_n)}(x,y)&=W^{(q)}(x-y)+\int_{[y,x]}W^{(q)}(x-z)W^{(q,p_1,\dots,p_n)}_{(a_1,\dots,a_n)}(z,y)\nu^{(p_1,\dots,p_n)}_{(a_1,\dots,a_n)}(dz)\label{volterra_lt_2_W}\\
			Z^{(q,p_1,\dots,p_n)}_{(a_1,\dots,a_n)}(x,y)&=Z^{(q)}(x-y)+\int_{[y,x]}W^{(q)}(x-z)Z^{(q,p_1,\dots,p_n)}_{(a_1,\dots,a_n)}(z,y)\nu^{(p_1,\dots,p_n)}_{(a_1,\dots,a_n)}(dz),\label{volterra_lt_2_Z}
		\end{align}
		for $x,y\in\R$, respectively, where
		\[
		\nu^{(p_1,\dots,p_n)}_{(a_1,\dots,a_n)}(dy):=\sum_{i=1}^n{\mathbf{d}}(1-e^{-p_i/{\mathbf{d}}})\delta_{a_i}(dy).
		\]
	\end{lemma}
	\begin{proof}
		We provide the proof only for $W^{(q, p_1, \dots, p_n)}_{(a_1, \dots, a_n)}$, as the proof for $Z^{(q, p_1, \dots, p_n)}_{(a_1, \dots, a_n)}$ follows similar lines. We prove the statement by induction. 
		
		(i) {We start with the case $n=1$. Using \eqref{volterra_lt_initial} we have for $x,y\in\R$
			\begin{align}\label{ind_n=1}
				W^{(q)}(x-y)&+\int_{[y,x]}W^{(q)}(x-z)W^{(q,p_1)}_{(a_1)}(z,y)\nu_{(a_1)}^{(p_1)}(dz) \notag\\&=W^{(q)}(x-y)+{\mathbf{d}}(1-e^{-p_1/{\mathbf{d}}})W^{(q)}(x-a_1)W^{(q,p_1)}_{(a_1)}(a_1,y)\notag\\
				&=W^{(q)}(x-y)\notag\\&+{\mathbf{d}}(1-e^{-p_1/{\mathbf{d}}})W^{(q)}(x-a_1)\left(W^{(q)}(a_1-y)+{\mathbf{d}}(e^{p_1/{\mathbf{d}}}-1)W^{(q)}(0+)W^{(q)}(a_1-y)\right)\notag\\
				&=W^{(q)}(x-y)+{\mathbf{d}}(1-e^{-p_1/{\mathbf{d}}})W^{(q)}(x-a_1)\left(W^{(q)}(a_1-y)+(e^{p_1/{\mathbf{d}}}-1)W^{(q)}(a_1-y)\right)\notag\\
				&=W^{(q)}(x-y)+{\mathbf{d}}(e^{p_1/{\mathbf{d}}}-1)W^{(q)}(x-a_1)W^{(q)}(a_1-y)\notag\\&=W^{(q,p_1)}_{(a_1)}(x,y),
			\end{align}
			which proves that indeed $W^{(q,p_1)}_{(a_1)}$ satisfies \eqref{volterra_lt_2_W} with $n=1$.}

		(ii) We now assume that $W^{(p_1,\dots,p_{n-1})}_{(a_1,\dots,a_{n-1})}$ satisfies \eqref{volterra_lt_2_W} and proceed to prove for $n$. To this end, we note that
		for $x,y\in\R$
		\begin{align*}
			&W^{(q)}(x-y)+\int_{[y,x]}W^{(q)}(x-z)W^{(q,p_1,\dots,p_n)}_{(a_1,\dots,a_n)}(z,y)\nu^{(p_1,\dots,p_n)}_{(a_1,\dots,a_n)}(dz)\notag\\&=W^{(q)}(x-y)+\sum_{i=1}^n\mathbf{d}(1-e^{-p_i/\mathbf{d}})W^{(q)}(x-a_i)W^{(q,p_1,\dots,p_n)}_{(a_1,\dots,a_n)}(a_i,y)\notag\\
			&=W^{(q)}(x-y)\notag\\&+\sum_{i=1}^n{\mathbf{d}}(1-e^{-p_i/{\mathbf{d}}})W^{(q)}(x-a_i)\left(W^{(q,p_1,\dots,p_{n-1})}_{(a_1,\dots,a_{n-1})}(a_i,y)+{\mathbf{d}}(e^{p_n/{\mathbf{d}}}-1)W^{(q)}(a_i-a_n)W^{(q,p_1,\dots,p_{n-1})}_{(a_1,\dots,a_{n-1})}(a_n,y)\right)\notag\\
			&=W^{(q)}(x-y)+\sum_{i=1}^{n-1}{\mathbf{d}}(1-e^{-p_i/{\mathbf{d}}})W^{(q)}(x-a_i)W^{(q,p_1,\dots,p_{n-1})}_{(a_1,\dots,a_{n-1})}(a_i,y)\notag\\
			&+{\mathbf{d}}(1-e^{-p_n/{\mathbf{d}}})W^{(q)}(x-a_n)\left(W^{(q,p_1,\dots,p_{n-1})}_{(a_1,\dots,a_{n-1})}(a_n,y)+{\mathbf{d}}(e^{p_n/{\mathbf{d}}}-1)W^{(q)}(0)W^{(q,p_1,\dots,p_{n-1})}_{(a_1,\dots,a_{n-1})}(a_n,y)\right)\notag\\
			&=W^{(q)}(x-y)+\sum_{i=1}^{n-1}{\mathbf{d}}(1-e^{-p_i/{\mathbf{d}}})W^{(q)}(x-a_i)W^{(q,p_1,\dots,p_{n-1})}_{(a_1,\dots,a_{n-1})}(a_i,y)\notag\\&+{\mathbf{d}}(e^{p_n/{\mathbf{d}}}-1)W^{(q)}(x-a_n)W^{(q,p_1,\dots,p_{n-1})}_{(a_1,\dots,a_{n-1})}(a_n,y)\notag\\
			&=W^{(q)}(x-y)+\int_{[y,x]}W^{(q)}(x-z)W^{(q,p_1,\dots,p_{n-1})}_{(a_1,\dots,a_{n-1})}(z,y)\nu^{(p_1,\dots,p_{n-1})}_{(a_1,\dots,a_{n-1})}(dz)\notag\\&+{\mathbf{d}}(e^{p_n/{\mathbf{d}}}-1)W^{(q)}(x-a_n)W^{(q,p_1,\dots,p_{n-1})}_{(a_1,\dots,a_{n-1})}(a_n,y)\notag\\
			&=W^{(q,p_1,\dots,p_{n-1})}_{(a_1,\dots,a_{n-1})}(x,y)+{\mathbf{d}}(e^{p_n/{\mathbf{d}}}-1)W^{(q)}(x-a_n)W^{(q,p_1,\dots,p_{n-1})}_{(a_1,\dots,a_{n-1})}(a_n,y)\notag\\
			&=W^{(q,p_1,\dots,p_n)}_{(a_1,\dots,a_n)}(x,y),
		\end{align*}
		where in sixth equality we used the induction assumption and in the second and last equality we used \eqref{volterra_lt}. 
		
		Finally the fact that $W^{(q, p_1, \dots, p_n)}_{(a_1, \dots, a_n)}$ is the unique solution to \eqref{volterra_lt_2_W} follows from Proposition \ref{exist_uni}.
	\end{proof}
	\begin{remark}\label{remark_revuz_measure} 
		Recall that the measure $\sum_{i=1}^n p_i \delta_{a_i}$ is the Revuz measure associated with the PcNAF \linebreak $\left\{\sum_{i=1}^n p_i L^{a_i}_t : t \geq 0 \right\}$. Moreover,
		$
		\nu^{(p_1,\dots,p_n)}_{(a_1,\dots,a_n)}(dz)
		= \mathbf{T}\!\left(\sum_{i=1}^n p_i \delta_{a_i}\right)(dz).
		$
		In particular, by Lemma~\ref{sf_lt_lemma}, Theorem~\ref{LiZhouProof_thm} can be viewed as a special case of Theorems~\ref{main_sf} and \ref{resol_ts}.
	\end{remark}
	\subsection{Properties of solutions to integral equations driven by Radon measures}\label{sec:integral equations}
	In this section, we provide some auxiliary results on the properties of solutions to integral equations driven by Radon measures. For any Radon measure $\nu$ such that the assumptions on Proposition \ref{exist_uni} hold (with $\lambda\equiv0$), we denote by $W^{(q)}_{\nu}$ to the solution to the following integral equation:
	\begin{align}\label{volterra_nu}
	W^{(q)}_{\nu}(x,y)=W^{(q)}(x-y)+\int_{[y,x]}W^{(q)}(x-u)W^{(q)}_{\nu}(u,y)\nu(du),\qquad x,y\in\R.
	\end{align}
	which is unique by Proposition \ref{exist_uni}.
	
	First, we will provide an alternative form of the equation \eqref{volterra_nu}.
	\begin{lemma}\label{alt_lemma}
		For $x,y\in\R$, and a Radon measure $\nu$, we have
		\begin{align}\label{alt_volterra}
			W^{(q)}_\nu(x,y)=W^{(q)}(x-y)+\int_{[y,x]}W_\nu^{(q)}(x,z)W^{(q)}(z-y)\nu(dz).
		\end{align}
	\end{lemma}
	\begin{proof}
		Let us denote by $f(x,y)$ the right-hand side of \eqref{alt_volterra}, then using \eqref{volterra_nu}, and Fubini's Theorem, we obtain
		\begin{align*}
			f(x,y)&=W^{(q)}(x-y)+\int_{[y,x]}\left(W^{(q)}(x-z)+\int_{[z,x]}W^{(q)}(x-u)W^{(q)}_\nu(u,z)\nu(du)\right)W^{(q)}(z-y)\nu(dz)\notag\\
			&=W^{(q)}(x-y)+\int_{[y,x]}W^{(q)}(x-u)W^{(q)}(u-y)\nu(du)\notag\\
			&+\int_{[y,x]}W^{(q)}(x-u)\int_{[y,u]}W^{(q)}_\nu(u,z)W^{(q)}(z-y)\nu(dz)\nu(du)\notag\\
			&=W^{(q)}(x-y)+\int_{[y,x]}W^{(q)}(x-u)f(u,y)\nu(du).
		\end{align*}
		Hence, $f$ is a solution to \eqref{volterra_nu}, and by the uniqueness of the solution to \eqref{volterra_nu}, we conclude that $f(x,y)=W^{(q)}_\nu(x,y)$ for $x,y\in\R$.
	\end{proof}
{In the following result, we establish a connection between the solutions of integral equations driven by two distinct Radon measures.}
	\begin{lemma}\label{equiv_lemma}
		Let $\nu_1$ and $\nu_2$ be two Radon measures. Consider two solutions $W^{(q)}_{\nu_1}$ and $W^{(q)}_{\nu_2}$ to \eqref{volterra_nu} driven by $\nu_1$ and $\nu_2$, respectively. Then, we have
		\begin{align*}
			W^{(q)}_{\nu_1}(x,y)-W^{(q)}_{\nu_2}(x,y)=\int_{[y,x]}W^{(q)}_{\nu_1}(x,u)W^{(q)}_{\nu_2}(u,y)(\nu_1-\nu_2)(du),\qquad x,y\in\R.
		\end{align*}
	\end{lemma}
	\begin{proof}
		For $x,y\in\R$, we note that, by \eqref{volterra_nu} together with Lemma \ref{alt_lemma} {and an application of Fubini's theorem,}
		\begin{align*}
			\int_{[y,x]}W^{(q)}_{\nu_1}(x,z)&W^{(q)}_{\nu_2}(z,y)\nu_2(dz)=\int_{[y,x]}W^{(q)}(x-z)W^{(q)}_{\nu_2}(z,y)\nu_2(dz)\notag\\
			&+\int_{[y,x]}\left(\int_{[z,x]}W^{(q)}_{\nu_1}(x,u)W^{(q)}(u-z)\nu_1(du)\right)W^{(q)}_{\nu_2}(z,y)\nu_2(dz)\notag\\
			&=\int_{[y,x]}W^{(q)}(x-z)W^{(q)}_{\nu_2}(z,y)\nu_2(dz)\notag\\
			&+\int_{[y,x]}W^{(q)}_{\nu_1}(x,u)\int_{[y,u]}W^{(q)}(u-z)W^{(q)}_{\nu_2}(z,y)\nu_2(dz)\nu_1(du)\notag\\
			&=W^{(q)}_{\nu_2}(x,y)-W^{(q)}(x-y)\notag\\
			&+\int_{[y,x]}W^{(q)}_{\nu_1}(x,u)\left(W^{(q)}_{\nu_2}(u,y)-W^{(q)}(u-y)\right)\nu_1(du)\notag\\
			&=W^{(q)}_{\nu_2}(x,y)-W^{(q)}_{\nu_1}(x,y)+\int_{[y,x]}W^{(q)}_{\nu_1}(x,u)W^{(q)}_{\nu_2}(u,y)\nu_1(du),
		\end{align*}
		which implies the result.
	\end{proof}
		In the last result of this section, we state the monotonicity property of solutions to such integral equations.
	\begin{lemma}\label{lemma_mon}
		Consider a Radon measure $\nu$, and the unique solution $W^{(q)}_{\nu}$ to {\eqref{volterra_nu}}.
		Then for fixed $y\in\R$ the mapping $x\mapsto W^{(q)}_{\nu}(x,y)$ 
		is non-decreasing. Additionally, for fixed $x\in\R$ the mapping $y\mapsto W^{(q)}_{\nu}(x,y)$ is non-increasing.
	\end{lemma}
	\begin{proof}
		(i) Fix $y\in\R$ and let $x_2\geq x_1\geq y$, then by \eqref{volterra_nu} we obtain
		\begin{align*}
			W_{\nu}^{(q)}(x_2,y)-W_{\nu}^{(q)}(x_1,y)&=W^{(q)}(x_2-y)-W^{(q)}(x_1-y)+\int_{(x_1,x_2]}W^{(q)}(x_2-z)W^{(q)}_{\nu}(z,y)\nu(dz)\notag\\
			&+\int_{[y,x_1]}(W^{(q)}(x_2-z)-W^{(q)}(x_1-z))W^{(q)}_{\nu}(z,y)\nu(dz)\geq0,
		\end{align*}
		where the last inequality follows from the fact that the mapping $x\mapsto W^{(q)}(x)$ is non-decreasing, and the fact that since $W^{(q)}(x)\geq0$ for all $x\in\R$, Proposition \ref{exist_uni} implies that $W^{(q)}_{\nu}(x,y)\geq0$ for $x\geq y$. Therefore,  $W_{\nu}^{(q)}(x_2,y)\geq W_{\nu}^{(q)}(x_1,y)$ which proves the claim.
		
		(ii) Fix $x\in\R$ and let {$x\geq y_2\geq y_1$}, then by Lemma \ref{alt_lemma}, we obtain
		\begin{align}\label{volterra_bounded}
			W_{\nu}^{(q)}(x,y_1)-W_{\nu}^{(q)}(x,y_2)&=W^{(q)}(x-y_1)-W^{(q)}(x-y_2)\\
			&+\int_{[y_2,x]}W^{(q)}_{\nu}(x,z)(W^{(q)}(z-y_1)-W^{(q)}(z-y_2))\nu(dz)\notag\\&+\int_{[y_1,y_2)}W^{(q)}_{\nu}(x,z)W^{(q)}(z-y_1)\nu(dz)\geq0. 
		\end{align}
		where the last inequality follows from the fact that the mapping $x\mapsto W^{(q)}(x)$ is non-decreasing and that \linebreak $W^{(q)}_{\nu}(x,z)\geq0$ for $x\geq z$. Therefore $W_{\nu}^{(q)}(x,y_1)\geq W_{\nu}^{(q)}(x,y_2)$ which proves the claim.
	\end{proof}
	{In the next result, we characterize} the function $Z^{(q)}_{\nu}$ given in \eqref{sf_Z} as the solution to an integral equation driven by a Radon measure $\nu$.
	\begin{lemma}\label{sf_Z_volterra}
	The function $Z^{(q)}_{\nu}$ defined in \eqref{sf_Z} is the unique solution to the integral equation
	\begin{align}\label{volterra_sf_Z}
		Z^{(q)}_{\nu}(x,c)=Z^{(q)}(x-c)+\int_{[c,x]}W^{(q)}(x-u)Z^{(q)}_{\nu}(u,c)\nu(du),\qquad x\geq c.
	\end{align}
	\end{lemma}
	\begin{proof}
	Note that uniqueness of the solution to \eqref{volterra_sf_Z} follows from Proposition \ref{exist_uni}. {Now, substituting \eqref{volterra_nu} into \eqref{sf_Z}, {using the definition of $Z^{(q)}$ in \eqref{fun_z_sf},} and applying Fubini's theorem, we obtain, for $x \geq c$, 
	}
	\begin{align*}
		Z^{(q)}_{\nu}(x,c)&=1+\int_{[c,x]}W^{(q)}(x-z)\nu(dz)+\int_{[c,x]}\int_{[z,x]}W^{(q)}(x-u)W^{(q)}_{\nu}(u,z)\nu(du)\nu(dz)\notag\\
		&+q\int_c^xW^{(q)}(x-z)dz+q\int_c^x\int_{[z,x]}W^{(q)}(x-u)W^{(q)}_{\nu}(u,z)\nu(du)dz\notag\\
		&=1+\int_{[c,x]}W^{(q)}(x-z)\nu(dz)+\int_{[c,x]}W^{(q)}(x-u)\int_{[c,u]}W^{(q)}_{\nu}(u,z)\nu(dz)\nu(du)\notag\\
		&+q\int_c^xW^{(q)}(x-z)dz+q\int_{[c,x]}W^{(q)}(x-u)\int_c^uW^{(q)}_{\nu}(u,z)dz\nu(du)\notag\\
		&=Z^{(q)}(x-c)+\int_{[c,x]}W^{(q)}(x-u)\left(1+\int_{[c,u]}W^{(q)}_{\nu}(u,z)\nu(dz)+q\int_c^uW^{(q)}_{\nu}(u,z)dz\right)\nu(du)\notag\\
		&=Z^{(q)}(x-c)+\int_{[c,x]}W^{(q)}(x-u)Z^{(q)}_{\nu}(u,c)\nu(du),
	\end{align*}
	which proves the claim.
	\end{proof}
	\subsection{Convergence of solutions of integral equations driven by Poisson random measures}\label{sec:convergence_sf}
			Fix a Radon measure $\nu$ on $\R$. We consider the decomposition $\nu=\nu^a+\nu^d$, where $\nu^a$ denotes the atomic part of $\nu$ and $\nu^d$ its diffuse part. Let $N_n(dz)$ be a Poisson random measure on $\R$ with intensity measure $n\nu^d(dz)$, and define, for each $n\in\N$,
			\begin{align}\label{measure_nu_n}
				\nu_n(dz)=\nu^a(dz)+\frac{1}{n}N_n(dz).
			\end{align} 
			{Fix $T>0$. For each $n\geq 1$ and $x\in[0,T]$, let $\mathcal{G}_x^n:=\sigma\{N_n[0,y]:y\leq x\}$ be the natural filtration associated with $N_n$, {and define the compensated Poisson random measure by $\tilde{N}_n(dz):=N_n(dz)-n\nu^d(dz)$}. 
				
			Proceeding as in Section II.3 of \cite{IW} (see also Theorem 4.2.3 in
			\cite{APP}), if $F:[0,T]\times\Omega\to\R$ is a predictable process with
			respect to $(\mathcal{G}_x^n)_{x\in[0,T]}$ such that
			\[
			\E\left[\int_{[0,T]}F^2(y)n\nu^d(dy)\right]<\infty,
			\]
			then the process $\left(\int_{[0,x]}F(y)\tilde{N}_n(dy)\right)_{x\geq0}$ is a square integrable $\mathcal{G}_x^n$-martingale, such that
			\begin{align}\label{square_integrable}
			\E\left[\left|\int_{[0,x]}F(y)\tilde{N}_n(dy)\right|^2\right]=
			\E\left[\int_{[0,x]}F^2(y)n\nu^d(dy)\right],\qquad x\in[0,T].
			\end{align}}
		{
		Then, by \eqref{square_integrable}, we have
		\begin{align}\label{square_integrable_ineq}
		\E\left[\left|\int_{[0,x]}F(y)N_n(dy)\right|^2\right]&\leq2\left(\E\left[\left|\int_{[0,x]}F(y)\tilde{N}_n(dy)\right|^2\right]+\E\left[\left|\int_{[0,x]}F(y)n\nu^d(dy)\right|^2\right]\right)\notag\\
		&\leq 2\left(\E\left[\int_{[0,x]}F^2(y)n\nu^d(dy)\right]+\E\left[\left|\int_{[0,x]}F(y)n\nu^d(dy)\right|^2\right]\right).
		\end{align}}
			On the other hand, since $\nu(dz)$ is a Radon measure, we have
			\[
			\E\left[\nu_n[y,T]\right]=\nu^a[y,T]+\frac{1}{n}n\nu^d[y,T]=\nu[y,T]<\infty.
			\]
			Hence, $\nu_n[y,T]<\infty$ almost surely for all $n\ge1$. 
			
			Therefore, by Remark \ref{exi_uni_T}, for each $n\ge1$ there exists almost surely
			a unique solution $W_{{\mathbf{T}}(\nu_n)}^{(q)}(\cdot,y)$ on $[y,T]$ to the
			integral equation
			\begin{align}\label{volterra_approx}
				W_{{\mathbf{T}}(\nu_n)}^{(q)}(x,y)&=W^{(q)}(x-y)+\int_{[y,x]}W^{(q)}(x-z)W_{{\mathbf{T}}(\nu_n)}^{(q)}(z,y){\mathbf{T}}(\nu_n)(dz),\qquad x\in[y,T],
			\end{align}
			where ${\mathbf{T}}$ and $\nu_n$ are given by \eqref{oper_T} and \eqref{measure_nu_n}, respectively.
			
			Similarly, since $\nu[y,T]<\infty$, there exists a unique solution $W_{{\mathbf{T}}(\nu)}^{(q)}(\cdot,y)$ on $[y,T]$ to 
			\begin{align}\label{volterra_limit}
				W_{{\mathbf{T}}(\nu)}^{(q)}(x,y)&=W^{(q)}(x-y)+\int_{[y,x]}W^{(q)}(x-z)W_{{\mathbf{T}}(\nu)}^{(q)}(z,y){\mathbf{T}}(\nu)(dz),\qquad x\in[y,T].
			\end{align}
			{\begin{lemma}\label{conv_sf}
				Fix $T>y$. Consider the sequence
				$\bigl(W_{{\mathbf{T}}(\nu_n)}^{(q)}(\cdot,y)\bigr)_{n\ge1}$,
				where for each $n\ge1$,
				$W_{{\mathbf{T}}(\nu_n)}^{(q)}(\cdot,y)$ is a.s. the unique solution to
				\eqref{volterra_approx}. Then,
				\[				
				\E\left[\sup_{x\in[y,T]}
				\left|
				W_{{\mathbf{T}}(\nu_n)}^{(q)}(x,y)
				-
				W_{{\mathbf{T}}(\nu)}^{(q)}(x,y)
				\right|^2
				\right]
				\longrightarrow 0,
				\qquad \text{as $n\to\infty$},
				\]
				where $W_{{\mathbf{T}}(\nu)}^{(q)}(\cdot,y)$ is the unique solution to
				\eqref{volterra_limit}.
			\end{lemma}}
			\begin{proof}
				
				(i) {In this step, we derive an alternative form of the equations \eqref{volterra_approx} and \eqref{volterra_limit}, suitable for applying Gronwall's lemma in the form given in \cite[Appendix, Theorem 5.1]{EK}. To this end, note that, since the measure $\nu^d$ is diffuse, it follows from \cite[Theorem VI.2.17]{Cinlar} that the Poisson random measure {$N_n$} is almost surely a random counting measure. In particular, each of its
					atoms has unit mass almost surely.} Therefore,
				\begin{align}\label{atoms_Poisson}
				{\mathbf{T}}\left(\frac{1}{n}N_n\right)(dz)={{\mathbf{d}}(1-e^{-1/n{\mathbf{d}}})N_n(dz)}\qquad \text{$\mathbb{P}$-a.s.}
				\end{align}
				Let $A^a:=\{x\in\R:\nu^a(\{x\})>0\}$ denote the set of atoms of $\nu^a$.
				Then, combining \eqref{volterra_approx} with \eqref{oper_T} and
				\eqref{atoms_Poisson}, we deduce that for every $x\ge y$, 
				\begin{align}\label{jump_sf}
				W_{{\mathbf{T}}(\nu_n)}^{(q)}(x,y)-W_{{\mathbf{T}}(\nu_n)}^{(q)}(x-,y)&=W^{(q)}(0)W_{{\mathbf{T}}(\nu_n)}^{(q)}(x,y){\mathbf{T}}(\nu_n)(\{x\})\notag\\
				&=W_{{\mathbf{T}}(\nu_n)}^{(q)}(x,y)\left((1-e^{-\nu^a(\{x\})/{\mathbf{d}}})1_{\{x\in A^a\}}+(1-e^{-1/n{\mathbf{d}}})N_n(\{x\})\right),\quad  \text{{$\mathbb{P}$-a.s.}}
				\end{align}
				Indeed, since $\nu^d$ is diffuse, the sets of atoms of the measure $\nu^a$ and $N_n$ are disjoint almost surely. 
				Therefore, by \eqref{jump_sf} we obtain {for $x\in[y,T]$,}
								{\begin{align}\label{jump_sf_2}
				W_{\mathbf{T}(\nu_n)}^{(q)}(x,y)=W_{\mathbf{T}(\nu_n)}^{(q)}(x-,y)\left\{\left(e^{\nu^a(\{x\})/{\mathbf{d}}}1_{\{x\in A^a\}}+e^{1/n{\mathbf{d}}}N_n(\{x\})\right)\lor 1\right\},\qquad \text{{$\mathbb{P}$-a.s.}} 
				\end{align}
				Since the atoms of $\nu^a$ and $N_n$ are almost surely distinct, it follows from
				\eqref{atoms_Poisson} and \eqref{jump_sf_2} that \eqref{volterra_approx} can be
				rewritten in the following form: 
				\begin{align*}
				W_{{\mathbf{T}}(\nu_n)}^{(q)}(x,y)&=W^{(q)}(x-y)+\int_{[y,x]}W^{(q)}(x-z)W_{{\mathbf{T}}(\nu_n)}^{(q)}(z,y){\mathbf{T}}(\nu^a)(dz)\notag\\
				&+\int_{[y,x]}W^{(q)}(x-z)W_{{\mathbf{T}}(\nu_n)}^{(q)}(z,y){\mathbf{T}}\left(\frac{1}{n}N_n\right)(dz)\notag\\
				&=W^{(q)}(x-y)+\int_{[y,x)}W^{(q)}(x-z)W_{{\mathbf{T}}(\nu_n)}^{(q)}(z,y){\mathbf{T}}(\nu^a)(dz)+W^{(q)}(0)W_{{\mathbf{T}}(\nu_n)}^{(q)}(x,y){\mathbf{T}}(\nu^a)(\{x\})\notag\\
				&+\int_{[y,x]}W^{(q)}(x-z)W_{{\mathbf{T}}(\nu_n)}^{(q)}(z-,y){\mathbf{d}}(e^{1/n{\mathbf{d}}}-1)N_n(dz)
				,\qquad \text{{$x\in[y,T]$, $\mathbb{P}$-a.s.}}
				\end{align*}
				Solving for $W_{{\mathbf{T}}(\nu_n)}^{(q)}(x,y)$ yields
				\begin{align}\label{pred_volterra}
					W_{{\mathbf{T}}(\nu_n)}^{(q)}(x,y)&=K_a (x)\Bigg[W^{(q)}(x-y)+\int_{[y,x)}W^{(q)}(x-z)W_{{\mathbf{T}}(\nu_n)}^{(q)}(z,y){\mathbf{T}}(\nu^a)(dz)\notag\\
					&+\int_{[y,x]}W^{(q)}(x-z)W_{{\mathbf{T}}(\nu_n)}^{(q)}(z-,y){\mathbf{d}}(e^{1/n{\mathbf{d}}}-1)N_n(dz)\Bigg]
					,\qquad \text{{$x\in[y,T]$, \ $\mathbb{P}$-a.s.,}}
				\end{align}
				where $K_a(x):=(1-W^{(q)}(0){\mathbf{T}}(\nu^a)(\{x\}))^{-1}$.
				Similarly, by \eqref{volterra_limit}, we have for $x\in[y,T]$,
				\begin{align}\label{pred_volterra_2}
					W_{{\mathbf{T}}(\nu)}^{(q)}(x,y)&=K_a (x)\Bigg[W^{(q)}(x-y)+\int_{[y,x)}W^{(q)}(x-z)W_{{\mathbf{T}}(\nu)}^{(q)}(z,y){\mathbf{T}}(\nu^a)(dz)\notag\\
					&+\int_{[y,x]}W^{(q)}(x-z)W_{{\mathbf{T}}(\nu)}^{(q)}(z-,y)\nu^d(dz)\Bigg]
					,
				\end{align}}

				(ii) In this step, we derive estimates for  $\E\big[(W_{\mathbf{T}(\nu_n)}^{(q)}(x,y))^2\big]$, for $n\geq 1$. We first construct a predictable version of the process $\{W_{\mathbf{T}(\nu_n)}^{(q)}(x-,y): x\in[y,T]\}$ for each $n\geq 1$.
				
				{We denote by $\Omega_0$ the subset of $\Omega$ on which $\nu_n[y,T]<\infty$ and \eqref{atoms_Poisson} holds for all $n\in\N$. 
                Note that $\p(\Omega_0)=1$.
					{In particular, on $\Omega_0$ the measure $N_n$ has a finite number of atoms of unit mass on $[y,T]$.}	
					By Remark \ref{exi_uni_T}, 
					it follows that for every $\omega\in\Omega_0$ and $n\geq1$, the function {$W_{{\mathbf{T}}(\nu_n)}^{(q)}(\cdot,y)[\omega]$} is well defined as the unique solution to \eqref{volterra_approx} on $[y,T]$. 
					For each $n\ge1$ and $x\in[y,T]$, we define
					\begin{align*}
						W_{{\mathbf{T}}(\nu_n)}^{(q)}(x,y)[\omega]=
						\begin{cases} 
							W_{{\mathbf{T}}(\nu_n)}^{(q)}(x,y)[\omega]  &\text { if $\omega\in\Omega_0$, }  \\ 
							0 &  \text { if $\omega\in\Omega\backslash\Omega_0$. } 
						\end{cases}
					\end{align*}
					Then, by the proof of Proposition \ref{exist_uni} 
					together with
					\eqref{jump_sf_2}
					, the process {$\left\{W_{\mathbf{T}(\nu_n)}^{(q)}(x-,y):x\in[y,T]\right\}$} is $\mathcal{G}_x^n$-adapted and left-continuous. Hence, it is predictable
					(see Lemma 22.1 in \cite{Ka}).} Moreover, for each $m\ge1$, we define the stopping time $\tau_m:=\inf\{z\geq0:W_{{\mathbf{T}}(\nu_n)}^{(q)}(z,y)\geq m\}$. 
				
				By the definition of the operator ${\mathbf{T}}$, as in \eqref{oper_T}, we have
				\begin{align*}
					(1-W^{(q)}(0){\mathbf{T}}(\nu^a)(\{x\}))^{-1}=(1-(1-e^{-\nu^a(\{x\})/{\mathbf{d}}})1_{\{x\in A^a\}})^{-1}=1_{\{x\not\in A^a\}}+e^{\nu^a(\{x\})/{\mathbf{d}}}1_{\{x\in A^a\}}.
				\end{align*}
				Hence, the fact that $\nu^a$ is a Radon measure, implies that
				\begin{align}\label{constant_nu_a}
					C_a(T,y):=\sup_{x\in[y,T]}(1-W^{(q)}(0){\mathbf{T}}(\nu^a)(\{x\}))^{-1}=\sup_{x\in[y,T]}\left(1_{\{x\not\in A^a\}}+e^{\nu^a(\{x\})/{\mathbf{d}}}1_{\{x\in A^a\}}\right)<\infty.
				\end{align}
				Using \eqref{pred_volterra}, we have 
				\begin{align*}
					(W_{{\mathbf{T}}(\nu_n)}^{(q)}(x\wedge\tau_m,y))^2&\leq 3C_a^2(T,y)\Bigg[(W^{(q)}(x-y))^2+\left(\int_{[y,x\wedge\tau_m)}W^{(q)}(x-z)W_{{\mathbf{T}}(\nu_n)}^{(q)}(z,y){\mathbf{T}}(\nu^a)(dz)\right)^2\notag\\
					&+\left(\int_{[y,x\wedge\tau_m]}W^{(q)}(x-z)W_{{\mathbf{T}}(\nu_n)}^{(q)}(z-,y){\mathbf{d}}(e^{1/n{\mathbf{d}}}-1)N_n(dz)\right)^2\Bigg],\quad \text{{$x\in[y,T]$, \ $\mathbb{P}$-a.s.}}
				\end{align*}
				Then, taking expectations in the previous inequality and applying
				\eqref{square_integrable_ineq}, we obtain, for all $x\in[y,T]$,
				\begin{align*}
					\E\left[(W_{{\mathbf{T}}(\nu_n)}^{(q)}(x\wedge\tau_m,y))^2\right]&\leq {3}C_a^2(T,y)\Bigg\{(W^{(q)}(x-y))^2+\E\left[\left(\int_{[y,x\wedge\tau_m)}W^{(q)}(x-z)W_{{\mathbf{T}}(\nu_n)}^{(q)}(z,y){\mathbf{T}}(\nu^a)(dz)\right)^2\right]\notag\\
					&+{2}\E\left[\int_{[y,x\wedge\tau_m]}\left(W^{(q)}(x-z)W_{{\mathbf{T}}(\nu_n)}^{(q)}(z-,y){\mathbf{d}}(e^{1/n{\mathbf{d}}}-1)\right)^2n\nu^d(dz)\right]\notag\\
					&+{2}\E\Bigg[\left(\int_{[y,x\wedge\tau_m]}W^{(q)}(x-z)W_{{\mathbf{T}}(\nu_n)}^{(q)}(z-,y)n{\mathbf{d}}(e^{1/n{\mathbf{d}}}-1)\nu^d(dz)\right)^2\Bigg]\Bigg\}.
				\end{align*}
				Applying \eqref{ine_T}, the monotonicity of the function $x\mapsto W^{(q)}(x)$,
				and Jensen's inequality, together with the bound $1-e^{-u}\leq u$ for $u>0$,
				we obtain, for all $x\in[y,T]$,
				\begin{align}\label{bound_lemma_conv_1}
					\E\left[(W_{{\mathbf{T}}(\nu_n)}^{(q)}(x\wedge\tau_m,y))^2\right]&\leq {3}C_a^2(T,y)\Bigg\{(W^{(q)}(x-y))^2+\nu^a[y, x)\E\left[\int_{[y,x\wedge\tau_m)}\left(W^{(q)}(x-z)W_{{\mathbf{T}}(\nu_n)}^{(q)}(z,y)\right)^2\nu^a(dz)\right]
					\notag\\
					&{+\frac{2e^{2/n{\mathbf{d}}}}{n}\E\left[\int_{[y,x\wedge\tau_m]}\left(W^{(q)}(x-z)W_{{\mathbf{T}}(\nu_n)}^{(q)}(z-,y)n{\mathbf{d}}{(1-e^{-1/n{\mathbf{d}}})}\right)^2\nu^d(dz)\right]}\notag\\
					&{+{2e^{2/n{\mathbf{d}}}}\nu^d[y,x]\E\Bigg[\int_{[y,x\wedge\tau_m]}\left(W^{(q)}(x-z)W_{{\mathbf{T}}(\nu_n)}^{(q)}(z-,y)n{\mathbf{d}}(1-e^{-1/n{\mathbf{d}}})\right)^2\nu^d(dz)\Bigg]\Bigg\}}\notag\\
					&\leq {3}C_a^2(T,y)\Bigg[(W^{(q)}(x-y))^2\notag\\
					&+{2}e^{2/n{\mathbf{d}}}\left(1+\nu([y,x))\right)\E\left[\int_{{[y,x\wedge\tau_m)}}(W^{(q)}(x-z))^2(W_{{\mathbf{T}}(\nu_n)}^{(q)}(z,y))^2
					\nu
					(dz)\right]\Bigg],\notag\\
					&\leq {3}C_a^2(T,y)\Bigg[(W^{(q)}(T-y))^2\notag\\
					&+{2}e^{2/n{\mathbf{d}}}\left(1+\nu([y,T))\right)(W^{(q)}(T-y))^2\E\left[\int_{{[y,x\wedge\tau_m)}}(W_{{\mathbf{T}}(\nu_n)}^{(q)}(z,y))^2
					\nu
					(dz)\right]\Bigg]\notag\\
					&\leq {3}C_a^2(T,y)\Bigg[(W^{(q)}(T-y))^2\notag\\
					&+{2}e^{2/n{\mathbf{d}}}\left(1+\nu([y,T))\right)(W^{(q)}(T-y))^2\int_{{[y,x)}}\E\left[(W_{{\mathbf{T}}(\nu_n)}^{(q)}(z\wedge\tau_m,y))^2\right]
					\nu
					(dz)\Bigg],
				\end{align}
				where we recall that $\nu=\nu^a+\nu^d$, and in the second inequality we used the fact that the measure $\nu^d$ is diffuse. 
{Observe that, by the definition of the stopping time $\tau_m$, $W_{{\mathbf{T}}(\nu_n)}^{(q)}((x\wedge\tau_m)-,y)\leq m$ for $x\in[y,T]$. Consequently, in view of \eqref{jump_sf_2}, we obtain that for every
	$z\in[y,x]$, 
	 \begin{align*}\E\left[(W_{{\mathbf{T}}(\nu_n)}^{(q)}(z\wedge\tau_m,y))^2\right]\leq m^2\left(\sup_{z\in[y,x]}e^{\nu^a(\{z\})/{\mathbf{d}}}+e^{1/n{\mathbf{d}}}\right)^2<\infty.
	 \end{align*}}

			Then, applying Gronwall's inequality (see \cite[Appendix, Theorem 5.1]{EK}) to
			\eqref{bound_lemma_conv_1}, letting $m\uparrow\infty$, and using Lemma
			\ref{lemma_mon} together with the monotone convergence theorem, we obtain,
			for all $x\in[y,T]$,
				{
				\begin{align}\label{bound_gron}
					\E\Bigg[(&W_{{\mathbf{T}}(\nu_n)}^{(q)}(x,y))^2\Bigg]\notag\\
					&\leq  {3}C_a^2(T,y)(W^{(q)}(T-y))^2\exp\left({6C_a^2(T,y)e^{2/{\mathbf{d}}}}(W^{(q)}(T-y))^2\nu([y,T])(1+\nu([y,T]))\right)<\infty.
				\end{align}}
				(iii) In this step, we prove the result. Using \eqref{pred_volterra}, \eqref{pred_volterra_2} and the fact that $\nu^d$ is diffuse, we have for $x\in[y,T]$
\begin{align}
 W_{{\mathbf{T}}(\nu_n)}^{(q)}(x,y)&-W_{{\mathbf{T}}(\nu)}^{(q)}(x,y)\\
 &= K_a(x)\Bigg[\int_{{[y,x)}}W^{(q)}(x-z)(W_{{\mathbf{T}}(\nu_n)}^{(q)}(z,y)-W_{{\mathbf{T}}(\nu)}^{(q)}(z,y)){\mathbf{T}}(\nu^a)(dz)\\
 &+{\int_{[y,x]}W^{(q)}(x-z)W_{{\mathbf{T}}(\nu_n)}^{(q)}(z-,y){\mathbf{d}}(e^{1/n{\mathbf{d}}}-1)\tilde{N}_n(dz)}\\
 &+\int_{[y,x]}W^{(q)}(x-z)n{\mathbf{d}}(e^{1/n{\mathbf{d}}}-1)(W_{{\mathbf{T}}(\nu_n)}^{(q)}(z,y)-W_{{\mathbf{T}}(\nu)}^{(q)}(z,y))\nu^d(dz)\\
 &+\int_{[y,x]}W^{(q)}(x-z)\left(n{\mathbf{d}}(e^{1/n{\mathbf{d}}}-1)-1\right)W_{{\mathbf{T}}(\nu)}^{(q)}(z,y)\nu^d(dz)\Bigg],\quad \text{{$\mathbb{P}$-a.s.}}
\end{align}
{Using {\eqref{ine_T}}, the fact that the mapping $x\mapsto W^{(q)}(x)$ {is not decreasing} and Jensen's inequality, we get, for $x\in[y,T]$, $\mathbb{P}$-a.s., 
				\begin{align*}
					\sup_{y\leq u\leq x} \Big|W_{{\mathbf{T}}(\nu_n)}^{(q)}&(u,y)-W_{{\mathbf{T}}(\nu)}^{(q)}(u,y)\Big|^2\notag\\
					&\leq 4C_a^2(T,y)\Bigg[\sup_{y\leq u\leq x}\bigg(\int_{[y,u]}W^{(q)}(u-z)W_{{\mathbf{T}}(\nu_n)}^{(q)}(z-,y){\mathbf{d}}(e^{1/n{\mathbf{d}}}-1)\tilde{N}_n(dz)\bigg)^2
					\notag\\
					&+{\nu^a}[y, T]\left(W^{(q)}(T-y)\right)^2\int_{{[y,x)}}
					\left(W_{{\mathbf{T}}(\nu)}^{(q)}(z,y)-W_{{\mathbf{T}}(\nu_n)}^{(q)}(z,y)\right)^2\nu^a(dz)\notag\\
					&+{\nu^d}[y, T]\left(n{\mathbf{d}}(e^{1/n{\mathbf{d}}}-1)\right)^2\left(W^{(q)}(T-y)\right)^2\int_{[y,x]} \left(W_{{\mathbf{T}}(\nu)}^{(q)}(z,y)-W_{{\mathbf{T}}(\nu_n)}^{(q)}(z,y)\right)^2\nu^d(dz)\notag\\
					&+{\nu^d}[y, T]\left(W^{(q)}(T-y)\right)^2\left(n{\mathbf{d}}(e^{1/n{\mathbf{d}}}-1)-1\right)^2\int_{[y,x]}\left(W_{{\mathbf{T}}(\nu)}^{(q)}(z,y)\right)^2\nu^d(dz)\Bigg].
				\end{align*}
Now, taking expectations in the previous inequality and applying Doob's maximal inequality along with \eqref{square_integrable}, we obtain that for $x\in[y,T]$
				\begin{align*}
					\E&\Bigg[\sup_{y\leq u\leq x}\bigg|W_{{\mathbf{T}}(\nu_n)}^{(q)}(u,y)-W_{{\mathbf{T}}(\nu)}^{(q)}(u,y)\bigg|^2\Bigg]\\
	&				\leq4C_a^2(T,y)\Bigg[
	4\E\bigg[\bigg(\int_{[y,x]}W^{(q)}(T-z)W_{{\mathbf{T}}(\nu_n)}^{(q)}(z-,y){\mathbf{d}}(e^{1/n{\mathbf{d}}}-1)\tilde{N}_n(dz)\bigg)^2\bigg]
					\notag\\
					&+{\nu^a}[y, T]\left(W^{(q)}(T-y)\right)^2\int_{{[y,x)}}
					\E\left[\left(W_{{\mathbf{T}}(\nu)}^{(q)}(z,y)-W_{{\mathbf{T}}(\nu_n)}^{(q)}(z,y)\right)^2\right]\nu^a(dz)\notag\\
					&+{\nu^d}[y, T]\left(W^{(q)}(T-y)\right)^2e^{2/n{\mathbf{d}}}\left(n{\mathbf{d}}(1-e^{-1/n{\mathbf{d}}})\right)^2\int_{[y,x]} \E\left[\left(W_{{\mathbf{T}}(\nu)}^{(q)}(z,y)-W_{{\mathbf{T}}(\nu_n)}^{(q)}(z,y)\right)^2\right]\nu^d(dz)\notag\\
					&+{\nu^d}[y, T]\left(W^{(q)}(T-y)\right)^2\left(n{\mathbf{d}}(e^{1/n{\mathbf{d}}}-1)-1\right)^2\int_{[y,x]}\left(W_{{\mathbf{T}}(\nu)}^{(q)}(z,y)\right)^2\nu^d(dz)\Bigg]\notag\\
					&\leq4C_a^2(T,y)\left(W^{(q)}(T-y)\right)^2\Bigg[
	\frac{4}{n}e^{2/n{\mathbf{d}}}\left(n{\mathbf{d}}(1-e^{-1/n{\mathbf{d}}})\right)^2\E\left[\int_{[y,x]}\Big(W_{{\mathbf{T}}(\nu_n)}^{(q)}(z-,y)\Big)^2\nu^d(dz)\right]
					\notag\\
					&+{\nu^a}[y, T]\int_{{[y,x)}}
					\E\left[\left(W_{{\mathbf{T}}(\nu)}^{(q)}(z,y)-W_{{\mathbf{T}}(\nu_n)}^{(q)}(z,y)\right)^2\right]\nu^a(dz)\notag\\
					&+{\nu^d}[y, T]e^{2/n{\mathbf{d}}}\left(n{\mathbf{d}}(1-e^{-1/n{\mathbf{d}}})\right)^2\int_{[y,x]} \E\left[\left(W_{{\mathbf{T}}(\nu)}^{(q)}(z,y)-W_{{\mathbf{T}}(\nu_n)}^{(q)}(z,y)\right)^2\right]\nu^d(dz)\notag\\
					&+{\nu^d}[y, T]\left(n{\mathbf{d}}(e^{1/n{\mathbf{d}}}-1)-1\right)^2\int_{[y,x]}\left(W_{{\mathbf{T}}(\nu)}^{(q)}(z,y)\right)^2\nu^d(dz)\Bigg].\notag
				\end{align*}
				Using the fact that $(1-e^{-x})\leq x$ for {$x>0$}, gives for $x\in[y,T]$ 
								\begin{align*}
					\E\Bigg[\sup_{y\leq u\leq x}\bigg|W_{{\mathbf{T}}(\nu_n)}^{(q)}&(u,y)-W_{{\mathbf{T}}(\nu)}^{(q)}(u,y)\bigg|^2\Bigg]\\
	&				\leq4C_a^2(T,y)\left(W^{(q)}(T-y)\right)^2\Bigg[
	\frac{4}{n}e^{2/{\mathbf{d}}}\E\left[\int_{[y,T]}\Big(W_{{\mathbf{T}}(\nu_n)}^{(q)}(z-,y)\Big)^2\nu^d(dz)\right]
					\notag\\
					&+\nu[y, T]e^{2/{\mathbf{d}}}\int_{{[y,x)}} \E\left[\left(W_{{\mathbf{T}}(\nu)}^{(q)}(z,y)-W_{{\mathbf{T}}(\nu_n)}^{(q)}(z,y)\right)^2\right]\nu(dz)\notag\\
					&+{\nu^d}[y, T]\left(n{\mathbf{d}}(e^{1/n{\mathbf{d}}}-1)-1\right)^2\int_{[y,T]}\left(W_{{\mathbf{T}}(\nu)}^{(q)}(z,y)\right)^2\nu^d(dz)\Bigg]\notag\\
					&\leq4C_a^2(T,y)\left(W^{(q)}(T-y)\right)^2\Bigg[
	\frac{4}{n}C_1(T,y)
					\notag\\
					&+\nu[y, T]e^{2/{\mathbf{d}}}\int_{{[y,x)}}\E\left[\left(W_{{\mathbf{T}}(\nu)}^{(q)}(z,y)-W_{{\mathbf{T}}(\nu_n)}^{(q)}(z,y)\right)^2\right]\nu(dz)\notag\\
					&+{\nu^d}[y, T]\left(n{\mathbf{d}}(e^{1/n{\mathbf{d}}}-1)-1\right)^2\int_{[y,T]}\left(W_{{\mathbf{T}}(\nu)}^{(q)}(z,y)\right)^2\nu^d(dz)\Bigg],\notag
				\end{align*}
				where $C_1(T,y)$ is a finite constant, which is obtained using \eqref{bound_gron}.}

				Finally, an application of Gronwall's inequality (see \cite[Appendix, Theorem 5.1]{EK}), gives
				\begin{align}\label{gron_2}
					\E\Bigg[\sup_{y\leq u\leq x} &\left|W_{{\mathbf{T}}(\nu_n)}^{(q)}(u,y)-W_{{\mathbf{T}}(\nu)}^{(q)}(u,y)\right|^2\Bigg]\notag\\&\leq \left[\frac{{K_1(T,y)}}{n}+\left(n{\mathbf{d}}(1-e^{-1/n{\mathbf{d}}})-1\right)^2{K_2(T,y)}\right]e^{{{K_3}(T,y)}},
				\end{align}
				{where $K_1(T,y)$, $K_2(T,y)$, and $K_3(T,y)$ are finite constants.}
				
				Therefore, by taking $n\to\infty$ in \eqref{gron_2} we get
				\begin{align*}
					\lim_{n\to\infty}\E\left[\sup_{y\leq u\leq x} \left|W_{{\mathbf{T}}(\nu_n)}^{(q)}(u,y)-W_{{\mathbf{T}}(\nu)}^{(q)}(u,y)\right|^2\right]=0.
				\end{align*}
			\end{proof}
			{For each $n\geq1$, and $T\geq y$, define
			\begin{align}\label{volterra_limit_Z_2}
				Z^{(q)}_{{\mathbf{T}}(\nu_n)}(x,y):&=1+\int_{[y,x]}W^{(q)}_{{\mathbf{T}}(\nu_n)}(x,u){\mathbf{T}}(\nu_n)(du)+q\int_y^xW^{(q)}_{{\mathbf{T}}(\nu_n)}(x,u)du,\qquad x\in[y,T].
			\end{align}
			Here, $W^{(q)}_{{\mathbf{T}}(\nu_n)}$ denotes the unique solution to
			\eqref{volterra_approx}, and the operator ${\mathbf{T}}$ and the measures
			$\nu_n$ are defined in \eqref{oper_T} and \eqref{measure_nu_n}, respectively.
			
			Similarly, let
			\begin{align}\label{volterra_limit_Z}
				Z^{(q)}_{{\mathbf{T}}(\nu)}(x,c):&=1+\int_{[y,x]}W^{(q)}_{{\mathbf{T}}(\nu)}(x,u)\mathbf{T}(\nu)(du)+q\int_y^xW^{(q)}_{{\mathbf{T}}(\nu)}(x,u)du,\qquad x\in[y,T],
			\end{align}
			where {$W^{(q)}_{{\mathbf{T}}(\nu)}$} is the unique solution to \eqref{volterra_limit}.
			
			The following lemma establishes the convergence of the sequence of scale
			functions $\bigl(Z^{(q)}_{{\mathbf{T}}(\nu_n)}\bigr)_{n\geq 1}$ toward
			$Z^{(q)}_{{\mathbf{T}}(\nu)}$ as $n\to\infty$. By Lemma \ref{sf_Z_volterra}, the proof is entirely analogous to that of Lemma \ref{conv_sf}, and is therefore omitted.
			\begin{lemma}\label{conv_sf_Z}
				For $T>0$. Consider the sequence $\left(Z_{{\mathbf{T}}(\nu_n)}^{(q)}\right)_{n\geq 1}$, where for each $n\geq1$, $Z_{{\mathbf{T}}(\nu_n)}^{(q)}$ is given by \eqref{volterra_limit_Z_2}. Then,
				\[				
				\E\left[\sup_{x\in[y,T]}
				\left|
				Z_{{\mathbf{T}}(\nu_n)}^{(q)}(x,y)
				-
				Z_{{\mathbf{T}}(\nu)}^{(q)}(x,y)
				\right|^2
				\right]
				\longrightarrow 0,
				\qquad \text{as $n\to\infty$},
				\]
				where $Z^{(q)}_{{\mathbf{T}}(\nu)}$ is given by \eqref{volterra_limit_Z}. 
				
			\end{lemma}}
			\section{Proof of the main results}\label{sec:proof_main_results}
			In this section, we prove the results stated in Section \ref{sec:main_results}. The key idea is to apply Theorem~\ref{LiZhouProof_thm}, which provides the desired fluctuation identities in the case where the PcNAF is given by a finite mixture of local times, that is, when its Revuz measure is a finite atomic measure.
					
					By \eqref{representation_equation}, any PcNAF $A$ can be represented as a mixture of local times with respect to its Revuz measure $\nu_A$. We extend the fluctuation identities to the general setting by approximating $\nu_A$ with a sequence of random finite atomic measures constructed from Poisson random measures independent of the Lévy process $X$. Since each element of this sequence is almost surely a finite sum of Dirac masses, Theorem~\ref{LiZhouProof_thm} applies conditionally on the corresponding Poisson random measure.
					
					We then show that the resulting sequence of fluctuation identities converges to those associated with the original PcNAF $A$. At the same time, the corresponding generalized scale functions converge to the unique solution of the limiting integral equation driven by the Revuz measure $\nu_A$.
					
					More precisely, throughout this section, given a PcNAF $A$ with Revuz measure $\nu_A$, we consider the sequence of random measures $(\nu_n)_{n\geq1}$ defined by
					\[
					\nu_n(dz)=\nu_A^{a}(dz)+\frac{1}{n}N_n(dz), \qquad n\geq1,
					\]
					where $\nu_A^{a}$ and $\nu_A^{d}$ denote the atomic and diffuse parts of $\nu_A$, respectively, and $N_n$ is a Poisson random measure with intensity $n\nu_A^{d}$, independent of $X$.
					
					We first establish an auxiliary convergence result for the joint Laplace transforms of the first passage times and the associated additive functionals under the approximating sequence $(\nu_n)_{n\geq1}$. This result is the key step in extending the fluctuation identities to the general setting.
				\begin{lemma} \label{new_convergence_lemma}
					For $q\geq0$, and $x\in(c,b)$,
					\begin{align}\label{new_conv_measure_0}
						\lim_{n\to\infty}\E_x\bigg[e^{-q\tau_b^+-\int_{\R}L^y_{\tau_b^+}\nu_{n}(dy)};\tau_b^+<\tau_c^-\bigg]=\E_x\bigg[e^{-q\tau_b^+-\int_{\R}L^y_{\tau_b^+}{\nu_A}(dy)};\tau_b^+<\tau_c^-\bigg].
					\end{align}
					\begin{align}\label{new_conv_measure_1}
						\lim_{n\to\infty}\E_x\bigg[e^{-q\tau_c^--\int_{\R}L^y_{\tau_c^-}\nu_{n}(dy)};\tau_c^-<\tau_b^+\bigg]=\E_x\bigg[e^{-q\tau_c^--\int_{\R}L^y_{\tau_c^-}{\nu_A}(dy)};\tau_c^-<\tau_b^+\bigg].
					\end{align}
					Additionally, for any bounded measurable function $f:\R\to\R$, we have
					\begin{align}\label{new_conv_measure_2}
						\lim_{n\to\infty}\E_x\bigg[e^{-\int_{\R}L^y_{t}\nu_{n}(dy)}f(X_t)1_{\{
							t<\tau_c^-\wedge\tau_b^+\}}\bigg]=\E_x\bigg[e^{-\int_{\R}L^y_{t}{\nu_A}(dy)}f(X_t)1_{\{
							t<\tau_c^-\wedge\tau_b^+\}}\bigg].
					\end{align}
				\end{lemma}
				\begin{proof}
					(i) We only prove \eqref{new_conv_measure_0}, since the proof of
					\eqref{new_conv_measure_1} is analogous. To this end, note first that
					\begin{align}\label{new_conv_measure}
						\Bigg|\E_x&\bigg[e^{-q\tau_b^+-\int_{\R}L^{y}_{\tau_b^+}\nu^a_A(dy)-\int_{\R}L^{y}_{\tau_b^+}\frac{1}{n}N_{n}(dy)};\tau_b^+<\tau_c^-\bigg]-\E_x\bigg[e^{-q{\tau_b^+}-\int_{\R}L^{y}_{\tau_b^+}\nu^a_A(dy)-\int_{\R}L^{y}_{\tau_b^+}\nu_A^d(dy)};\tau_b^+<\tau_c^-\bigg]\Bigg|\notag\\
						&=\Bigg|\E_x\bigg[e^{-q{\tau_b^+}-\int_{\R}L^{y}_{\tau_b^+}\nu^a_A(dy)}
						\bigg{(}\E\bigg{[}e^{-\int_{\R}g(y)\frac{1}{n}N_{n}(dy)}\bigg{]}\bigg{|}_{g(y)=L^{y}_{\tau_b^+}}-e^{-\int_{\R}L^{y}_{\tau_b^+}\nu_A^d(dy)} \bigg{)};\tau_b^+<\tau_c^-\bigg]\Bigg|\notag\\
						&=\Bigg|\E_x\bigg[e^{-q{\tau_b^+}-\int_{\R}L^{y}_{\tau_b^+}\nu^a_A(dy)}
						\bigg{(}\exp\bigg{\{}-\int_{\R}\bigg{(}1-e^{-\frac{1}{n}L^{y}_{\tau_b^+}}\bigg{)}n \nu_A^d(dy) \bigg{\}}-e^{-\int_{\R}L^{y}_{\tau_b^+}\nu_A^d(dy)} \bigg{)};\tau_b^+<\tau_c^-\bigg]\Bigg|.
					\end{align}	
					In the last equality, we used the Laplace functional of a Poisson random
					measure (see \cite[Chapter VI, Theorem 2.9]{Cinlar}), together with the
					fact that $N_n$ is independent of $X$.
					{On the other hand, note that 
						\begin{align*}
							\int_{\R}\bigg{(}1-e^{-\frac{1}{n}L^{y}_{\tau_b^+}}\bigg{)}n \nu_A^d(dy)\leq \int_{\R}L^{y}_{\tau_b^+}\nu_A^d(dy)<\infty,\qquad \text{$\mathbb{P}$-a.s.}
						\end{align*}
						Thus, by the dominated convergence theorem,
						\begin{align*}
							\lim_{n\to\infty}\exp\bigg{(}-\int_{\R}\bigg{(}1-e^{-\frac{1}{n}L^{y}_{\tau_b^+}}\bigg{)}n \nu_A^d(dy) \bigg{)}
							=\exp\bigg{(}-\int_{\R}L^{y}_{\tau_b^+}\nu_A^d(dy) \bigg{)},\qquad \text{$\mathbb{P}$-a.s.}
					\end{align*}}
					Therefore, letting $n\to\infty$ in \eqref{new_conv_measure} and applying dominated convergence once more yields \eqref{new_conv_measure_0}.
					
					(ii) We now prove \eqref{new_conv_measure_2}. Observe that
					\begin{align}\label{new_conv_measure_4}
						\Bigg|\E_x&\bigg[e^{-\int_{\R}L^{y}_{t}\nu^a_A(dy)-\int_{\R}L^{y}_{t}\frac{1}{n}N_{n}(dy)}f(X_t)1_{\{t<\tau_c^-\land\tau_b^+\}}\bigg]-\E_x\bigg[e^{-\int_{\R}L^{y}_{t}\nu^a_A(dy)-\int_{\R}L^{y}_{t}{\nu_A^d}(dy)}f(X_t)1_{\{t<\tau_c^-\land\tau_b^+\}}\bigg]\Bigg|\notag\\
						&=\Bigg|\E_x\bigg[e^{-\int_{\R}L^{y}_{t}\nu^a_A(dy)}
						\bigg{(}\E\bigg{[}e^{-\int_{\R}g(y)\frac{1}{n}N_{n}(dy)}\bigg{]}\bigg{|}_{g(y)=L^{y}_{t}}-e^{-\int_{\R}L^{y}_{t}{\nu_A^d}(dy)} \bigg{)}f(X_t)1_{\{t<\tau_c^-\land\tau_b^+\}}\bigg]\Bigg|\notag\\
						&=\Bigg|\E_x\bigg[e^{-\int_{\R}L^{y}_{t}\nu^a_A(dy)}
						\bigg{(}\exp\bigg{\{}-\int_{\R}\bigg{(}1-e^{-\frac{1}{n}L^{y}_t}\bigg{)}n \nu_A^d(dy) \bigg{\}}-e^{-\int_{\R}L^{y}_{t}{\nu_A^d}(dy)} \bigg{)}f(X_t)1_{\{t<\tau_c^-\land\tau_b^+\}}\bigg]\Bigg|
						\notag\\
						&\leq\|f\|_{\infty}\E_x\bigg[e^{-\int_{\R}L^{y}_{t}\nu^a_A(dy)}
						\bigg{|}\exp\bigg{\{}-\int_{\R}\bigg{(}1-e^{-\frac{1}{n}L^{y}_t}\bigg{)}n \nu_A^d(dy) \bigg{\}}-e^{-\int_{\R}L^{y}_{t}{\nu_A^d}(dy)} \bigg{|}1_{\{t<\tau_c^-\land\tau_b^+\}}\bigg].
					\end{align}
					{In the second equality, we used Theorem~2.9 in Chapter~VI of \cite{Cinlar}, together with the fact that
					the Poisson random measure $N_n$ is independent of $X$.
					
					Proceeding as in step~(i), we may let $n\to\infty$ in \eqref{new_conv_measure_4}.
					An application of the dominated convergence theorem then yields \eqref{new_conv_measure_2}.}
			\end{proof}	
			\subsection{Proof of Theorem \ref{main_sf}}
			(i) {We begin by extending identity \eqref{local_time_exit} of Theorem~\ref{LiZhouProof_thm} to the case of a (possibly infinite) mixture of local times at different levels. 
			To this end,} let $\mu$ be an atomic Radon measure of the form $\mu(dz)=\sum_{i=1}^{\infty}q_i\delta_{b_i}(dz)$ with $q_i>0$ and $b_i\in\R$ for every $i\in\mathbb{N}$. By Remark \ref{exi_uni_T}, there exists a unique solution to the following integral equation:
			\begin{align}\label{volterra_fin}
				W_{{\mathbf{T}}(\mu)}^{(q)}(x,c)=W^{(q)}(x-c)+\int_{[c,x]}W^{(q)}(x-z)W_{{\mathbf{T}}(\mu)}^{(q)}(z,c){\mathbf{T}}(\mu)(dz),\qquad x\in[c,b],
			\end{align}
			Moreover, if we define the measure $\mu_n(dz)=\sum_{i=1}^{n}q_i\delta_{b_i}(dz)$ for $n\in\mathbb{N}$, then Remark \ref{exi_uni_T}, implies that there exists a unique solution to:
			\begin{align}\label{volterra_fin_approx}
				W_{{\mathbf{T}}(\mu_n)}^{(q)}(x,c)=W^{(q)}(x-c)+\int_{[c,x]}W^{(q)}(x-z)W_{{\mathbf{T}}(\mu_n)}^{(q)}(z,c){\mathbf{T}}(\mu_n)(dz),\qquad x\in[c,b],
			\end{align}
			Since $W^{(q)}(u)\geq0$ for all $u\in\R$, and since ${\mathbf{T}}(\mu)$ and ${\mathbf{T}}(\mu_n)$ are non-negative measures, it follows from Proposition \ref{exist_uni} that $W_{{\mathbf{T}}(\mu)}^{(q)}(\cdot,y)$ and $W_{{\mathbf{T}}(\mu_n)}^{(q)}(x,y)$ are non-negative on $[c,b]$. Therefore, using Lemma \ref{equiv_lemma}, we obtain that for $x\in[c,b]$:
			\begin{align}\label{approx_atom}
			W_{{\mathbf{T}}(\mu)}^{(q)}(x,c)-W_{{\mathbf{T}}(\mu_n)}^{(q)}(x,c)=\int_{[c,x]}W_{{\mathbf{T}}(\mu)}^{(q)}(x,z)W_{{\mathbf{T}}(\mu_n)}^{(q)}(z,c)({\mathbf{T}}(\mu)-{\mathbf{T}}(\mu_n))(dz)\geq0.
			\end{align}
			{By Lemma \ref{lemma_mon}, the mapping $x\mapsto W_{{\mathbf{T}}(\mu)}^{(q)}(x,c)$ is non-decreasing}
			, and thus $W_{{\mathbf{T}}(\mu)}^{(q)}(\cdot,c)$ is bounded on $[c,b]$. Thus, using \eqref{ine_T} and \eqref{approx_atom}, we conclude that for $x\in[c,b]$:
			\begin{align*}
			W_{{\mathbf{T}}(\mu)}^{(q)}(x,c)-W_{{\mathbf{T}}(\mu_n)}^{(q)}(x,c)&\leq (W_{{\mathbf{T}}(\mu)}^{(q)}(b,c))^2\sum_{i=n+1}^{\infty}{\mathbf{d}} (1-e^{-q_i/{\mathbf{d}}})1_{\{b_i\in[c,b]\}}\notag\\
			&\leq (W_{{\mathbf{T}}(\mu)}^{(q)}(b,c))^2\sum_{i=n+1}^{\infty}q_i1_{\{b_i\in[c,b]\}}.
			\end{align*}
			Since $\mu[c,b]<\infty$, we have 
			\[
			\sum_{i=n+1}^{\infty}q_i1_{\{b_i\in[c,b]\}}\downarrow0, \qquad \text{as $n\to\infty$}.
			\]
			Thus, 
			\begin{align}\label{lim_atomic}
			{W_{{\mathbf{T}}(\mu)}^{(q)}(x,c)-W_{{\mathbf{T}}(\mu_n)}^{(q)}(x,c)\to0},\qquad \text{as $n\to\infty$, uniformly in $[c,b]$}.
			\end{align}
			Note that 
			\[
			\int_{\R}L_{\tau_b^+}^y\mu_n(dy)={\sum_{i=1}^{n}L_{\tau_b^+}^{b_i}q_i}.
			\]
			{Then, }
			\begin{align}\label{lt_conv}
			\int_{\R}L_{\tau_b^+}^y\mu_n(dy)\to\int_{\R}L_{\tau_b^+}^y\mu(dy),\qquad\text{{as $n\to\infty$}}.
			\end{align}
			On the other hand, by Theorem~\ref{LiZhouProof_thm}, Lemma~\ref{sf_lt_lemma}, and Remark~\ref{remark_revuz_measure}, we obtain, for all $x\in[c,b]$,
			\begin{align*}
				\E_x\left[e^{-q\tau_b^+-\int_{\R}L^y_{\tau_b^+}\mu_n(dy)};\tau_b^+<\tau_c^-\right]=\frac{W_{{\mathbf{T}}(\mu_n)}^{(q)}(x,c)}{W_{{\mathbf{T}}(\mu_n)}^{(q)}(b,c)}.
			\end{align*}
			Therefore, using dominated convergence, \eqref{lim_atomic}, and \eqref{lt_conv}, we obtain that for $x\in[c,b]$:
			\begin{align}\label{sf_point_m}
				\E_x\left[e^{-q\tau_b^+-\int_{\R}L^y_{\tau_b^+}\mu(dy)};\tau_b^+<\tau_c^-\right]=\frac{W_{{\mathbf{T}}(\mu)}^{(q)}(x,c)}{W_{{\mathbf{T}}(\mu)}^{(q)}(b,c)}.
			\end{align}
				(ii) {{By Proposition~\ref{PropD2}, the Revuz measure of $A$, $\nu_A$, is a Radon measure.} Consider the decomposition of $\nu_A$ in its atomic and diffuse parts, denoted by $\nu_A^a$ and $\nu_A^d$, respectively. {In addition, recall that $N_n(dz)$ is a Poisson random measure  with intensity measure $n\nu_A^d(dz)$, independent of the process $X$, and that}}
				\[
				\nu_n(dz)={\nu^a_A}(dz)+\frac{1}{n}N_n(dz).
				\]
				{Since the measure $n\nu_A^d$ is diffuse, it follows from \cite[Theorem VI.2.17]{Cinlar}
					that the Poisson random measure $N_n$ is almost surely a random counting measure.
					Therefore, conditioning on $N_n$ and applying \eqref{sf_point_m} (which is justified
					because $\nu_n$ is almost surely
					a counting measure), we obtain}
				\begin{align}\label{lim_2}
					\E_x\bigg[&e^{-q\tau_b^+-\int_{\R}L^y_{\tau_b^+}\nu_n(dy)};\tau_b^+<\tau_c^-\bigg]\notag\\&=\E_x\left[\bE\left[e^{-q\tau_b^+-\int_{\R}L^y_{\tau_b^+}\nu^a_A(dy)-\int_{\R}L^y_{\tau_b^+}\frac{1}{n}N_n(dy)};\tau_b^+<\tau_c^-\Big|N_n\right]\right]=\E_x\left[\frac{W_{{\mathbf{T}}(\nu_n)}^{(q)}(x,c)}{W_{{\mathbf{T}}(\nu_n)}^{(q)}(b,c)}\right],
				\end{align}
				where $W_{{\mathbf{T}}(\nu_n)}^{(q)}$ is, almost surely, the unique solution to \eqref{volterra_approx}.

				By Lemma \ref{conv_sf}, we may extract a subsequence, which we denote by
				$(W_{{\mathbf{T}}(\nu_{n_k})}^{(q)}(\cdot,c))_{k\geq 1}$, that converges
				uniformly on compact subsets of $[c,T]$ almost surely to
				$W_{{\mathbf{T}}(\nu_A)}^{(q)}(\cdot,c)$ as $k\to\infty$.
				Here, $W_{{\mathbf{T}}(\nu_A)}^{(q)}(\cdot,c)$ is the unique solution to
				\eqref{volterra_sf_w}.
				Moreover, by Lemma \ref{lemma_mon}, we have $\frac{W_{{\mathbf{T}}(\nu_{n_k})}^{(q)}(x, c)}{W_{{\mathbf{T}}(\nu_{n_k})}^{(q)}(b, c)} \leq 1$ for $x \leq b$. Therefore, {using} \eqref{lim_2} and applying
				Lemma \ref{new_convergence_lemma} together with the dominated convergence theorem
				yields
				\begin{align*}
					\E_x\left[e^{-q\tau_b^+-\int_{\R}L^y_{\tau_b^+}\nu_A(dy)};\tau_b^+<\tau_c^-\right]&=
					{\lim_{k\to\infty}\E_x\left[e^{-q\tau_b^+-\int_{\R}L^y_{\tau_b^+}\nu_{n_k}(dy)};\tau_b^+<\tau_c^-\right]}\notag\\&=\lim_{k\to\infty}\E_x\left[\frac{W_{{\mathbf{T}}(\nu_{n_k})}^{(q)}(x,c)}{W_{{\mathbf{T}}(\nu_{n_k})}^{(q)}(b,c)}\right]\notag\\
					&=\frac{W_{{\mathbf{T}}(\nu_A)}^{(q)}(x,c)}{W_{{\mathbf{T}}(\nu_A)}^{(q)}(b,c)},
				\end{align*}
				{which proves the claim.} 
				\qedsymbol 
				\subsection{Proof of Theorem \ref{two_sided_down_thm}}
				(i) {We begin by extending identity \eqref{tsd_u_0} in Theorem~\ref{LiZhouProof_thm} to the case of a (possibly infinite) mixture of local times at different levels. 
				Let $\mu$ and $(\mu_n)_{n\geq1}$ be the measures introduced in part (i) of the proof of Theorem~\ref{main_sf}.}
				
				For $x\in[c,b]$, let
				\begin{align*}
				Z^{(q)}_{{\mathbf{T}}(\mu_n)}(x,c):&=1+\int_{[c,x]}W^{(q)}_{{\mathbf{T}}(\mu_n)}(x,y){\mathbf{T}}(\mu_n)(dy)+q\int_c^xW^{(q)}_{{\mathbf{T}}(\mu_n)}(x,y)dy,\qquad x\in[c,b].
				\end{align*}
				Using \eqref{lim_atomic} together with {dominated} convergence, we obtain, for $x\in [c,b]$,
				\begin{align}\label{lim_atom_Z}
				\lim_{n\to\infty}Z^{(q)}_{{\mathbf{T}}(\mu_n)}(x,c)&=1+\lim_{n\to\infty}\sum_{i=1}^{\infty}\mathbf{d}(1-e^{-q_i/\mathbf{d}})W^{(q)}_{{\mathbf{T}}(\mu_n)}(x,b_i)1_{\{b_i\in[c,b]\}}1_{\{i\leq n\}}+\lim_{n\to\infty}q\int_c^xW^{(q)}_{{\mathbf{T}}(\mu_n)}(x,y)dy\notag\\
				&=1+\sum_{i=1}^{\infty}\mathbf{d}(1-e^{-q_i/\mathbf{d}})W^{(q)}_{{\mathbf{T}}(\mu)}(x,b_i)1_{\{b_i\in[c,b]\}}+q\int_c^xW^{(q)}_{{\mathbf{T}}(\mu)}(x,y)dy\notag\\
				&=1+\int_{[c,x]}W^{(q)}_{{\mathbf{T}}(\mu)}(x,y){\mathbf{T}}(\mu)(dy)+q\int_c^xW^{(q)}_{{\mathbf{T}}(\mu)}(x,y)dy\notag\\
				&=Z^{(q)}_{{\mathbf{T}}(\mu)}(x,c).
				\end{align}
				
				By Lemma \ref{sf_Z_volterra}, for each $n\geq 1$, the function $Z^{(q)}_{{\mathbf{T}}(\mu_n)}(\cdot,c)$ is the unique solution to the following integral equation:
					\begin{align}\label{volterra_fin_Z}
						Z_{{\mathbf{T}}(\mu_n)}^{(q)}{(x,c)}={Z^{(q)}(x-c)}+\int_{[c,x]}W^{(q)}(x-z)Z_{{\mathbf{T}}(\mu_n)}^{(q)}{(z,c)}{\mathbf{T}}(\mu_n)({dz}),\qquad x\in[c,b].
					\end{align}
				Therefore, by Theorem~\ref{LiZhouProof_thm}, Lemma~\ref{sf_lt_lemma}, and Remark~\ref{remark_revuz_measure}, we obtain, for $x\in[c,b]$,
				\begin{align}\label{fluc_iden_dis_lt_Z}
					\E_x\left[e^{-q\tau_c^--\int_{\R}L^y_{\tau_c^-}\mu_n(dy)};\tau_c^-<\tau_b^+\right]=Z^{(q)}_{{\mathbf{T}}(\mu_n)}(x,c)-\frac{Z^{(q)}_{{\mathbf{T}}(\mu_n)}(b,c)}{W^{(q)}_{{\mathbf{T}}(\mu_n)}(b,c)}W^{(q)}_{{\mathbf{T}}(\mu_n)}(x,c).
				\end{align}
				By a similar argument to that leading to identity \eqref{lt_conv}, 
				we have 
				\begin{align}\label{lt_conv_Z}
					\int_{\R}L_{\tau_c^-}^y\mu_n(dy)\to\int_{\R}L_{\tau_c^-}^y\mu(dy)\qquad\text{{as $n\to\infty$}}.
				\end{align}
				Hence, combining \eqref{lim_atomic}, \eqref{lim_atom_Z}, and \eqref{lt_conv_Z} in \eqref{fluc_iden_dis_lt_Z} and applying the dominated convergence theorem, we obtain for $x\in[c,b]$ that
					\begin{align}\label{fluc_iden_dis_lt_Z_2}
						\E_x\left[e^{-q\tau_c^--\int_{\R}L^y_{\tau_c^-}\mu(dy)};\tau_c^-<\tau_b^+\right]=Z^{(q)}_{{\mathbf{T}}(\mu)}(x,c)-\frac{Z^{(q)}_{{\mathbf{T}}(\mu)}(b,c)}{W^{(q)}_{{\mathbf{T}}(\mu)}(b,c)}W^{(q)}_{{\mathbf{T}}(\mu)}(x,c).
					\end{align}

					(ii) 
					%
					Recall that $\nu_A$ denotes the Revuz measure of the PcNAF $A$, which, by Proposition \ref{PropD2}, is a Radon measure. 
						We denote by $(\nu_n)_{n\ge1}$ the sequence of random measures introduced in part (ii) of the proof of Theorem~\ref{main_sf}.
						
						Using \eqref{fluc_iden_dis_lt_Z_2} and following the same line of reasoning that leads to \eqref{lim_2}, we obtain
					\begin{align}\label{lim_3}
						\E_x\left[e^{-q\tau_c^--\int_{\R}L^y_{\tau_c^-}\nu_n(dy)};\tau_c^-<\tau_b^+\right]&=\E_x\left[\E\left[e^{-q\tau_c^--\int_{\R}L^y_{\tau_c^-}\nu_A^a(dy)-\int_{\R}L^y_{\tau_c^-}\frac{1}{n}N_n(dy)};\tau_c^-<\tau_b^+\Big|{N_n}\right]\right]\notag\\
						&=\E_x\left[Z^{(q)}_{{\mathbf{T}}(\nu_{n})}(x,c)-\frac{W_{{\mathbf{T}}(\nu_n)}^{(q)}(x,c)}{W_{{\mathbf{T}}(\nu_n)}^{(q)}(b,c)}Z^{(q)}_{{\mathbf{T}}(\nu_{n})}(b,c)\right],
					\end{align}
					{where $Z^{(q)}_{{\mathbf{T}}(\nu_{n})}$ is given by \eqref{volterra_limit_Z_2}}.
				
By Lemma~\ref{conv_sf_Z}, we have
					\begin{align}\label{lim_Z_L_2}
						\lim_{k\to\infty}\E_x\left[Z^{(q)}_{{\mathbf{T}}(\nu_{n_k})}(x,c)\right]=Z^{(q)}_{{\mathbf{T}}(\nu_A)}(x,c),\qquad x\geq c.
					\end{align} 
					{Moreover, by Lemmas~\ref{conv_sf} and \ref{conv_sf_Z}, we may extract subsequences
						$\bigl(Z^{(q)}_{{\mathbf{T}}(\nu_{n_k})}(x,c)\bigr)_{k\geq 1}$ and
						$\bigl(W^{(q)}_{{\mathbf{T}}(\nu_{n_k})}(x,c)\bigr)_{k\geq 1}$
						which converge uniformly on compact subsets of $[c,T]$ almost surely, as $k\to\infty$, to
						$Z^{(q)}_{{\mathbf{T}}(\nu_A)}(x,c)$ and
						$W^{(q)}_{{\mathbf{T}}(\nu_A)}(x,c)$, respectively.
					In addition, since $\frac{W_{{\mathbf{T}}(\nu_{n_k})}^{(q)}(x,c)}{W_{{\mathbf{T}}(\nu_{n_k})}^{(q)}(b,c)}\leq 1$ we may apply the generalized dominated convergence theorem together with
					\eqref{lim_Z_L_2} to obtain}
					\begin{align}\label{lim_Z_L_3}
						\lim_{k\to\infty}\E_x\left[\frac{W_{{\mathbf{T}}(\nu_{n_k})}^{(q)}(x,c)}{W_{{\mathbf{T}}(\nu_{n_k})}^{(q)}(b,c)}Z^{(q)}_{{\mathbf{T}}(\nu_{n_k})}(b,c)\right]=\frac{W_{{\mathbf{T}}(\nu_A)}^{(q)}(x,c)}{W_{{\mathbf{T}}(\nu_A)}^{(q)}(b,c)}Z^{(q)}_{{\mathbf{T}}(\nu_A)}(b,c),\qquad c\leq x\leq b.
					\end{align}
						{Therefore, by \eqref{lim_3}, Lemma~\ref{new_convergence_lemma}, and \eqref{lim_Z_L_2}--\eqref{lim_Z_L_3}, we obtain}  
					\begin{align*}
						\E_x\left[e^{-q\tau_c^--\int_{\R}L^y_{\tau_c^-}\nu_{A}(dy)};\tau_c^-<\tau_b^+\right]
						&=\lim_{k\to\infty}\E_x\bigg[e^{-q\tau_c^--\int_{\R}L^y_{\tau_c^-}\nu_{n_k}(dy)};\tau_c^-<\tau_b^+\bigg]\notag\\&=\lim_{k\to\infty}\E_x\left[Z^{(q)}_{{\mathbf{T}}(\nu_{n_k})}(x,c)-\frac{W_{{\mathbf{T}}(\nu_{n_k})}^{(q)}(x,c)}{W_{{\mathbf{T}}(\nu_{n_k})}^{(q)}(b,c)}Z^{(q)}_{{\mathbf{T}}(\nu_{n_k})}(b,c)\right]\notag\\
						&=Z^{(q)}_{{\mathbf{T}}(\nu_A)}(x,c)-\frac{W_{{\mathbf{T}}(\nu_A)}^{(q)}(x,c)}{W_{{\mathbf{T}}(\nu_A)}^{(q)}(b,c)}Z^{(q)}_{{\mathbf{T}}(\nu_A)}(b,c).
					\end{align*}
					The fact that $Z^{(q)}_{{\mathbf{T}}(\nu_A)}$ is the unique solution to \eqref{volterra_sf_Z_new} follows from Lemma \ref{sf_Z_volterra}.
					\qedsymbol 
					\subsection{Proof of Theorem \ref{resol_ts}}
					(i) {As in the proof of Theorem~\ref{main_sf}, we begin by extending identity \eqref{local_time_resolvent} of
					Theorem~\ref{LiZhouProof_thm} to the case of a (possibly infinite) mixture of
					local times at different levels. {Let $\mu$ and $(\mu_n)_{n\geq1}$ be the measures introduced in part (i) of the proof of Theorem~\ref{main_sf}.} 
					
					Then, by Theorem~\ref{LiZhouProof_thm}, Lemma~\ref{sf_lt_lemma}, and
					Remark~\ref{remark_revuz_measure}, we obtain that for $x\in(c,b)$ and any bounded,
					 measurable function $f:\R\to\R$, 
					\begin{align*}
						\E_x\left[\int_0^{\tau_b^+\wedge\tau_c^-}e^{-qt-\int_{\R}L^y_t\mu_n(dy)}f(X_t)dt\right]=\int_{c}^bf(y)\left[\frac{W_{{\mathbf{T}}(\mu_n)}^{(q)}(x,c)}{W_{{\mathbf{T}}(\mu_n)}^{(q)}(b,c)}W_{{\mathbf{T}}(\mu_n)}^{(q)}(b,y)-W_{{\mathbf{T}}(\mu_n)}^{(q)}(x,y)\right]dy,
					\end{align*}
					where $W_{{\mathbf{T}}(\mu_n)}^{(q)}$ is given as the unique solution to \eqref{volterra_fin_approx}.
					
					Next, using the dominated convergence theorem together with
					\eqref{lim_atomic} and \eqref{lt_conv}, we may pass to the limit as $n\to\infty$
					to obtain, for $x\in[c,b]$,{
					\begin{align}\label{fluc_iden_resol_infinite}
						\E_x\left[\int_0^{\tau_b^+\wedge\tau_c^-}e^{-qt-\int_{\R}L^y_t\mu(dy)}f(X_t)dt\right]=\int_{c}^bf(y)\left[\frac{W_{{\mathbf{T}}(\mu)}^{(q)}(x,c)}{W_{{\mathbf{T}}(\mu)}^{(q)}(b,c)}W_{{\mathbf{T}}(\mu)}^{(q)}(b,y)-W_{{\mathbf{T}}(\mu)}^{(q)}(x,y)\right]dy,
					\end{align}}
					
							(ii) {By Proposition \ref{PropD2}, the measure $\nu_A$ is Radon.} 
							%
							{We denote by $(\nu_n)_{n\ge1}$ the sequence of random measures introduced in part (ii) of the proof of Theorem~\ref{main_sf}.}
							
							{Using \eqref{fluc_iden_resol_infinite} and following the same line of reasoning that leads to \eqref{lim_2}, we obtain}
							\begin{align}\label{lim_0_resol}
								\E_x\Bigg[\int_0^{\tau_b^+\wedge\tau_c^-}&e^{-qt-\int_{\R}L^y_t\nu_n(dy)}f(X_t)dt\Bigg]\notag\\
&=\E_x\left[\E\left[\int_0^{\tau_b^+\wedge\tau_c^-}e^{-qt-\int_{\R}L^y_t\nu_n(dy)}f(X_t) dt \bigg{|}N_n\right]\right]\notag\\
								&={\E_x\left[\int_{c}^bf(y)\left(\frac{W_{{\mathbf{T}}(\nu_n)}^{(q)}(x,c)}{W_{{\mathbf{T}}(\nu_n)}^{(q)}(b,c)}W_{{\mathbf{T}}(\nu_n)}^{(q)}(b,y)-W_{{\mathbf{T}}(\nu_n)}^{(q)}(x,y)\right)dy\right]}\notag\\
								&=\int_{c}^bf(y)\E_x\left[\frac{W_{{\mathbf{T}}(\nu_n)}^{(q)}(x,c)}{W_{{\mathbf{T}}(\nu_n)}^{(q)}(b,c)}W_{{\mathbf{T}}(\nu_n)}^{(q)}(b,y)-W_{{\mathbf{T}}(\nu_n)}^{(q)}(x,y)\right]dy,
							\end{align}}
							where for each $n\geq1$, $W_{{\mathbf{T}}(\nu_n)}^{(q)}$ is almost surely the unique solution to \eqref{volterra_approx}.
						{Proceeding as in the proof of Theorem \ref{main_sf} and using Lemma \ref{conv_sf}, we can extract a subsequence  $(W_{{\mathbf{T}}(\nu_{n_k})}^{(q)}(\cdot, c))_{k\geq 1}$ such that
							\begin{align*}
								\lim_{k\to\infty}\E_x\left[W^{(q)}_{{\mathbf{T}}(\nu_{n_k})}(x,c)\right]=W^{(q)}_{{\mathbf{T}}(\nu_A)}(x,c),\qquad x\geq c,
							\end{align*}
							and
							\begin{align*}
								\lim_{k\to\infty}\E_x\left[\frac{W_{{\mathbf{T}}(\nu_{n_k})}^{(q)}(x,c)}{W_{{\mathbf{T}}(\nu_{n_k})}^{(q)}(b,c)}W^{(q)}_{{\mathbf{T}}(\nu_{n_k})}(b,y)\right]=\frac{W_{{\mathbf{T}}(\nu_A)}^{(q)}(x,c)}{W_{{\mathbf{T}}(\nu_A)}^{(q)}(b,c)}W^{(q)}_{{\mathbf{T}}(\nu_A)}(b,y),\qquad c\leq x\leq b.
							\end{align*}}
							Hence, for any $x,y\in(c,b)$,
							\begin{align}\label{lim_2_resol}
								\lim_{k\to\infty}\E\Bigg[\frac{W_{{\mathbf{T}}(\nu_{n_k})}^{(q)}(x,c)}{W_{{\mathbf{T}}(\nu_{n_k})}^{(q)}(b,c)}W_{{\mathbf{T}}(\nu_{n_k})}^{(q)}(b,y)&-W_{{\mathbf{T}}(\nu_{n_k})}^{(q)}(x,y)\Bigg]\notag\\&=\frac{W_{{\mathbf{T}}(\nu_A)}^{(q)}(x,c)}{W_{{\mathbf{T}}(\nu_A)}^{(q)}(b,c)}W_{{\mathbf{T}}(\nu_A)}^{(q)}(b,y)-W_{{\mathbf{T}}(\nu_A)}^{(q)}(x,y).
							\end{align}
							{Therefore, by \eqref{lim_0_resol}, Lemma~\ref{new_convergence_lemma}, and \eqref{lim_2_resol}, we obtain for any bounded measurable function $f$}
							{\begin{align*}
								\E_x\Bigg[\int_0^{\tau_b^+\wedge\tau_c^-}&e^{-qt-\int_{\R}L^y_t\nu_A(dy)}f(X_t)dt\Bigg]\notag\\
								&=\int_0^{\infty}e^{-qt}\E_x\left[e^{-\int_{\R}L^y_t\nu_A(dy)}f(X_t)1_{\{t<\tau_b^+\wedge\tau_c^-\}}\right]dt\notag\\
								&=\int_0^{\infty}e^{-qt}\lim_{k\to\infty}\E_x\left[e^{-\int_{\R}L^y_t\nu_{n_k}(dy)}f(X_t)1_{\{t<\tau_b^+\wedge\tau_c^-\}}\right]dt\notag\\
								&=\lim_{k\to\infty}\E_x\left[\int_0^{\tau_b^+\wedge\tau_c^-}e^{-qt-\int_{\R}L^y_t\nu_{n_k}(dy)}f(X_t)dt\right]\notag\\
								&=\lim_{k\to\infty}\int_{c}^bf(y)\E\left[\frac{W_{{\mathbf{T}}(\nu_{n_k})}^{(q)}(x,c)}{W_{{\mathbf{T}}(\nu_{n_k})}^{(q)}(b,c)}W_{{\mathbf{T}}(\nu_{n_k})}^{(q)}(b,y)-W_{{\mathbf{T}}(\nu_{n_k})}^{(q)}(x,y)\right]dy\notag\\
								&=\int_{c}^bf(y)\left(\frac{W_{{\mathbf{T}}(\nu_{A})}^{(q)}(x,c)}{W_{{\mathbf{T}}(\nu_{A})}^{(q)}(b,c)}W_{{\mathbf{T}}(\nu_{A})}^{(q)}(b,y)-W_{{\mathbf{T}}(\nu_{A})}^{(q)}(x,y)\right)dy.
							\end{align*}}
							Here $W_{{\mathbf{T}}(\nu_A)}^{(q)}$ is the unique solution to
							\eqref{volterra_sf_w}. The interchange of limit and integration is justified by
							the generalized dominated convergence theorem, since Lemma~\ref{lemma_mon}
							implies that
							\begin{align*}
								\E\left[\left|\frac{W_{{\mathbf{T}}(\nu_{n_k})}^{(q)}(x,c)}{W_{{\mathbf{T}}(\nu_{n_k})}^{(q)}(b,c)}W_{{\mathbf{T}}(\nu_{n_k})}^{(q)}(b,y)-W_{{\mathbf{T}}(\nu_{n_k})}^{(q)}(x,y)\right|\right]\leq \E\left[W_{{\mathbf{T}}(\nu_{n_k})}^{(q)}(b,c)\right].
							\end{align*}
				\qedsymbol 
					\subsection{Proof of Lemma \ref{lemma_one_sided_limit}}
					 
					 (i) Fix $T>d$. Note that, {by Proposition \ref{PropD2},} $\nu_A(dz)$ and $\eta dz$ are both Radon measures. Then, using \eqref{ine_T}, it follows that ${\mathbf{T}}(\nu_A)(dz)$ is also a Radon measure. Moreover, proceeding like in Remark \ref{exi_uni_T}, and since {$\nu_A$} is Radon, we obtain that
					 \begin{align}\label{assum_prop_ex_uni_new}
					 	\sup_{x\in[d,T]}(1-{W^{(q+\eta)}(0)}({\mathbf{T}}(\nu_A)-\eta dz)\{x\})^{-1}=\sup_{x\in[d,T]}(1-{W^{(q+\eta)}(0)}{\mathbf{T}}(\nu_A)\{x\})^{-1}<\infty.
					 \end{align}
					 Hence, by Proposition \ref{exist_uni}, there exists a unique solution to \eqref{fun_u} on $[d,T]$.
					
					(ii) 
					{Let $c\le d$.}
					By combining \eqref{scale_fun_prop} and \eqref{volterra_sf_w} with
					Lemma~\ref{equiv_lemma}, we obtain, for $x\in[c,T]$,
					\begin{align*}
					W_{{\mathbf{T}}(\nu_A)}^{(q)}(x,c)&=W^{(\eta+q)}(x-c)+\int_{[c,x]}W^{(\eta+q)}(x-z)W_{{\mathbf{T}}(\nu_A)}^{(q)}(z,c)\left({\mathbf{T}}(\nu_A)(dz)-\eta dz\right)\notag\\
					&=W^{(\eta+q)}(x-c)+\int_{[c,x)}W^{(\eta+q)}(x-z)W_{{\mathbf{T}}(\nu_A)}^{(q)}(z,c)\left({\mathbf{T}}(\nu_A)(dz)-\eta dz\right)\notag\\&+W^{(\eta+q)}(0)W_{{\mathbf{T}}(\nu_A)}^{(q)}(x,c){\mathbf{T}}(\nu_A)\{x\}.
					\end{align*}
					Therefore, solving for $W_{{\mathbf{T}}(\nu_A)}^{(q)}(x,c)$ shows that
					$W_{{\mathbf{T}}(\nu_A)}^{(q)}(\cdot,c)$ satisfies, for $x \in [c,T]$, the integral equation
					\begin{align*}
						W_{{\mathbf{T}}(\nu_A)}^{(q)}(x,c)&=K(x)\left[W^{(\eta+q)}(x-c)+\int_{[c,x)}W^{(\eta+q)}(x-z)W_{{\mathbf{T}}(\nu_A)}^{(q)}(z,c)\left({\mathbf{T}}(\nu_A)(dz)-\eta dz\right)\right],
					\end{align*}
					where {$K(x)=\left(1-W^{(\eta+q)}(0){\mathbf{T}}(\nu_A)\{x\}\right)^{-1}$}. 
					
					{Note that, by the definition of $\nu_A$ in \eqref{revu_dec_one_sided} {and by \eqref{assum_prop_ex_uni_new}}, we obtain that
					\begin{align*}
						C(T):=\sup_{x\leq T} K(x)=\sup_{x\in[d,T]} K(x)<\infty.
					\end{align*}}
					Proceeding in the same way, we find that $u^{(q)}_{{\mathbf{T}}(\nu_A)}$
					satisfies the following integral equation for $x\in[c,T]$:
					\begin{align}
						u^{(q)}_{{\mathbf{T}}(\nu_A)}(x)=K(x)\left[e^{\Phi(\eta+q)x}+\int_{[d,x)}W^{(\eta+q)}(x-z)u^{(q)}_{{\mathbf{T}}(\nu_A)}(z)\left({\mathbf{T}}(\nu_A)(dz)-\eta dz\right)\right],\qquad x\in\R.
					\end{align}
					Thus, for $c<d$ and $x\in[c,T]$, we obtain
					\begin{align*}
						\frac{W_{{\mathbf{T}}(\nu_A)}^{(q)}(x,c)}{{W^{(\eta+q)}(-c)}}-u_{{\mathbf{T}}(\nu_A)}^{(q)}(x)&=K(x)\Bigg[\frac{W^{(\eta+q)}(x-c)}{{W^{(\eta+q)}(-c)}}-e^{\Phi(\eta+q)x}\notag\\&+\int_{[d,x)}W^{(\eta+q)}(x-z)\left(\frac{W_{{\mathbf{T}}(\nu_A)}^{(q)}(z,c)}{{W^{(\eta+q)}(-c)}}-u_{{\mathbf{T}}(\nu_A)}^{(q)}(z)\right)\left({\mathbf{T}}(\nu_A)(dz)-\eta dz\right)\Bigg],
					\end{align*}
					where we used that $\int_{[c,d)}W^{(\eta+q)}(x-z)\frac{W_{{\mathbf{T}}(\nu_A)}^{(q)}(z,c)}{{W^{(\eta+q)}(-c)}}\left({\mathbf{T}}(\nu_A)(dz)-\eta dz\right)=0$, which follows from the definition of $\nu_A$ in \eqref{revu_dec_one_sided}.
					Consequently,
					\begin{align*}
						\left|\frac{W_{{\mathbf{T}}(\nu_A)}^{(q)}(x,c)}{{W^{(\eta+q)}(-c)}}-u_{{\mathbf{T}}(\nu_A)}^{(q)}(x)\right|&\leq {C(T)}\Bigg[\left|\frac{W^{(\eta+q)}(x-c)}{W^{(\eta+q)}(-c)}-e^{\Phi(\eta+q)x}\right|\notag\\&+W^{(\eta+q)}(T-d)\int_{[d,x)}\left|\frac{W_{{\mathbf{T}}(\nu_A)}^{(q)}(z,c)}{{W^{(\eta+q)}(-c)}}-u_{{\mathbf{T}}(\nu_A)}^{(q)}(z)\right|\left({\mathbf{T}}(\nu_A)(dz)+\eta dz\right)\Bigg].
					\end{align*}
					Then, using Gronwall's inequality (see \cite[Appendix, Theorem 5.1]{EK}) and \eqref{W_q_limit}, gives for $x\leq T$
					\begin{align*}
						\lim_{c\to-\infty}&\left|\frac{W_{{\mathbf{T}}(\nu_A)}^{(q)}(x,c)}{{W^{(\eta+q)}(-c)}}-u_{{\mathbf{T}}(\nu_A)}^{(q)}(x)\right|\notag\\&\leq\lim_{c\to-\infty}\left|\frac{W^{(\eta+q)}(x-c)}{W^{(\eta+q)}(-c)}-e^{\Phi(\eta+q)x}\right|{C(T)}\exp\left\{{C(T)}W^{(\eta+q)}(T-d){\left[{\mathbf{T}}(\nu_A)[d, T]+\eta(T-d)\right]}\right\}\notag\\&=0,
					\end{align*}
					as stated. \qedsymbol 
					
					We finish this section by providing the proof of Proposition \ref{one_sided_downwards}.
					\subsection{Proof of Proposition \ref{one_sided_downwards}}
					\begin{proof}
						(i) By identity \eqref{two_sided_down_iden} we have that 
						\begin{align*}
							Z^{(q)}_{{\mathbf{T}}(\nu_A)}(x,c)-\frac{Z^{(q)}_{{\mathbf{T}}(\nu_A)}(b,c)}{W_{{\mathbf{T}}(\nu_A)}^{(q)}(b,c)}W_{{\mathbf{T}}(\nu_A)}^{(q)}(x,c)\geq0,\qquad x\leq b,
						\end{align*}	
						which implies that
						\begin{align*}
							\frac{Z^{(q)}_{{\mathbf{T}}(\nu_A)}(x,c)}{W_{{\mathbf{T}}(\nu_A)}^{(q)}(x,c)}\geq \frac{Z^{(q)}_{{\mathbf{T}}(\nu_A)}(b,c)}{W_{{\mathbf{T}}(\nu_A)}^{(q)}(b,c)},\qquad x\leq b.
						\end{align*}
						Thus, the mapping $b\mapsto \frac{Z^{(q)}_{{\mathbf{T}}(\nu_A)}(b,c)}{W_{{\mathbf{T}}(\nu_A)}^{(q)}(b,c)}$ is non-increasing and therefore convergent as $b\to\infty$, and thus $C^{(q)}_{{\mathbf{T}}(\nu_A)}$ is well-defined. Then taking $b\to\infty$ in \eqref{two_sided_down_iden} and using monotone convergence yields \eqref{one_sided_down_iden}.
						
						(ii) Note that by {Theorem} \ref{resol_ts} we have that 
						\[
						\frac{W_{{\mathbf{T}}(\nu_A)}^{(q)}(x,c)}{W_{{\mathbf{T}}(\nu_A)}^{(q)}(b,c)}W_{{\mathbf{T}}(\nu_A)}^{(q)}(b,y)-W_{{\mathbf{T}}(\nu_A)}^{(q)}(x,y)\geq0,\qquad x,y\in[c,b],
						\]
						given that it is the resolvent density of the processes killed by the additive functional $A$. This inequality implies that
						\[
						\frac{W_{{\mathbf{T}}(\nu_A)}^{(q)}(b,y)}{W_{{\mathbf{T}}(\nu_A)}^{(q)}(b,c)}\geq \frac{W_{{\mathbf{T}}(\nu_A)}^{(q)}(x,y)}{W_{{\mathbf{T}}(\nu_A)}^{(q)}(x,c)},\qquad x,y\in[c,b].
						\]
						Therefore, the mapping {$b\mapsto \frac{W_{{\mathbf{T}}(\nu_A)}^{(q)}(b,y)}{W_{{\mathbf{T}}(\nu_A)}^{(q)}(b,c)}$} is non-decreasing. {Since, by Lemma \ref{lemma_mon}, this ratio is bounded by $1$,} the limit
						\[
						{c^{(q)}_{{\mathbf{T}}(\nu_A)}}(y,c):=\lim_{b\to\infty}\frac{W_{{\mathbf{T}}(\nu_A)}^{(q)}(b,y)}{W_{{\mathbf{T}}(\nu_A)}^{(q)}(b,c)},\qquad y\geq c,
						\]
						exists and is well defined. { Moreover, letting $b\to\infty$ in
							\eqref{resolvent_iden} and applying the monotone convergence theorem
							yields \eqref{resol_down}.}
						
	By the definition of $c^{(q)}_{{\mathbf{T}}(\nu_A)}$ and by the monotonicity of the mapping $y\mapsto W_{{\mathbf{T}}(\nu_A)}^{(q)}(b,y)$ from Lemma \ref{lemma_mon}, the function $y\mapsto {c^{(q)}_{{\mathbf{T}}(\nu_A)}}(y,c)$ is non-increasing. Thus, $c^{(q)}_{{\mathbf{T}}(\nu_A)}(y,c)\leq c^{(q)}_{{\mathbf{T}}(\nu_A)}(c,c)=1$ for every $y\geq c$. 
					\end{proof}
\section*{Acknowledgements}

J. L. P\'erez gratefully acknowledges the support of the Fulbright Program during his sabbatical stay at Arizona State University. He also wishes to thank the faculty and staff of the School of Mathematical and Statistical Sciences at ASU for their hospitality. 
				\appendix
\section{Some properties of the additive functionals and their Revuz measures for spectrally negative L\'evy processes}\label{prop_PCNAF}
{In this section, we provide some properties of PcNAFs associated with spectrally negative L\'evy processes and of their corresponding Revuz measures, as defined in Section \ref{PcNAFforLevy}.

By identity (8.10) in \cite{Getoor}, the Revuz measure of a PcNAF of $X$ coincides with that of $X^{\Ei}$, where $X^{\Ei}$ denotes the L\'evy process $X$ killed at an independent exponential time $\Ei$ with rate $1$. Consequently, in several proofs throughout this section, we will work with PcNAFs and their Revuz measures associated with $X^{\Ei}$. For a PcNAF $A$, we define the Revuz measure $\nu_A$ as}
\begin{align}
\nu_A(B)=\lim_{t\downarrow0}\frac{1}{t}
\int_\R {\E_x}\left(\int_{(0, t]}1_B(X^{\Ei}_s)d A_s\right) dx
,\quad B\in\mathcal{B}(\R). 
\label{Revuz_measure_Ei}
\end{align} 
{We denote by $\Lambda$ the Lebesgue measure, which we use here as the reference measure. 
{Note that $\Lambda$ is a potential} (for the definition of potentials, see \cite[p.6]{Getoor}), since}
\begin{align}
\int_\bR \E_x\left[\int_0^\infty 1_B(X^{\Ei}_t) dt\right] dx
=&\int_\bR \E_x\left[\int_0^\infty e^{-t}1_B(X_t) dt\right] dx\\
=&\int_0^\infty e^{-t}\E\left[\int_\bR 1_B(X_t+x)dx\right] dt \\
=&\int_0^\infty e^{-t}\Lambda( B) dt \\
=&\Lambda(B),\quad B\in\mathcal{B}(\R). \label{dis}
\end{align}
	Of course, one may choose a reference measure other than $\Lambda$ for $X$ or $X^{\Ei}$. However, throughout this paper we adopt $\Lambda$ as the reference measure for both $X$ and $X^{\Ei}$, since $\Lambda$ is a potential with respect to $X^{\Ei}$. Combining this fact with \cite[Remark 1, p.~227]{BB} {and Proposition 8.11 in \cite{Getoor}}, the Revuz measure defined in \cite[p.~231]{BB} can be regarded as consistent with the definition used here.
	Consequently, in the proofs of several results below, we may freely use both the definition of the Revuz measure given above and the results obtained in \cite[Section 6.5]{BB}.

{We now recall some facts about PcNAFs, following \cite[Section 6.5]{BB}. Our goal is to verify that the associated Revuz measure is a Radon measure. In \cite[Section 6.5]{BB}, this is carried out in an abstract setting via the relationship between Revuz measures and $\Lambda$-strongly smooth measures (for the definition of $\Lambda$-strongly smooth, see \cite[p.217]{BB}).

In that definition, the terms ``$\xi$-polar'' and ``$\rho$-negligible'' are used. 
Here, ``the set $B\in\B(\R)$ is $\Lambda$-polar'' means that 
\begin{align}
\int_\R \p_x (T_B <\infty) dx =0, 
\end{align}
where $T_B:=\inf\{t\geq 0: X^{\Ei}_t \in B\}$. 
By \cite[Proposition 1.7.4 and Theorem 1.8.5]{BB}, this definition coincides with that of the $\Lambda$-polar set defined by \cite[p.54]{BB}.
	{Note that for any $x \in \R$, since $X^{\Ei}$ has no positive jumps and by \cite[Theorem 3.12]{K}, we have}
\begin{align}
\p_y(T_{\{x\}}<\infty)=\p_y(T_{[x, \infty)} <\infty)= e^{-\Phi(1)(x-y)}>0,\qquad y\in(-\infty,x),
\end{align}
and thus for any non-empty set $B\in\B(\R)$, 
\begin{align}
\int_\R \p_x (T_B <\infty) dx
\geq \int_{-\infty}^{b} \p_x (T_{\{b\}} <\infty) dx
=\int_0^\infty e^{-\Phi(1)x}dx>0,
\end{align}
where $b$ is a point in $B$. {Hence, the only $\Lambda$-polar set for $X^{\Ei}$ is the empty set.}

{By this fact, and since the measure $\rho$ in \cite[p.217]{BB} corresponds to $\Lambda$ due to \eqref{dis}, the terms 
``$\xi$-polar'' and ``$\rho$-negligible'' in \cite[p.217]{BB} correspond respectively to ``empty'' and ``$\Lambda$-negligible'', in our setting. Thus, we first need to establish the equivalence between being $\Lambda$-strongly smooth and being a Radon measure.}
\begin{proposition}\label{PropD1}
For an $\sigma$-finite measure $\xi$, the following are equivalent.
\begin{enumerate}
\item $\xi$ is $\Lambda$-strongly smooth. 
\item $\xi$ is a Radon measure.
\end{enumerate}
\end{proposition}
\begin{proof}
We prove (1)$\Rightarrow$(2). 
It is sufficient to prove that $\xi ([-a, a])<\infty$ for any $a>0$ under the assumption that $\xi$ is $\Lambda$-strongly smooth. 
By the definition of $\Lambda$-strongly smooth and by \cite[Theorem 1.8.5]{BB}, there exists an increasing sequence $\{B_n\}_{n\in\N}\subset \mathcal{B}(\R)$ such that $\xi(B_n)<\infty$ and 
\begin{align}
\left\{x \in\R:   \inf_{n\in\N}\E_x\left[u(X^{\Ei}_{T_{\R\backslash B_n}}) 1_{\{T_{\R\backslash B_n}<\infty\}} \right]>0\right\}=\varnothing
\label{nest}
\end{align}
for any excessive function $u$ which has the expression  $u(x):=\E_x\left[\int_0^\infty f(X^{\Ei}_t)dt \right]$ for some measurable function $f$ with $f \in (0,1]$. 
We assume $[-a, a]\not\subset B_n$ for all $n\in\N$ and show that this does not hold. In fact, if $f\equiv 1$ {then} $u\equiv 1$, and 
\begin{align}
\E_{-a}\left[u(X^{\Ei}_{T_{\R\backslash B_n}}){1_{\{T_{\R\backslash B_n}<\infty\}}}\right]
=\p_{-a}\left(T_{\R\backslash B_n}<\infty \right)
\geq \p_{-a}\left(T_{(a, \infty)}<\infty \right)=e^{-2\Phi(1) a}
\end{align}
for all $n\in\N$,  {where the inequality uses the fact that if $X$ starts from $-a$ and reaches $(a,\infty)$, it must pass through all values in $[-a, a]$ beforehand, since $X$ has no positive jumps. Hence, we have that  \[-a\in \left\{x \in\R:   \inf_{n\in\N}\E_x\left[u(X^{\Ei}_{T_{\R\backslash B_n}}) \right]>0\right\}\] which is a contradiction to \eqref{nest}.}
%
Thus, $[-a, a]\subset B_n$ for large enough $n$ and $\xi([-a, a])<\infty$ holds. 

We prove (2)$\Rightarrow$(1). 
It is sufficient to show the existence of a sequence $\{B_n\}_{n\in\N}$ satisfying the condition stated previously, under the assumption that $\xi$ is a Radon measure.
We take $B_n=[-n, n]$, then $\xi(B_n)<\infty$. 
Then, for $x \in \bR$ and any excessive function $u$ which has the expression above, we have, for $n\in\N$ such that $|x|<n$,
\begin{align}
\E_x\left[u(X^{\Ei}_{T_{\R\backslash B_n}}) \right]
\leq\p_0\left[T_{(-\infty, -n-x)\cup(n-x, \infty)}<\infty \right]
=\E_0\left[e^{-(\tau^-_{-n-x}\land\tau^+_{n-x})}\right]\to 0 \quad\text{as}\quad n\to\infty. 
\end{align}
Thus, \eqref{nest} holds. 
\end{proof}
\subsection{Proof of Proposition \ref{PropD2}}
{We will use some results about the Revuz measure presented in \cite[Section 6.5]{BB}. 
		However, in \cite[Section 6.5]{BB}, a theory is developed concerning the Revuz measures corresponding to positive left additive functionals (PLAFs).
 Therefore, we will first explain that, in the case of $X^{\Ei}$, the Revuz measures associated with PcNAFs and PLAFs coincide.} 
By \eqref{dis} and Proposition 8.11 in \cite{Getoor}, 
it holds 
\begin{align}
\nu_A(B)=\int_\R\E_x\left[\int_{(0, \infty)}1_B(X^{\Ei}_t)dA_t \right]dx
,\quad B\in\mathcal{B}(\R). \label{another_def}
\end{align}
{On the other hand,} by \cite[Proposition 6.5.5]{BB}, there exists PLAF $\hat{A}=\{\hat{A}_t:t\geq 0\}$ such that 
\begin{align}
\label{euivLR}
A_t(\omega)=\hat{A}_{t+}(\omega)-\hat{A}_{0+}(\omega),\quad\omega\in\Omega, \ t\geq 0.
\end{align}

{First,} we want to confirm that $\left\{x\in\R:\E_x\left[\hat{A}_{0+}\right]>0 \right\}$ is at most countable, and thus that  $\int_\R\E_x\left[\hat{A}_{0+}\right]dx=0$ in the case of spectrally negative L\'evy processes. 
{By the definition of a PLAF (see page 229 in \cite{BB}) and the discussion at the bottom of page 231 in \cite{BB}}, there exists function $f_{\hat{A}}$ 
from $\R$ to $[0, \infty)$ such that $\sum_{s\in[0,t)} (\hat{A}_{s+}-\hat{A}_{s})=\sum_{s\in[0,t)}f_{\hat{A}}(X_s) $ for $t \geq 0$ and thus $\E_x\left[\hat{A}_{0+}\right]=f_{\hat{A}}(x)$ for $x \in \bR$. 
{We take $a, b \in \R$ and consider a path of $X^{\Ei}$ starting from $a$ up to the first hitting time of $b$. Then, by the absence of positive jumps, the path of $X^{\Ei}$ visits all values in $[a, b)$; that is, under $\p_a$,}
\begin{align}
\sum_{x\in[a, b)}f_{\hat{A}}(x)\leq 
\sum_{t\leq T_{[b,\infty)}}f_{\hat{A}}(X^{\Ei}_t)\leq\hat{A}_{T_{[b, \infty)}} < \infty, \qquad\text{on }\{{T_{[b,\infty)}<\infty}\}. 
\end{align} 
Thus, $\{x\in[a, b):f_{\hat{A}}(x)>0\}$ is at most countable for any $a, b \in\bR$ and we get the conclusion above. 

Therefore, by \eqref{euivLR}, {and using the facts that $A_0=0$ and that $A$ is right-continuous 
}, we have 
\begin{align}
\nu_A(B)&=\int_\R\E_x\left[\int_{[0, \infty)}1_B(X^{\Ei}_t)dA_t \right]dx\notag\\&=\int_\R\E_x\left[\int_{[0, \infty)}1_B(X^{\Ei}_t)d\hat{A}_t \right]dx-\int_\R1_B(x)\E_x\left[\hat{A}_{0+}\right]dx
\\&=\int_\R\E_x\left[\int_{[0, \infty)}1_B(X^{\Ei}_t)d\hat{A}_t \right]dx
=:\nu_{\hat{A}}(B),\qquad\qquad B\in\mathcal{B}(\R).
\label{eqRev}
\end{align}
	{Recall that in \eqref{dis}, we showed that $\Lambda$, which we are using as the reference measure here, is a potential that can be expressed in terms of a potential kernel. Combining this fact with \cite[Remark 1 in p.227]{BB}, we conclude that $\nu_{\hat{A}}$ coincides with the Revuz measure associated with $\hat{A}$ defined in \cite[p.231]{BB}}. 
By \cite[Corollary 6.5.12]{BB}, the measure $\nu_{\hat{A}}$ is $\Lambda$-strongly smooth. Thus, by Proposition \ref{PropD1} and \eqref{eqRev}, the proof is complete. 
\subsection{Proof of Proposition \ref{conv_radon}}
	By \cite[Corollary 6.5.4]{BB}, there exists a proper optional homogeneous random measure  $k_\nu $ whose Revuz measure is $\nu $. {Moreover, using} \cite[Corollary 4.5.12]{BB}, there also exists a proper optional perfect homogeneous random measure $k'_\nu $ {whose resolvent is a regular supermedian kernel, and} whose Revuz measure, in the sense of the definition in \cite[p.~226]{BB}, coincides with $ \nu $. {Furthermore, since the measure $\nu $ is $\Lambda $-strongly smooth (by Proposition~\ref{PropD1}),} it follows from \cite[Theorem 6.5.9]{BB} that there exists a PLAF $\hat{A}$ such that its Revuz measure is $\nu$. 
By \cite[Proposition 6.5.5]{BB}, we can take a PcNAF $A$ satisfying \eqref{euivLR} and thus by \eqref{eqRev}, the Revuz measure associated with $A$ is $\nu$. 
\subsection{Proof of Proposition \ref{Prfof(2.10)}}
	It is easy to confirm that the Revuz measure of $L^a$ is $\delta_a$.  
			Let $\nu_A$ be the Revuz measure of the PcNAF $A$, and consider the PcNAF $A^\prime=\{A^\prime_t:t\geq0\}$ defined by 
			\begin{align}
				A^\prime_t:=\int_\R L^a_t \nu_A(da),\qquad t\geq0,
				\label{rep_A}
			\end{align}
			{and denote by $\nu_{A^\prime}$ to its Revuz measure.} 
			Then, for a non-negative measurable function $f$, 
			\begin{align}
			\int_\R f(y)\nu_A(dy)
			= \int_\R  \int_\R f(z)\delta_y(dz)\nu_A(dy)
			&=\int_\R  \left( \lim_{t\downarrow 0}\frac{1}{t}\int_\R\E_x\left[\int_{(0,t]}f({X^{\Ei}_s}) dL^y_s\right]dx\right)\nu_A(dy)
			\\
			&=\lim_{t\downarrow 0}\frac{1}{t}\int_\R  \left( \int_\R\E_x\left[\int_{(0,t]}f({X^{\Ei}_s}) dL^y_s\right]dx\right)\nu_A(dy)
			\\
			&=\lim_{t\downarrow 0}\frac{1}{t}\int_\R\E_x\left[\int_{(0,t]}f({X^{\Ei}_s}) dA^\prime_s\right]dx= \int_\R f(y)\nu_{A^\prime}(dy),
			\end{align}
			In the second equality above, we used \eqref{Revuz_measure}, and in the third equality, we applied the monotone convergence theorem, since the inner term in the integral is non-decreasing in $t$ by (8.9) in \cite{Getoor}. In other words, the respective Revuz measures of $A$ and $A'$ coincide.
			
		{Let us define the PLAFs $\hat{A}$ and $\hat{A}^\prime$ for $X^{\Ei}$ in the same way as in \eqref{euivLR}, using the PcNAFs $A$ and $A^\prime$ for $X^{\Ei}$, respectively.
		Then, by the proof of Proposition \ref{PropD2}, the Revuz measures of $\hat{A}$ and $\hat{A}^\prime$ are both equal to $\nu_A$. Moreover, $\nu_A$ coincides with the Revuz measure associated with both $\hat{A}$ and $\hat{A}^\prime$, as defined in \cite[p.231]{BB}, by the same discussion following identity \eqref{eqRev}.} 
{On the other hand, note that exceptional sets are $\Lambda$-strongly inessential (see p.229 of \cite{BB}). Moreover, $\Lambda$-strongly inessential are $\Lambda$-polar (see p.168 of \cite{BB}), and hence are empty in this case.}
			{Hence,} by \cite[Theorem 6.5.11 ii) and iii)]{BB} 
		we have 
		\begin{align}
		\p_x\left(\{\hat{A}_t\neq\hat{A}^\prime_t\text{ for some }t\geq 0\} \right)=0,\quad x \in\R,
		\end{align}
		and thus 
		\begin{align}
		\p_x\left(\{A_t\neq A^\prime_t\text{ for some }t\geq 0\} \right)=0,\quad x \in\R. 
		\end{align}
As a result, we have 
			\eqref{representation_equation}. 
			\section{Proof of Theorem \ref{LiZhouProof_thm}}\label{LiZhouProof}
		We assume that the L\'evy process $X$ has paths of bounded variation and prove the statement of Theorem \ref{LiZhouProof_thm}. 
				We proceed by induction on $n$. 
                Since there is no loss of generality in assuming that $a_n<b$, we proceed under this assumption.
                Assume that
				\eqref{local_time_exit}, \eqref{tsd_u_0}, and \eqref{local_time_resolvent}
				hold for $n=k-1$, for some $k\in\mathbb{N}$.
				We then establish these identities for $n=k$.
				The base case $n=1$ follows directly from the general argument,
				and we therefore omit its proof.
				
				(i) [Two-sided exit problem].- For $c<x<b$, define
				\begin{align*}
					g(x;b,c):=\E_{x}\left[e^{-q\tau_b^+-\sum_{i=1}^kp_iL^{a_i}_{\tau_b^+}};\tau_b^+<\tau_c^-\right].
				\end{align*}
				{Then, for $x<a_k$
				, it follows from the strong Markov property,
					\eqref{lt_bv}, and the induction hypothesis (see \eqref{local_time_exit})
					that}
				\begin{align}\label{ts_1}
					g(x;b,c)&=\E_{x}\left[e^{-q\tau_{a_k}^+-\sum_{i=1}^kp_iL^{a_i}_{\tau_{a_k}^+}}g(a_k;b,c);\tau_{a_k}^+<\tau_c^-\right]\notag\\
					&=g(a_k;b,c)\E_{x}\left[e^{-q\tau_{a_k}^+-\sum_{i=1}^{k-1}p_iL^{a_i}_{\tau_{a_k}^+}}e^{-p_k/{\mathbf{d}}};\tau_{a_k}^+<\tau_c^-\right]\notag\\
					&=e^{-p_k/{\mathbf{d}}}g(a_k;b,c)\frac{W^{(q,p_1,\dots,p_{k-1})}_{(a_1,\dots,a_{k-1})}(x,c)}{W^{(q,p_1,\dots,p_{k-1})}_{(a_1,\dots,a_{k-1})}(a_k,c)}.
				\end{align}
				On the other hand, for $x\geq a_k$, using the strong Markov property together with \eqref{ts_1} gives
				\begin{align}\label{lt_bv_5}
					g(x;&b,c)=\E_{x}\left[e^{-q\tau_{b}^+};\tau_{b}^+<\tau_{a_k}^-\right]+\E_{x}\left[e^{-q\tau_{a_k}^-}g(X_{\tau_{a_k}^-};b,c);\tau_{a_k}^-<\tau_b^+\right]\notag\\
					&=\frac{W^{(q)}(x-a_k)}{W^{(q)}(b-a_k)}+e^{-p_k/{\mathbf{d}}}g(a_k;b,c)\E_{x}\left[e^{-q\tau_{a_k}^-}\frac{W^{(q,p_1,\dots,p_{k-1})}_{(a_1,\dots,a_{k-1})}(X_{\tau_{a_k}^-},c)}{W^{(q,p_1,\dots,p_{k-1})}_{(a_1,\dots,a_{k-1})}(a_k,c)};\tau_{a_k}^-<\tau_b^+\right]\notag\\
					&=\frac{W^{(q)}(x-a_k)}{W^{(q)}(b-a_k)}+e^{-p_k/{\mathbf{d}}}g(a_k;b,c)\left(\frac{W^{(q,p_1,\dots,p_{k-1})}_{(a_1,\dots,a_{k-1})}(x,c)}{W^{(q,p_1,\dots,p_{k-1})}_{(a_1,\dots,a_{k-1})}(a_k,c)}-\frac{W^{(q)}(x-a_k)}{W^{(q)}(b-a_k)}\frac{W^{(q,p_1,\dots,p_{k-1})}_{(a_1,\dots,a_{k-1})}(b,c)}{W^{(q,p_1,\dots,p_{k-1})}_{(a_1,\dots,a_{k-1})}(a_k,c)}\right),
				\end{align}
				{where, in the third equality, we used 
				the fact that the Markov property and the induction hypothesis imply that for $c\leq a_k\leq x\leq b$
				\begin{align}\label{lema2.1}
				&\frac{W^{(q,p_1,\dots,p_{k-1})}_{(a_1,\dots,a_{k-1})}(x,c)}{W^{(q,p_1,\dots,p_{k-1})}_{(a_1,\dots,a_{k-1})}(b,c)}=\E_{x}\left[e^{-q\tau_b^+-\sum_{i=1}^{k-1}p_iL^{a_i}_{\tau_b^+}};\tau_{b}^+<\tau_{c}^-\right]\notag\\&\qquad\qquad=\E_{x}\left[e^{-q\tau_b^+};\tau_{b}^+<\tau_{a_k}^-\right]+\E_{x}\left[e^{-q\tau_{a_k}^-}\E_{X_{\tau_{a_k}^-}}\left[e^{-q\tau_b^+-\sum_{i=1}^{k-1}p_iL^{a_i}_{\tau_b^+}};\tau_{b}^+<\tau_{c}^-\right];\tau_{a_k}^-<\tau_b^+\right]\notag\\
				&\qquad\qquad=\frac{W^{(q)}(x-a_k)}{W^{(q)}(b-a_k)}+\E_{x}\left[e^{-q\tau_{a_k}^-}\frac{W^{(q,p_1,\dots,p_{k-1})}_{(a_1,\dots,a_{k-1})}(X_{\tau_{a_k}^-},c)}{W^{(q,p_1,\dots,p_{k-1})}_{(a_1,\dots,a_{k-1})}(b,c)};\tau_{a_k}^-<\tau_b^+\right].
				\end{align}
			}
				Taking $x=a_k$ in \eqref{lt_bv_5} and solving for $g(a_k;b,c)$, we obtain that
				\begin{align}\label{lt_bv_4}
					g(a_k;b,c)&=e^{p_k/{\mathbf{d}}}\left((e^{p_k/{\mathbf{d}}}-1)\frac{W^{(q)}(b-a_k)}{W^{(q)}(0)}+\frac{W^{(q,p_1,\dots,p_{k-1})}_{(a_1,\dots,a_{k-1})}(b,c)}{W^{(q,p_1,\dots,p_{k-1})}_{(a_1,\dots,a_{k-1})}(a_k,c)}\right)^{-1}\notag\\
					&=e^{p_k/{\mathbf{d}}}\frac{W^{(q,p_1,\dots,p_{k-1})}_{(a_1,\dots,a_{k-1})}(a_k,c)}{W^{(q,p_1,\dots,p_{k})}_{(a_1,\dots,a_{k})}(b,c)},
				\end{align}
				where in the last equality we used that $W^{(q)}(0)={\mathbf{d}}^{-1}$ together with \eqref{volterra_lt}.
				
				Then, using \eqref{lt_bv_4} in \eqref{lt_bv_5} gives , for $x\geq a_k$,
				\begin{align*}
					g(x;b,c)&=\frac{W^{(q)}(x-a_k)}{W^{(q)}(b-a_k)}\left(1-\frac{W^{(q,p_1,\dots,p_{k-1})}_{(a_1,\dots,a_{k-1})}(b,c)}{W^{(q,p_1,\dots,p_{k})}_{(a_1,\dots,a_{k})}(b,c)}\right)+\frac{W^{(q,p_1,\dots,p_{k-1})}_{(a_1,\dots,a_{k-1})}(x,c)}{W^{(q,p_1,\dots,p_{k})}_{(a_1,\dots,a_{k})}(b,c)}\notag\\
					&=\frac{{\mathbf{d}}(e^{p_k/{\mathbf{d}}}-1)W^{(q)}(x-a_k)W^{(q,p_1,\dots,p_{k-1})}_{(a_1,\dots,a_{k-1})}(a_k,c)+W^{(q,p_1,\dots,p_{k-1})}_{(a_1,\dots,a_{k-1})}(x,c)}{W^{(q,p_1,\dots,p_{k})}_{(a_1,\dots,a_{k})}(b,c)}\notag\\
					&=\frac{W^{(q,p_1,\dots,p_{k})}_{(a_1,\dots,a_{k})}(x,c)}{W^{(q,p_1,\dots,p_{k})}_{(a_1,\dots,a_{k})}(b,c)},
				\end{align*}
				where in the second and third equalities we used \eqref{volterra_lt}.

				We also use \eqref{lt_bv_4} in \eqref{ts_1} and obtain, for $x<a_k$, 
				\begin{align}
					g(x;b,c)=\frac{W^{(q,p_1,\dots,p_{k-1})}_{(a_1,\dots,a_{k-1})}(a_k,c)}{W^{(q,p_1,\dots,p_{k})}_{(a_1,\dots,a_{k})}(b,c)}\frac{W^{(q,p_1,\dots,p_{k-1})}_{(a_1,\dots,a_{k-1})}(x,c)}{W^{(q,p_1,\dots,p_{k-1})}_{(a_1,\dots,a_{k-1})}(a_k,c)}
					=\frac{W^{(q,p_1,\dots,p_{k})}_{(a_1,\dots,a_{k})}(x,c)}{W^{(q,p_1,\dots,p_{k})}_{(a_1,\dots,a_{k})}(b,c)}.
				\end{align}
				
				(ii) [Two-sided exit problem downwards].- For $c<x<b$, define
				\begin{align*}
				\tilde{g}(x;b,c):=\E_{x}\left[e^{-q\tau_c^--\sum_{i=1}^kp_iL^{a_i}_{\tau_c^-}};\tau_c^-<\tau_b^+\right].
				\end{align*}{
				{By the strong Markov property,
				\eqref{lt_bv}, and the induction hypothesis (see \eqref{tsd_u_0}), we obtain, for $x<a_k$
				, that}}
				\begin{align}\label{tsd_1}
				\tilde{g}(x;b,c)&=\E_x\left[e^{-q\tau_c^--\sum_{i=1}^{k-1}p_iL^{a_i}_{\tau_c^-}};\tau_c^-<\tau_{a_k}^+\right]+\tilde{g}(a_k;b,c)\E_x\left[e^{-q\tau_{a_k}^+-\sum_{i=1}^{k-1}p_iL^{a_i}_{\tau_{a_k}^+}}e^{-p_k/{\mathbf{d}}};\tau_{a_k}^+<\tau_c^-\right]\notag\\
				&
				\begin{aligned}
				=Z^{(q,p_1,\dots,p_{k-1})}_{(a_1,\dots,a_{k-1})}(x,c)-&\frac{W^{(q,p_1,\dots,p_{k-1})}_{(a_1,\dots,a_{k-1})}(x,c)}{W^{(q,p_1,\dots,p_{k-1})}_{(a_1,\dots,a_{k-1})}(a_k,c)}Z^{(q,p_1,\dots,p_{k-1})}_{(a_1,\dots,a_{k-1})}(a_k,c)\\
				&\qquad +\frac{W^{(q,p_1,\dots,p_{k-1})}_{(a_1,\dots,a_{k-1})}(x,c)}{W^{(q,p_1,\dots,p_{k-1})}_{(a_1,\dots,a_{k-1})}(a_k,c)}e^{-p_k/{\mathbf{d}}}\tilde{g}(a_k;b,c).
				\end{aligned}
				\end{align} 
				{
					Applying the strong Markov property along with the induction hypothesis and \eqref{lema2.1}, we obtain for $c\leq a_k\leq x\leq b$, that
				\begin{align*}
					Z^{(q,p_1,\dots,p_{k-1})}_{(a_1,\dots,a_{k-1})}(x,c)-&\frac{W^{(q,p_1,\dots,p_{k-1})}_{(a_1,\dots,a_{k-1})}(x,c)}{W^{(q,p_1,\dots,p_{k-1})}_{(a_1,\dots,a_{k-1})}(b,c)}Z^{(q,p_1,\dots,p_{k-1})}_{(a_1,\dots,a_{k-1})}(b,c)=\E_{x}\left[e^{-q\tau_c^--\sum_{i=1}^{k-1}p_iL^{a_i}_{\tau_c^-}};\tau_{c}^-<\tau_{b}^+\right]\notag\\&=\E_{x}\left[e^{-q\tau_{a_k}^-}\E_{X_{\tau_{a_k}^-}}\left[e^{-q\tau_c^--\sum_{i=1}^{k-1}p_iL^{a_i}_{\tau_c^-}};\tau_{c}^-<\tau_{b}^+\right];\tau_{a_k}^-<\tau_b^+\right]\notag\\
					&=\E_{x}\left[e^{-q\tau_{a_k}^-}Z^{(q,p_1,\dots,p_{k-1})}_{(a_1,\dots,a_{k-1})}(X_{\tau_{a_k}^-},c);\tau_{a_k}^-<\tau_b^+\right]\notag\\&-\E_{x}\left[e^{-q\tau_{a_k}^-}\frac{W^{(q,p_1,\dots,p_{k-1})}_{(a_1,\dots,a_{k-1})}(X_{\tau_{a_k}^-},c)}{W^{(q,p_1,\dots,p_{k-1})}_{(a_1,\dots,a_{k-1})}(b,c)};\tau_{a_k}^-<\tau_b^+\right]Z^{(q,p_1,\dots,p_{k-1})}_{(a_1,\dots,a_{k-1})}(b,c)\notag\\
					&=\E_{x}\left[e^{-q\tau_{a_k}^-}Z^{(q,p_1,\dots,p_{k-1})}_{(a_1,\dots,a_{k-1})}(X_{\tau_{a_k}^-},c);\tau_{a_k}^-<\tau_b^+\right]\notag\\&-\left(\frac{W^{(q,p_1,\dots,p_{k-1})}_{(a_1,\dots,a_{k-1})}(x,c)}{W^{(q,p_1,\dots,p_{k-1})}_{(a_1,\dots,a_{k-1})}(b,c)}-\frac{W^{(q)}(x-a_k)}{W^{(q)}(b-a_k)}\right)Z^{(q,p_1,\dots,p_{k-1})}_{(a_1,\dots,a_{k-1})}(b,c).
				\end{align*}
				Therefore, for {$a_k\leq x\leq b$}, we have
				\begin{align}\label{lemma2.2}
				\E_{x}\Big[e^{-q\tau_{a_k}^-}Z^{(q,p_1,\dots,p_{k-1})}_{(a_1,\dots,a_{k-1})}&(X_{\tau_{a_k}^-},c);\tau_{a_k}^-<\tau_b^+\Big]\notag\\&=Z^{(q,p_1,\dots,p_{k-1})}_{(a_1,\dots,a_{k-1})}(x,c)-\frac{W^{(q)}(x-a_k)}{W^{(q)}(b-a_k)}Z^{(q,p_1,\dots,p_{k-1})}_{(a_1,\dots,a_{k-1})}(b,c).
				\end{align}}
				}
				Now, applying the strong Markov property and using \eqref{tsd_1} together with \eqref{lema2.1} and \eqref{lemma2.2}, we obtain for {$x\geq a_k$},
				\begin{align}\label{tsd_2}
				&\tilde{g}(x;b,c)=\E_x\left[e^{-q\tau_{a_k}^-}\tilde{g}(X_{\tau_{a_k}^-};b,c);\tau_{a_k}^-<\tau_b^+\right]\notag\\
				&=\E_x\left[e^{-q\tau_{a_k}^-}Z^{(q,p_1,\dots,p_{k-1})}_{(a_1,\dots,a_{k-1})}(X_{\tau_{a_k}^-},c);\tau_{a_k}^-<\tau_b^+\right]\notag\\&+\left(\frac{\tilde{g}(a_k;b,c)e^{-p_k/{\mathbf{d}}}}{W^{(q,p_1,\dots,p_{k-1})}_{(a_1,\dots,a_{k-1})}(a_k,c)}-\frac{Z^{(q,p_1,\dots,p_{k-1})}_{(a_1,\dots,a_{k-1})}(a_k,c)}{W^{(q,p_1,\dots,p_{k-1})}_{(a_1,\dots,a_{k-1})}(a_k,c)}\right)\E_x\left[e^{-q\tau_{a_k}^-}W^{(q,p_1,\dots,p_{k-1})}_{(a_1,\dots,a_{k-1})}(X_{\tau_{a_k}^-},c);\tau_{a_k}^-<\tau_b^+\right]\notag\\
				&=Z^{(q,p_1,\dots,p_{k-1})}_{(a_1,\dots,a_{k-1})}(x,c)-\frac{W^{(q)}(x-a_k)}{W^{(q)}(b-a_k)}Z^{(q,p_1,\dots,p_{k-1})}_{(a_1,\dots,a_{k-1})}(b,c)+\left(\frac{\tilde{g}(a_k;b,c)e^{-p_k/{\mathbf{d}}}}{W^{(q,p_1,\dots,p_{k-1})}_{(a_1,\dots,a_{k-1})}(a_k,c)}-\frac{Z^{(q,p_1,\dots,p_{k-1})}_{(a_1,\dots,a_{k-1})}(a_k,c)}{W^{(q,p_1,\dots,p_{k-1})}_{(a_1,\dots,a_{k-1})}(a_k,c)}\right)\notag\\
				&\times\left(W^{(q,p_1,\dots,p_{k-1})}_{(a_1,\dots,a_{k-1})}(x,c)-\frac{W^{(q)}(x-a_k)}{W^{(q)}(b-a_k)}W^{(q,p_1,\dots,p_{k-1})}_{(a_1,\dots,a_{k-1})}(b,c)\right)\notag\\
				&=Z^{(q,p_1,\dots,p_{k})}_{(a_1,\dots,a_{k})}(x,c)-\frac{W^{(q)}(x-a_k)}{W^{(q)}(b-a_k)}Z^{(q,p_1,\dots,p_{k})}_{(a_1,\dots,a_{k})}(b,c)+\left(\frac{\tilde{g}(a_k;b,c)e^{-p_k/{\mathbf{d}}}}{W^{(q,p_1,\dots,p_{k-1})}_{(a_1,\dots,a_{k-1})}(a_k,c)}-\frac{Z^{(q,p_1,\dots,p_{k-1})}_{(a_1,\dots,a_{k-1})}(a_k,c)}{W^{(q,p_1,\dots,p_{k-1})}_{(a_1,\dots,a_{k-1})}(a_k,c)}\right)\notag\\
				&\times\left(W^{(q,p_1,\dots,p_{k-1})}_{(a_1,\dots,a_{k-1})}(x,c)-\frac{W^{(q)}(x-a_k)}{W^{(q)}(b-a_k)}W^{(q,p_1,\dots,p_{k-1})}_{(a_1,\dots,a_{k-1})}(b,c)\right),
				\end{align}
				where in the last equality we used \eqref{volterra_lt}.
				
				Setting $x=a_k$ and solving for $\tilde{g}(a_k;b,c)$, we obtain that
				
				\begin{align}\label{tsd_3}
				\tilde{g}(a_k;b,c)&=\frac{\frac{Z^{(q,p_1,\dots,p_{k-1})}_{(a_1,\dots,a_{k-1})}(a_k,c)}{W^{(q,p_1,\dots,p_{k-1})}_{(a_1,\dots,a_{k-1})}(a_k,c)}W^{(q,p_1,\dots,p_{k
				})}_{(a_1,\dots,a_{k
				})}(b,c)-Z^{(q,p_1,\dots,p_{k
				})}_{(a_1,\dots,a_{k
				})}(b,c)}{ \displaystyle (e^{p_k/{\mathbf{d}}}-1)\frac{W^{(q)}(b-a_k)}{W^{(q)}(0)}+\frac{W^{(q,p_1,\dots,p_{k-1})}_{(a_1,\dots,a_{k-1})}(b,c)}{W^{(q,p_1,\dots,p_{k-1})}_{(a_1,\dots,a_{k-1})}(a_k,c)}}e^{p_k/{\mathbf{d}}}\notag\\
				&=\left(Z^{(q,p_1,\dots,p_{k-1})}_{(a_1,\dots,a_{k-1})}(a_k,c)-\frac{W^{(q,p_1,\dots,p_{k-1})}_{(a_1,\dots,a_{k-1})}(a_k,c)}{W^{(q,p_1,\dots,p_{k})}_{(a_1,\dots,a_{k})}(b,c)}Z^{(q,p_1,\dots,p_{k})}_{(a_1,\dots,a_{k})}(b,c)\right)e^{p_k/{\mathbf{d}}},
				\end{align}
				where in the last equality we used the recursion given in \eqref{volterra_lt}.
				Then, for $x\geq a_k$, 
				using \eqref{tsd_3} in \eqref{tsd_2}, yields
				\begin{align*}
				\tilde{g}(x;b,c)&=Z^{(q,p_1,\dots,p_{k})}_{(a_1,\dots,a_{k})}(x,c)-\frac{W^{(q)}(x-a_k)}{W^{(q)}(b-a_k)}Z^{(q,p_1,\dots,p_{k})}_{(a_1,\dots,a_{k})}(b,c)\notag\\&-\frac{Z^{(q,p_1,\dots,p_{k})}_{(a_1,\dots,a_{k})}(b,c)}{W^{(q,p_1,\dots,p_{k})}_{(a_1,\dots,a_{k})}(b,c)}\left(W^{(q,p_1,\dots,p_{k})}_{(a_1,\dots,a_{k})}(x,c)-\frac{W^{(q)}(x-a_k)}{W^{(q)}(b-a_k)}W^{(q,p_1,\dots,p_{k})}_{(a_1,\dots,a_{k})}(b,c)\right),
				\end{align*}
				which simplifies to \eqref{tsd_u_0}. For the case $x<a_k$, we use \eqref{tsd_3} in \eqref{tsd_1}, and the fact that $W^{(q,p_1,\dots,p_{k-1})}_{(a_1,\dots,a_{k-1})}(x,c)=W^{(q,p_1,\dots,p_{k})}_{(a_1,\dots,a_{k})}(x,c)$ and $Z^{(q,p_1,\dots,p_{k-1})}_{(a_1,\dots,a_{k-1})}(x,c)=Z^{(q,p_1,\dots,p_{k})}_{(a_1,\dots,a_{k})}(x,c)$. 
				
				(iii) [Resolvents].- For $c<x<b$ {and a non-negative, measurable, and bounded function $f:\R\to\R$, we define}
				\begin{align*}
				h(x;b,c):=\E_x\left[\int_0^{\tau_b^+\wedge\tau_c^-}e^{-qt-\sum_{i=1}^kp_iL^{a_i}_t}f(X_t)	dt\right].
				\end{align*}
				{For $x<a_k$, an application of the strong Markov property,
					together with \eqref{lt_bv} and the induction hypothesis
					(cf. \eqref{local_time_resolvent}), yields}
				\begin{align}\label{resolv_1}
				h(x;b,c)&=\E_x\left[\int_0^{\tau_{a_k}^+\wedge\tau_c^-}e^{-qt-\sum_{i=1}^kp_iL^{a_i}_t}f(X_t)	dt\right]+\E_x\left[\int_{\tau_{a_k}^+}^{\tau_b^+\wedge\tau_c^-}e^{-qt-\sum_{i=1}^kp_iL^{a_i}_t}f(X_t)dt\right]\notag\\
				&=\E_x\left[\int_0^{\tau_{a_k}^+\wedge\tau_c^-}e^{-qt-\sum_{i=1}^kp_iL^{a_i}_t}f(X_t)	dt\right]+\E_x\left[e^{-q\tau_{a_k}^+- p_k/\mathbf{d}}h(a_k;b,c);\tau_{a_k}^+<\tau_c^-\right]\notag\\
				&=\int_c^{a_k}f(y)\left\{\frac{W^{(q,p_1,\dots,p_{k-1})}_{(a_1,\dots,a_{k-1})}(x,c)}{W^{(q,p_1,\dots,p_{k-1})}_{(a_1,\dots,a_{k-1})}(a_k,c)}W^{(q,p_1,\dots,p_{k-1})}_{(a_1,\dots,a_{k-1})}(a_k,y)-W^{(q,p_1,\dots,p_{k-1})}_{(a_1,\dots,a_{k-1})}(x,y)\right\}dy\notag\\&+e^{-p_k/\mathbf{d}}h(a_k;b,c){\frac{W^{(q,p_1,\dots,p_{k-1})}_{(a_1,\dots,a_{k-1})}(x,c)}{W^{(q,p_1,\dots,p_{k-1})}_{(a_1,\dots,a_{k-1})}(a_k,c)}},
				\end{align}
				where in the last equality we used the induction assumption together with Step (i).
				
				On the other hand, for $x\geq a_k$, we have by the strong Markov property, and using \eqref{resolv_1} together with the induction assumption, we get
				\begin{align}\label{resol_2}
				h(x;b,c)&=\E_x\left[\int_0^{\tau_{a_k}^-\wedge\tau_b^+}e^{-qt-\sum_{i=1}^kp_iL^{a_i}_t}f(X_t)	dt\right]+\E_x\left[\int_{\tau_{a_k}^-}^{\tau_b^+\wedge\tau_c^-}e^{-qt-\sum_{i=1}^kp_iL^{a_i}_t}f(X_t)dt\right]\notag\\
				&=\E_x\left[\int_0^{\tau_{a_k}^-\wedge\tau_b^+}e^{-qt-\sum_{i=1}^kp_iL^{a_i}_t}f(X_t)	dt\right]+\E_x\left[e^{-q\tau_{a_k}^-}h(X_{\tau_{a_k}^-};b,c);\tau_{a_k}^-<\tau_b^+\right]\notag\\
				&=\int_{a_k}^bf(y)\left\{\frac{W^{(q,                                    p_1,\dots,p_{k-1})}_{(a_1,\dots,a_{k-1})}(x,a_k)}{W^{(q,p_1,\dots,p_{k-1})}_{(a_1,\dots,a_{k-1})}(b,a_k)}W^{(q,p_1,\dots,p_{k-1})}_{(a_1,\dots,a_{k-1})}(b,y)-W^{(q,p_1,\dots,p_{k-1})}_{(a_1,\dots,a_{k-1})}(x,y)\right\}dy\notag\\&+\int_c^{a_k}f(y)\Bigg\{\frac{W^{(q,p_1,\dots,p_{k-1})}_{(a_1,\dots,a_{k-1})}(a_k,y)}{W^{(q,p_1,\dots,p_{k-1})}_{(a_1,\dots,a_{k-1})}(a_k,c)}\E_x\left[e^{-q\tau_{a_k}^-}W^{(q,p_1,\dots,p_{k-1})}_{(a_1,\dots,a_{k-1})}(X_{\tau_{a_k}^-},c);\tau_{a_k}^-<\tau_b^+\right]\notag\\&-\E_x\left[e^{-q\tau_{a_k}^-}W^{(q,p_1,\dots,p_{k-1})}_{(a_1,\dots,a_{k-1})}(X_{\tau_{a_k}^-},y);\tau_{a_k}^-<\tau_b^+\right]\Bigg\}dy\notag\\&+h(a_k;b,c)\E_x\left[e^{-q\tau_{a_k}^--p_k/\mathbf{d}}\frac{W^{(q,p_1,\dots,p_{k-1})}_{(a_1,\dots,a_{k-1})}(X_{\tau_{a_k}^-},c)}{W^{(q,p_1,\dots,p_{k-1})}_{(a_1,\dots,a_{k-1})}(a_k,c)};\tau_{a_k}^-<\tau_b^+\right].
				\end{align}
				Note that $W^{(q,p_1,\dots,p_{k-1})}_{(a_1,\dots,a_{k-1})}(z,y)=W^{(q)}(z-y)$ for $y> a_{k-1}$.
				Hence, using \eqref{lema2.1} in \eqref{resol_2} yields, for $x\ge a_k$,
				\begin{align}\label{resol_3}
				h(x;b,c)&=\int_c^{b}f(y)\Bigg\{\frac{W^{(q,p_1,\dots,p_{k-1})}_{(a_1,\dots,a_{k-1})}(a_k,y)}{W^{(q,p_1,\dots,p_{k-1})}_{(a_1,\dots,a_{k-1})}(a_k,c)}\left(W^{(q,p_1,\dots,p_{k-1})}_{(a_1,\dots,a_{k-1})}(x,c)-\frac{W^{(q)}(x-a_k)}{W^{(q)}(b-a_k)}W^{(q,p_1,\dots,p_{k-1})}_{(a_1,\dots,a_{k-1})}(b,c)\right)\notag\\&-\left(W^{(q,p_1,\dots,p_{k-1})}_{(a_1,\dots,a_{k-1})}(x,y)-\frac{W^{(q)}(x-a_k)}{W^{(q)}(b-a_k)}W^{(q,p_1,\dots,p_{k-1})}_{(a_1,\dots,a_{k-1})}(b,y)\right)\Bigg\}dy\notag\\
				&+\frac{e^{-p_k/\mathbf{d}}h(a_k;b,c)}{W^{(q,p_1,\dots,p_{k-1})}_{(a_1,\dots,a_{k-1})}(a_k,c)}\left(W^{(q,p_1,\dots,p_{k-1})}_{(a_1,\dots,a_{k-1})}(x,c)-\frac{W^{(q)}(x-a_k)}{W^{(q)}(b-a_k)}W^{(q,p_1,\dots,p_{k-1})}_{(a_1,\dots,a_{k-1})}(b,c)\right).
				\end{align}
				Setting $x=a_k$ in \eqref{resol_3} and solving for $h(a_k;b,c)$, we obtain that{ 
				\begin{align}\label{resol_4}
				h(a_k;b,c)&=\frac{e^{p_k/\mathbf{d}}}{W^{(q,p_1,\dots,p_{k})}_{(a_1,\dots,a_{k})}(b,c)}\int_c^bf(y)\bigg\{W^{(q,p_1,\dots,p_{k-1})}_{(a_1,\dots,a_{k-1})}(b,y)W^{(q,p_1,\dots,p_{k-1})}_{(a_1,\dots,a_{k-1})}(a_k,c)\notag\\&-W^{(q,p_1,\dots,p_{k-1})}_{(a_1,\dots,a_{k-1})}(b,c)W^{(q,p_1,\dots,p_{k-1})}_{(a_1,\dots,a_{k-1})}(a_k,y)\bigg\},
				\end{align}}
				where we used identity \eqref{volterra_lt}.
				
				Thus, by applying \eqref{resol_4} to \eqref{resol_3} and \eqref{resolv_1}, using identity \eqref{volterra_lt}, and performing some computations, we obtain \eqref{local_time_resolvent}
			\section{Proof of Proposition \ref{exist_uni}}\label{proof_exist_uni}
				(i) First, we show the uniqueness of the solution to \eqref{volterra_prop}. To this end, note that using \eqref{volterra_prop}, we obtain for $x\geq y$
				\begin{align*}
				H(x,y)= h(x-y)+\int_{[y,x)}W^{(q)}(x-z)H(z,y)(\nu-\lambda)(dz)+W^{(q)}(0)H(x,y)(\nu-\lambda)(\{x\}).
				\end{align*}
				Then, solving for $H(x,y)$, we have
				\begin{align}\label{volterra_no_atom}
				H(x,y)=K(x)\left[h(x-y)+\int_{[y,x)}W^{(q)}(x-u)H(u,y)(\nu-\lambda)(du)\right],\qquad x\in[y,T],
				\end{align}
				with, 
				\begin{align}\label{bound_atom}
				K(x):=(1-W^{(q)}(0)(\nu-\lambda)\{x\})^{-1}\leq\sup_{x\in[y,T]}\Big{\{}(1-W^{(q)}(0)(\nu-\lambda)\{x\})^{-1}\Big{\}}:=C(y,T)<\infty,
				\end{align}
				where the last inequality follows from the assumption that $\inf_{x\in[y,T]}(1-W^{(q)}(0)(\nu-\lambda)\{x\})>0$.
				
				Now, consider two solutions $H(\cdot,y)$ and $\tilde{H}(\cdot,y)$ to \eqref{volterra_prop}. Using \eqref{volterra_no_atom} together with \eqref{bound_atom}, we obtain for $x\in[y,T]$
				\begin{align*}
					\sup_{y\leq u\leq x}|H(u,y)-\tilde{H}(u,y)|&\leq C(y,T)\sup_{y\leq u\leq x}\int_{[y,u)}W^{(q)}(u-z)|H(z,y)-\tilde{H}(z,y)|(\nu+\lambda)(dz)\notag\\
					&\leq C(y,T)W^{(q)}(T-y)\int_{[y,x)}\sup_{y\leq u\leq z}|H(u,y)-\tilde{H}(u,y)|(\nu+\lambda)(dz).
				\end{align*}
				Hence, by an application of Gronwall's inequality (see \cite[Appendix, Theorem 5.1]{EK}) we have that
				\begin{align*}
					\sup_{y\leq u\leq x}|H(u,y)-\tilde{H}(u,y)|=0,\quad x\in[y,T],
				\end{align*}
				and thereby the solution is unique.
								
				(ii) We now show the existence of the solution to \eqref{volterra_prop} in a {finite} interval. Using  \eqref{volterra_no_atom} we have that
				\begin{equation}\label{H_y}
				H(y,y)=K(y)h(0).
				\end{equation}
				On the other hand, note that using \eqref{H_y}, equation \eqref{volterra_no_atom} can be rewritten as follows for $x\in[y,T]$:
				\begin{align}\label{volterra_no_atom_y}
				H(x,y)&=K(x)\left[h(x-y)+\int_{(y,x)}W^{(q)}(x-u)H(u,y)(\nu-\lambda)(du)+W^{(q)}(x-y)H(y,y)(\nu-\lambda)\{y\}{1_{\{x>y\}}}\right]\notag\\
				&=K(x)\left[\tilde{h}(x,y)+\int_{(y,x)}W^{(q)}(x-u)H(u,y)(\nu-\lambda)(du)\right],
				\end{align}
				with $\tilde{h}(x,y):=h(x-y)+W^{(q)}(x-y)K(y)h(0)(\nu-\lambda)\{y\}1_{\{x>y\}}$
				, for $x\in[y,T]$.

				Let us denote by $(\nu+\lambda)^a$ and $(\nu+\lambda)^d$ the atomic and diffuse parts of the measure $\nu+\lambda$, respectively. Then, we can write
				\[
				(\nu+\lambda)^a(dz)=\sum_{i=1}^{\infty}p_i\delta_{a_i}(dz),\qquad\text{where $a_i \in \mathbb{R}$ and $p_i \geq 0$ for $i = 1, 2, \dots$}
				\]
				Using the fact that $(\nu+\lambda)[y,T]<\infty$, we can find $M\in\mathbb{N}$ such that
				\begin{align}\label{bound_atomic}
				\sum_{i=M+1,a_i\in[y,T]}^{\infty}p_i<\frac{1}{4C(y,T)W^{(q)}(T-y)}.
				\end{align}
				On the other hand, since the measure $(\nu+\lambda)^d$ is diffuse, for any $s>0$, the mapping 
				\[x\mapsto\int_{(y,x)}e^{-s(x-z)}W^{(q)}(x-z)(\nu+\lambda)^d(dz)
				\]
				is continuous. 
				
				Let us consider a sequence $(s_n)_{n\geq0}$ such that $s_n\uparrow\infty$ as $n\to\infty$. Then, if we denote
				\[
				F_n(x):=\int_{(y,x)}e^{-s_n(x-z)}W^{(q)}(x-z)(\nu+\lambda)^d(dz),\qquad x\in[y,T],
				\]
				we have that the sequence of continuous functions $(F_n)_{n\geq 1}$ is non-increasing. Additionally, using that $\nu$ and $\lambda$ are Radon measures, we obtain by dominated convergence
				\begin{align*}
				\lim_{n\to\infty}F_n(x)&=\lim_{n\to\infty}\int_{(y,x)}e^{-s_n(x-z)}W^{(q)}(x-z)(\nu+\lambda)^d(dz)\notag\\
				&\leq W^{(q)}(T-y)\lim_{n\to\infty}\int_{(y,x)}e^{-s_n(x-z)}(\nu+\lambda)^d(dz)=0,\qquad \text{for all $x\in[y,T]$}.
				\end{align*}
				Then, by Dini's Theorem, $F_n\downarrow 0$ uniformly in $[y,T]$. Therefore, we can find $s_0>0$ such that
				\begin{align}\label{bound_exi}
					\sup_{x \in[y, T]}\int_{[y,x]}e^{-s_0(x-z)}W^{(q)}(x-z)(\nu+\lambda)^d(dz)<\frac{1}{4C(y,T)}.
				\end{align}
				Let us denote by $\{b_1,b_2,\dots,b_j\}$ the subset of atoms of the set $\{a_1,\dots,a_M\}$ that lie inside the interval {$[y, T]$}, arranged in increasing order so that $b_1<b_2<\dots<b_j$.
				Now, we define the following operator for $x\in[y,b_1]$:
				\begin{align*}
					\mathcal{K}(f)(x)=\mathcal{K}^{(1)}(f)(x)=K(x)\int_{(y,x)}e^{-s_0(x-z)}W^{(q)}(x-z)f(z)(\nu-\lambda)(dz),\quad \mathcal{K}^{(n)}(f)(x)=\mathcal{K}(\mathcal{K}^{(n-1)}(f))(x),
				\end{align*}
				and
				\begin{align}\label{fun_H}
					H_0(x,y)=e^{-s_0x}{\tilde{h}(x,y)}K(x),\qquad H_{n+1}(x,y)=H_0(x,y)+\mathcal{K}(H_n(\cdot,y))(x).
				\end{align}
				By \eqref{bound_atomic} and \eqref{bound_exi}, we have
				\begin{align}\label{bound_k}
					|\mathcal{K}f(x)|\leq \frac{1}{2}\sup_{z\in{[y,b_1]}}|f(z)|,\qquad x\in[y,b_1].
				\end{align}
				Then, for $m>n$ and $x\in[y,b_1]$,
				\begin{align}
					|H_m(x,y)-H_n(x,y)|=\left|\sum_{k=n+1}^m\mathcal{K}^{(k)}( H_0(\cdot,y))(x)\right|\leq 2^{-n}\sup_{z\in[y,b_1]}|H_0(z, y)|.\label{Cauchy_seq}
				\end{align}
				Hence, $\{H_n(x,y):x\in[y,b_1]\}_{n\geq1}$ is a Cauchy sequence and, therefore, converges to a limit $\hat{H}$ in $[y,b_1]$. Then, taking $n\to\infty$ in \eqref{fun_H} , we get that 
				\begin{align*}
					\hat{H}(x,y)=K(x)\left(e^{-s_0x}\tilde{h}(x,y)+\int_{(y,x)}e^{-s_0(x-z)}W^{(q)}(x-z)\hat{H}(z,y){(\nu-\lambda)(dz)}\right),\qquad x\in[y,b_1].
				\end{align*}
				{
					Here, we used the dominated convergence theorem with \eqref{bound_k} and the fact that for $x \in [y, b_1]$, 
					\begin{align}
						|{H}_m(x, y)|\leq2\sup_{z\in[y,b_1]}|H_0(z, y)|,
					\end{align}
					which is obtained from \eqref{Cauchy_seq} with $n=0$.
					}
				Therefore, the function $H(x,y)=e^{s_0x}\hat{H}(x,y)$ satisfies \eqref{volterra_no_atom_y} for $[y,b_1]$. 
	By tracing the reverse process from \eqref{volterra_prop} to \eqref{volterra_no_atom_y}, we can see that $H(\cdot,y)$ is a solution to \eqref{volterra_prop}.

			Finally, note that if $h(u)\geq0$ for all $u\in[y,T]$ and $(\nu-\lambda)$ is a non-negative measure, then by \eqref{assum_atoms} we have  that $\tilde{h}(x,y)\geq0$ for all $x\in[y,T]$. Therefore,  the  construction of the solution $H$ implies that $H(x,y)\geq0$ for all $x\in[y,b_1]$.

				(iii) {We now extend the solution to the whole interval $[y, b_2]$.} 
				For $x\in{[y,b_2]}$, we consider the following integral equation:
			\begin{align}\label{volterra_atom_2}
			\tilde{H}(x,y)= h_1(x,y)+\int_{[y,x]\backslash\{b_1\}}W^{(q)}(x-z)\tilde{H}(z,y)(\nu-\lambda)(dz),
			\end{align}
			where $h_1(x,y)=h(x-y)+W^{(q)}(x-b_1)H(b_1,y)(\nu-\lambda)\{b_1\}$. 
	
	By rewriting \eqref{volterra_atom_2} in the form of \eqref{volterra_no_atom_y}, we obtain 
\begin{align}\label{volterra_atom_4}
\tilde{H}(x,y)= K(x)\left[\tilde{h}_1(x,y)+\int_{(y,x)\backslash\{b_1\}}W^{(q)}(x-z)\tilde{H}(z,y)(\nu-\lambda)(dz)\right], 
\end{align}
where $\tilde{h}_1(x,y):=h_1(x, y)+W^{(q)}(x-y)K(y)h(0)(\nu-\lambda)\{y\}1_{\{x>y\}}$.
			
			Then, proceeding as in Step (ii), by taking
			\begin{align*}
				\mathcal{K}(f)(x)=\mathcal{K}^{(1)}(f)(x)&=K(x)\int_{(y,x)\backslash\{b_1\}}e^{-s_0(x-z)}W^{(q)}(x-z)f(z){(\nu-\lambda)}(dz), \\
				&\mathcal{K}^{(n)}(f)(x)=\mathcal{K}(\mathcal{K}^{(n-1)}(f))(x),
			\end{align*}
			and
			\begin{align*}
				H_0(x,y)=e^{-s_0x}\tilde{h}_1(x,y)K(x),\qquad H_{n+1}(x,y)=H_0(x,y)+\mathcal{K}(H_n(\cdot,y))(x).
			\end{align*}
			Following the same argument as in (ii), the limit $\tilde{H}(x, y):=e^{s_0x}\lim_{n\to\infty}H_n(x, y)$ exists for $x \in [y, b_2]$ and is the solution to \eqref{volterra_atom_4} and \eqref{volterra_atom_2}.

			By the uniqueness of the solution to \eqref{volterra_prop} and since $h_1(x, y)=h(x- y)$ for $x \in[y, b_1)$, 
			 we have that $\tilde{H}(x,y)=H(x,y)$ for $x\in[y,b_1)$. 
			Note that, by \eqref{volterra_atom_2},  
			\begin{align}
			\tilde{H}(b_1,y)=& h(b_1-y)+W^{(q)}(0)H(b_1,y)(\nu-\lambda)\{b_1\}+\int_{[y,b_1)}W^{(q)}({b_1}-z)\tilde{H}(z,y)(\nu-\lambda)(dz)\\
			=&h(b_1-y)+\int_{[y,b_1]}W^{(q)}({b_1}-z)H(z,y)(\nu-\lambda)(dz)=H(b_1, y). 
			\end{align}
			{Hence, by \eqref{volterra_atom_2}
			\begin{align*}
			\tilde{H}(x,y)=& h(x-y)+W^{(q)}(x-b_1)\tilde{H}(b_1,y)(\nu-\lambda)\{b_1\}+\int_{[y,x]\backslash\{b_1\}}W^{(q)}(x-z)\tilde{H}(z,y)(\nu-\lambda)(dz)\\
			=&h(x-y)+\int_{[y,x]}W^{(q)}(x-z)\tilde{H}(z,y)(\nu-\lambda)(dz).
			\end{align*}
			Therefore, $\tilde{H}(\cdot, y)$ is the solution of \eqref{volterra_prop}. }
			
			 Hence, by setting $H(x,y)=\tilde{H}(x,y)$ for all $x\in[y,b_2]$, 
			 $H(\cdot,y)$ is a solution to \eqref{volterra_prop} on {$[y,b_2]$}. 
			
			On the other hand, note that if $h(u)\geq0$ for all $u\in[y,T]$ and $\nu-\lambda$ is a non-negative measure, then, by the previous step, $H(b_1,y)\geq0$. This implies that $\tilde{h}(x,y)\geq0$ for all $x\in[y,b_2]$. Therefore, as in Step {(i)}, the construction of the solution ensures that  $H(x,y)\geq0$ for all $x\in[y,b_2]$.
			
			(iv) Proceeding in the same manner for all the atoms $\{b_1,\dots,b_j\}$ and up to $T$, we can construct the solution to \eqref{volterra_prop} in the interval $[y,T]$. Moreover, if $h(u)\geq0$ for all $u\in[y,T]$ and $\nu-\lambda$ is a non-negative measure, a similar argument to that in Step (ii) shows that $H(x,y)\geq0$ for all $x\in[y,T]$.

			\end{document}